\documentclass[12pt,reqno]{amsart}
\usepackage{amsmath}
\usepackage{graphicx}
\usepackage{amssymb,amsthm}
\usepackage{hyperref}
\usepackage{esint}
\usepackage{epstopdf}
\usepackage{color}
\usepackage{url}
\usepackage{mathrsfs}
\usepackage{enumerate}
\usepackage[top=1in, left=1in, right=1in, bottom=1in]{geometry}

\newtheorem{theorem}{Theorem}[section]
\newtheorem{lemma}[theorem]{Lemma}
\newtheorem{proposition}[theorem]{Proposition}
\newtheorem{corollary}[theorem]{Corollary}

\newtheorem{remark}[theorem]{Remark}
\newtheorem{definition}{Definition}[section]

\def\Z{{\mathbb Z}}
\def\R{{\mathbb R}}

\def\N{{\mathbb N}}

\def\cA{{\mathcal A}}

\def\cH{{\mathcal H}}

\def\cL{{\mathcal L}}

\def\cQ{{\mathcal Q}}

\def\sQ{{\mathscr Q}}

\def\a{\alpha}
\def\b{\beta}
\def\e{\varepsilon}
\def\d{\delta}

\def\G{\Gamma}
\def\k{\kappa}
\def\l{\lambda}
\def\L{\Lambda}
\def\m{\mu}
\def\n{\nabla}
\def\p{\partial}
\def\r{\rho}
\def\s{\sigma}
\def\t{\tau}
\def\w{\omega}
\def\W{\Omega}
\def\g{\gamma}

\def\1{\left(}
\def\2{\right)}
\def\3{\left\{}
\def\4{\right\}}
\def\8{\infty}
\def\sm{\setminus}
\def\ss{\subseteq}
\def\cc{\subset\subset}

\def\uu{\overline{u}}
\def\ud{\underline{u}}
\def\vu{\overline{v}}
\def\vd{\underline{v}}

\newcommand{\mres}{\mathbin{\vrule height 1.6ex depth 0pt width
0.13ex\vrule height 0.13ex depth 0pt width 1.3ex}}

\DeclareMathOperator*{\dvg}{div}

\DeclareMathOperator*{\supp}{supp}
\DeclareMathOperator*{\osc}{osc}

\begin{document}

\title{A Free Boundary Problem Related to Thermal Insulation: Flat Implies Smooth}

\author{Dennis Kriventsov}
\address{Courant Institute of Mathematical Sciences, New York University, New York}
\email{dennisk@cims.nyu.edu}  
\date{September 4, 2017}

\begin{abstract}
 We study the regularity of the interface for a new free boundary problem introduced in \cite{CK}. We show that for minimizers of the functional
\[
 F_1(A,u) = \int_A |\n u|^2 d\cL^n + \int_{\p A} u^2 + \bar{C} \cL^n(A)
\]
over all pairs $(A,u)$ of open sets $A$ containing a fixed set $\W$ and functions $u\in H^1(A)$ which equal $1$ on $\W$, the boundary $\p A$ locally coincides with the union of the graphs of two $C^{1,\a}$ functions near most points. Specifically, this happens at all points where the interface is trapped between two planes which are sufficiently close together. The proof combines ideas introduced by Ambrosio, Fusco, and Pallara for the Mumford-Shah functional with new arguments specific to the problem considered.
\end{abstract}

\maketitle

\tableofcontents

\section{Introduction}

In \cite{CK}, the following problem is considered: let $\W$ be a bounded domain, $A$ a smooth domain containing $\W$, and $u$ a function in $H^1(A)$ with $u=1$ on $\p \W$. If we define
\begin{equation}\label{eq:F_1}
 F_1 (A,u) = \int_{A} |\n u|^2 d\cL^n + \int_{\p A} u^2 d\cH^{n-1} + \bar{C}\cL^n(A),
\end{equation}
is there a minimizing pair $(A,u)$ which attains the infimum in $F_1$ among all such pairs $(A,u)$? This problem is motivated by a question in thermal insulation, analogous to the much-studied Alt-Caffarelli problem \cite{AC} after taking a suitable scaling limit.

Minimizing pairs of $F_1$ will have the function $u$ be harmonic in the interior of $A\sm \W$, and $u\in (0,1]$. Furthermore, $u$ will satisfy the Robin condition $u_\nu = -u$ along $\p A$, where $\nu$ denotes the outward unit normal to $A$. There will also be a ``free boundary condition'' satisfied by $u$ on $\p A$, of the form
\[
 |\n u|^2 + u^2 (H_{\p A} -2) = -\bar{C},
\]
where $H_{\p A}$ is the mean curvature of the boundary of $A$, as oriented by the outward normal $\nu$. For the derivation of this condition, see the proof of Theorem \ref{thm:higherreg}.

The authors of \cite{CK} show that such minimizing pairs exist, at least if the smoothness assumption on $A$ is relaxed. The essential insight there is that an \emph{a priori} bound on $u$, which guarantees that $u\geq \d=\d(\W)>0$ on $A$, allows one to control the size of $\p A$, facilitating the application of the direct method of the calculus of variations. However, some care must be taken with regards to the boundary integral, as minimizing pairs $(A,u)$ of the functional may have the property that two separate (local) connected components of $A$ share boundary. To apply the direct method, the appropriate way of relaxing the notion of competitor is by identifying the pair $(A,u)$ with the function $u1_A$, which is in $BV$ (and, indeed, in the subspace $SBV$, see \cite[Proposition 4.4]{AFP}) provided $A$ is smooth and $u$ is bounded, and defining
\begin{equation}\label{eq:bvF_1}
 F_1(u) = \int_{\{u>0\}} |\n u|^2 +\bar{C} d\cL^n + \int_{J_u} \uu^2 +\ud^2 d\cH^{n-1}.
\end{equation}
Here $J_u$ is the jump set of $u$ and $\uu,\ud$ are the approximate upper and lower limits; these notions are defined in the next section.  In \cite{CK}, it is established that this functional admits a minimizer, and that any minimizer satisfies a density estimate which ensures that if one sets $A=\{u>0\}\sm \bar{J}_u$, then $\p A$ and $J_u$ agree up to a set of zero $\cH^{n-1}$ measure.

The purpose of this paper is to address the following regularity question: at generic points of $\p A$ (or $\bar{J}_u$), is the free boundary $\p A$ locally smooth? Due to the fact that two local connected components of $A$ may touch at some of their boundary points and then separate away from each other (and do so in ways over which we have little control, as there does not seem to be any obvious nondegeneracy property forcing them to separate at some particular rate), the best conclusion we may hope for is that $\p A$ is locally given by a pair of smooth graphs, each parametrizing the boundary of one of the components. The following theorem, to the proof of which the entire paper is devoted, says that, indeed, this happens at any point where $\p A$ is trapped between two hyperplanes which are very close together. We use the notation $Q_r= \{(x',x_n)\in \R^{n-1}\times \R : |x'|\leq r, |x_n|\leq r\}$ for a cylinder.

\begin{theorem}\label{thm:intro}Let $u\in SBV(Q_r)$ be such that $u\in \{0\}\cup [\d,1]$ for $\cL^n$-a.e. point, and assume that for any $v\in SBV(Q_r)$ with $\supp(u-v)\cc Q_r$, we have that
\[
 F_1(u)\leq F_1 (v).
\]
Assume also that $0\in \bar{J}_u$. Then there is a number $\e=\e(n,\d,\bar{C})$ such that if $r<\e$ and
\[
 \bar{J}_u\cap Q_r \ss \{|x_n|\leq \e r\},
\]
then $\bar{J}_u\cap Q_{r/2}$ coincides with the union of the graphs of a pair of $C^{1,\a}$ functions $g_\pm : \R^{n-1}\rightarrow \R$, with $\|g_\pm\|_{C^{1,\a}}\leq 1$ and $g_-\leq g_+$. If one of $u(0,\pm \frac{r}{2})=0$, then $g_+=g_-$ and is in fact $C^\8$.
\end{theorem}

We will actually prove a more general version of this theorem applicable to quasiminimizers, which are introduced in the following section. We note that the hypotheses of the theorem are satisfied on a sufficiently small cylinder around every point in $J_u$, and hence at $\cH^{n-1}$-a.e. point of the free boundary. To what extent this may be improved is explored in \cite{CK}, where some geometric criteria are given for the flatness condition. It is also likely (although we do not pursue this point) that a slight reduction in the Hausdorff dimension of $\bar{J}_u\sm J_u$ is achievable by a generic argument involving uniform rectifiability; see the book of Guy David \cite[Sections 74-75]{D}. However, whether or not $\bar{J}_u \sm J_u$ is empty remains an interesting open question even in two dimensions. 

Similar results are known for a wide variety of minimization problems, although with the conclusion being that the boundary in question is locally a graph of just one function. To mention only a few motivating examples, for sets locally minimizing perimeter in $\R^n$, this is a celebrated result of De Giorgi \cite{DGms}. Federer and Fleming obtained such results in the more general framework of integer-rectifiable currents \cite{Federer}. The proof given here is more closely related to that of Allard in his work on stationary varifolds \cite{A}. For the one-phase free boundary problem which motivated this problem, such results were obtained in \cite{AC}.

The heaviest debt, however, is owed to the analogous theorem obtained by Ambrosio and Pallara in \cite{APpt1} and with Fusco in \cite{AFPpt2} for minimizers of the Mumford-Shah functional. The proof given in this paper follows very closely the exposition of that proof as recorded and simplified in the book \cite{AFP}. Notice, however, that the functional considered here introduces a variety of difficulties not present in the Mumford-Shah case, having to do with the $u$ dependence of the surface term, but also with the presence of distinct phases where $u$ is $0$ and where it is strictly positive. This is reflected in the difference in the statements of Theorem \ref{thm:intro} and the result of \cite{APpt1,AFPpt2}. Indeed, we do not require their hypothesis that the renormalized energy
\begin{equation}\label{eq:intro}
 \frac{1}{r^{n-1}}\int_{Q_r}|\n u|^2 d\cL^n 
\end{equation}
is smaller than $\e$, but we also only obtain that $\bar{J}_u$ is the union of two graphs. This removal of a hypothesis is achieved at the very end of the proof by means of a blow-up argument and is very specific to our problem; it is well known that such an improvement is impossible in the Mumford-Shah setting and has to do with ``crack-tip'' singularities. We also mention that there are other proofs of analogous results in the Mumford-Shah setting in two dimensions, due to Bonnet \cite{Bon} and David \cite{David}. 

The structure of the proof is as follows. Section \ref{sec:prelim} explains the notation and summarizes some relevant facts from geometric measure theory. Section \ref{sec:laplace} is concerned with regularity of $u$ away from $\bar{J}_u$. Section \ref{sec:EL1} derives the Euler-Lagrange equation and the first of two excess-flatness bounds central to the argument. A number of lemmas having to do with the density of the surface energy term on cylinders are proved in Section \ref{sec:gmt}. Then Section \ref{sec:lip} contains a Lipschitz parametrization theorem, Section \ref{sec:efbound2} derives the second, more delicate excess-flatness bound, and Section \ref{sec:graphs} combines the results of Sections \ref{sec:lip} and \ref{sec:efbound2}.

The key steps in the proof occur in Sections \ref{sec:energy} and \ref{sec:flatness}. Essentially, there are two governing quantities in the argument, and the goal is to show that they both decay geometrically in $r$. The first is the renormalized energy \eqref{eq:intro}, while the second measures how close $J_u$ lies to a pair of hyperplanes, which we call flatness (this is defined in Section 2). Section \ref{sec:energy} contains a compactness argument due to Ambrosio, Fusco, and Pallara that shows that if the energy is much larger than the flatness, then it will decay geometrically on a smaller cylinder. Section \ref{sec:flatness} addresses the other possibility (when the flatness dominates), in which case the flatness will decay. Section \ref{sec:EL2} is devoted to constructing auxiliary competitors used in Section \ref{sec:flatness}. Finally, Section \ref{sec:iteration} combines these ideas into one geometric decay lemma, and this is iterated in Section \ref{sec:mainpf} to obtain Theorem \ref{thm:intro}, under some extra assumptions. Section \ref{sec:blowup}  presents a blow-up argument which removes those assumptions, and finally Section \ref{sec:highreg} explains why the last sentence of Theorem \ref{thm:intro}, concerning higher regularity, is valid.

\section{Preliminaries}\label{sec:prelim}

We will use the notation $\cH^{s}$ for $s$-dimensional Hausdorff measure and $\cL^n$ for Lebesgue measure on $\R^n$. We fix a coordinate system for the Euclidean space $\R^n$, with basis vectors $e_1,\ldots, e_n$. Points $x\in \R^n$ will often be decomposed as $x=(x',x_n)$, where $x_n = e_n \cdot x \in \R$ and $x'\in \R^{n-1}$. We will use the letter $\pi$ to refer to the hyperplane $\{x_n = 0\}$, and other affine planes may be given names such as $\pi', \pi_*$, and so on. We will also use $\pi$ to denote the orthogonal projection of $\R^n$ onto $\R$; the typical use will be to write $\pi(E)$ for the set $\{x':(x',x_n)\in E \text{ for some } x_n \}.$

Given a smooth vector field $T: \R^n \rightarrow \R^n$, the divergence of $T$ will be expressed as $\dvg T = Tr(\n T) = \sum_i e_i \cdot \n T e_i$. Given a unit vector $\nu$, we write $\dvg\nolimits^\nu T = \dvg T - \nu \cdot \n T \nu$. Given a smooth hypersurface $S$ with unit normal vector $\nu_x$ at each point $x\in S$ (the orientation does not matter), the \emph{tangential divergence} is $\dvg\nolimits^{\nu_x} T$. If $T$ is only Lipschitz, then $\dvg T$ exists at $\cL^n$-a.e. point; if $S$ is countably $\cH^{n-1}$ rectifiable, then $\dvg\nolimits^{\nu_x}T$ exists at $\cH^{n-1}$-a.e. point of $S$ (see \cite[Section 2.11]{AFP} for details). 

The sets $B_r(x)$ are Euclidean balls in $\R^n$ centered at $x$ with radius $r$, while $D_r(x)=B_r(x)\cap \{x_n=0\}$ will be referred to as disks. The numbers $\w_{n}$ are the volumes of the unit balls: $\w_{n}=\cL^n (B_1)$. The set $Q_r = D_r \times (-r,r)$ represents the open cylinder of size $r$. More generally, the cylinder $Q_{r,\nu}(x)$ is the cylinder centered at $x$ with faces orthogonal to the unit vector $\nu$.

Given a set $E$, the function $1_E$ is $1$ on $E$ and $0$ otherwise. All integrals are Lebesgue integrals with respect to the given measure; we also use
\[
\fint_E f d\cL^n = \frac{1}{\cL^n(E)} \int_E f d\cL^n.
\]
 Given a measure $\mu$ and a measurable set $E$, the restriction of $\mu$ to $E$ is denoted by $\mu \mres E$:
\[
\mu \mres E (A) = \mu(E\cap A).
\]
Given a $\mu$-integrable function $f$, the measure $f d\mu$ is given by
\[
f d\mu (A) = \int_A f d\mu.
\]

We use standard notation for function spaces: for an open set $E$, $L^p(E)$ denotes the usual Lebesgue space, $H^1(E)$ the Sobolev space with norm
\[
\int_E u^2 + |\n u|^2 d\cL^n,
\]
and $H^1_0(E)$ the closure of $C^\8_c(E)$ (smooth functions with compact support) in $H^1(E)$. The spaces $C^\a(E) = C^{\lceil\a \rceil,\a -\lceil\a \rceil}(E)$ are H\"older spaces. The subscript $X_{\text{loc}}(E)$ is used to refer to spaces
\[
X_{\text{loc}}(E) = \bigcap_{K\cc E} X(K).
\] 

We use the terms supremum, infimum, support, and oscillation to refer to the essential supremum, essential infimum, essential support, and essential oscillation, respectively. They are abbreviated $\sup,\inf,\supp,\osc$; here $\osc_E u = \sup_E u - \inf_E u$.

\subsection{Functions of Bounded Variation}

The space $BV(\mathbb{R}^{n})$ contains functions $u$ in $L^1 (\R^n)$ with distributional gradients representable by finite Borel regular measures (which we write as $Du$), and is equipped with the usual norm:
\[
 \|u\|_{BV(\R^n)}=\|u\|_{L^1(\R^n)}+|Du|(\R^n).
\]
For a function of bounded variation, we define the \emph{measure-theoretic upper and lower limits} by
\[
 \uu(x)=\inf\3 t : \limsup_{r\searrow 0}\fint_{B_{r}(x)}(u-t)_{+}d\cL^n=0\4
\]
and
\[
 \ud(x)=\sup\3 t : \limsup_{r\searrow 0}\fint_{B_{r}(x)}(t-u)_{+}d\cL^n=0\4.\]
 Points for which $\uu(x)=\ud(x)$ are referred to as \emph{points of approximate continuity}, while $S_u=\{x: \uu(x)>\ud(x)\}$ is the singular set. Notice that points of approximate continuity are precisely those $x$ for which the blow-ups $v_{x,r}(y)=u(x+ry)$ converge to constant functions in $L^1_{\text{loc}}$ as $r\searrow 0$. Usually we will prefer to work with $K_u=\bar{S}_u$, the topological closure.
 
 The set $J_u\ss S_u$ of \emph{jump points} for a function of bounded variation consists of those points for which the blow-ups $v_{x,r}(y)=u(x+ry)$ converge to a function constant on each side of a hyperplane; i.e. there exists a unit vector $\nu_x$ such that
 \[
  v_{x,r}\xrightarrow{L^1_{\text{loc}}} 1_{\{y: \nu_x\cdot y < 0\}}\ud(x)+1_{ \{y: \nu_x\cdot y > 0\}}\uu(x)
 \]
as $r\searrow 0$. The vector $\nu_x$, oriented as above, will be referred to as the \emph{approximate jump vector}. For the properties of the sets $J_u$ and $S_u$ we refer to \cite[Sections 3.6-3.7]{AFP}; the most consequential fact about them for us is the Federer-Vol'pert theorem:

\begin{proposition}\label{prop:FedVol} Let $u\in BV(\R^n)$. Then $S_u$ is countably $\cH^{n-1}$-rectifiable, $\cH^{n-1}(S_u \sm J_u)=0$, and
\[
 Du \mres S_u = (\uu-\ud)\nu_x d\cH^{n-1}\mres J_u.
\]
\end{proposition}

Using this theorem and the Lebesgue decomposition, we arrive at a representation $Du=\n u d\cL^n+ (\uu-\ud)\nu_x d\cH^{n-1}\mres J_u + D^{c}u$, where $\n u$ is the density of the part of $Du$ which is absolutely continuous with respect to Lebesgue measure. The third term represents the remaining singular part of $Du$, and $u$ is said to be a \emph{special function of bounded variation} if that term vanishes. We set $SBV(\R^n)$ to be the linear subspace of $BV$ containing such functions.

The proposition below is a compactness and closure property of $SBV$, which is a special case of \cite[Theorem 4.7, 4.8]{AFP} (except the very last statement, which instead may be deduced from Theorem 5.22; the reason that the integrand is jointly convex is explained in Example 5.23b, while the extra assumption on the integrand being bounded from below is only used to show that $\cH^{n-1}(J_{u_i})$ is uniformly bounded, which we assume separately).

\begin{proposition}\label{prop:SBVcpt}
 Let $\{u_i\}_{i\in \N}$ be a sequence of functions in $SBV(B_R)\cap L^\8 (B_R)$. Then the following hold:
 \begin{itemize}
  \item[(i)] If
  \[ 
   \sup_i \1 \int_{B_R}|\n u_i|^2 d\cL^n+\|u_i\|_{L^\8 (B_R)}+\cH^{n-1}(J_{u_i})\2<\infty,
  \]
   then there is a subsequence $u_{i_k}\rightarrow u$ in $L^1(B_R)$, and $u\in SBV$.
  \item[(ii)] Furthermore, along the subsequence above we have
  \[
   \int_{B_R}|\n u|^2 d\cL^n\leq \liminf_k \int_{B_R}|\n u_{i_k}|^2 d\cL^n,
  \]
  \[
   \cH^{n-1}(J_{u})\leq \liminf_k \cH^{n-1}(J_{u_{i_k}}),
  \]
  and
  \[
   \int_{J_u}\uu^2+\ud^2 d\cH^{n-1}\leq \liminf_k \int_{J_{u_{i_k}}}\uu^2_{i_k}+\ud^2_{i_k} d\cH^{n-1}.
  \]
 \end{itemize}
\end{proposition}

\subsection{Sets of Finite Perimeter}

A Borel set $E\ss \R^n$ is said to have finite perimeter if $1_E$ lies in $BV$. Given such a set $E$, we define the Borel measure
\[
 P(E;A)=|D1_E|(A).
\]

For an arbitrary set $E$, let $E^{(0)}$ be the set of points with Lebesgue density $0$ for $E$, and $E^{(1)}$ the set of points of Lebesgue density $1$. The \emph{essential boundary} is defined by $\p^e E = \R^n \sm (E^{(0)} \cup E^{(1)})$, and if $E$ has finite perimeter, this agrees with $S_{1_E}$. The \emph{reduced boundary} $\partial^{*}E$ of a set of finite perimeter is the set of points whose blow-ups converge to half-spaces, or $J_{1_E}$. From Proposition \ref{prop:FedVol}, we have that $\cH^{n-1}(\p^e E\sm \p^* E)=0$.

\subsection{Minimizers and Related Quantities}

We define, for $u\in SBV$ and $\W$ a Borel set, the functional
\begin{equation}\label{eq:F}
 F(u;\W) = \int_{\W} |\n u|^2 d\cL^n + \int_{J_u\cap \W} \uu^2+\ud^2 d\cH^{n-1}. 
\end{equation}
Motivated by the minimizers constructed in \cite{CK}, we introduce the following notion of quasiminimizers:

\begin{definition} A nonnegative function $u\in SBV$ is called an \emph{$\L,r_0,s$-quasiminimizer} of $F$ on $\W$ if for every ball $B_r(x)\ss \W$ with $r<r_0$ and every $v\in SBV$ with $u-v$ supported on $B_r(x)$, we have
\begin{equation}\label{eq:min}
 F(u;\W)\leq F(v;\W) + \L r^{n-1+s}.
\end{equation}
Such a quasiminimizer is called \emph{$\d$-regular} if, in addition, it satisfies 
\begin{equation}\label{eq:dregsize}
u\in \{0\}\cup [\d,\d^{-1}] \qquad \text{for } \cL^n\text{-a.e. } x\in \W 
\end{equation}
and also
\begin{equation}\label{eq:dregden}
 \d r^{n-1} \leq \cH^{n-1}(J_u \cap B_r(x)) \leq \d^{-1} r^{n-1}
\end{equation}
for every $B_r(x)\ss \W$ with $x\in K_u$ and $r<r_0$. The set $\sQ(\W,\L,r_0,s,\d)$ contains the $\d$-regular $\L,r_0,s$-quasiminimizers on $\W$. As $s$ and $\delta$ will often be fixed, we may use the compact notation $\sQ(\W,\L,r_0)$ for the same. Constants which depend only on $n,s,$ and $\d$ are called \emph{universal}.
\end{definition}

Following the argument in \cite{CK}, it may be verified that if a quasiminimizer satisfies \eqref{eq:dregsize}, then it automatically satisfies the lower bound in \eqref{eq:dregden} (the argument there was written for a specific $\L,r_0,1$ quasiminimizer, but extends easily). The upper bound in \eqref{eq:dregden}, with a constant depending only on $\L r_0^s$ and universal quantities, holds for any such quasiminimizer satisfying \eqref{eq:dregsize} as well: on any ball $B_r(x)\ss\W$ with $r<r_0$ we have that
\begin{equation}\label{eq:remball}
 \int_{B_r}|\n u|^2 d\cL^n + \d^2 \cH^{n-1}(J_u \cap B_r) \leq \d^{-2} n\w_{n-1}r^{n-1} + \L r^{n-1+s}.
\end{equation}
This comes from using $u1_{\R^n \sm B_r(x)}$ as a competitor. Notice that we did not even need to assume $x\in K_u$.

On the other hand, at least the lower bound on $u$ in \eqref{eq:dregsize} is not a local property of quasiminimizers, and will come from the global structure of the problem. As it is also an essential nondegeneracy property, we will always assume it. For a $\d$-regular quasiminimizer, $\cH^{n-1}(K_u \sm J_u)=0$ (see \cite[Section 2.9]{AFP}), so we will prefer to work with the closed set $K_u$.


We will use the following notions of flatness, excess, and energy; the first two are analogous to the quantities Allard uses in his theory of varifolds.

\begin{definition} Given $u\in \sQ(\W,\L,r_0)$, a unit vector $\nu$, and a cylinder $Q_{r,\nu}(x)\ss \W$, the \emph{flatness} of $u$ is given by
\[
 f(u,x,r,\nu) = \frac{1}{r^{n+1}}\int_{Q_{r,\nu}(x)\cap K_u} |(y-x)\cdot \nu|^2 d\cH^{n-1}(y). 
\]
The \emph{excess} is given by
\[
 e(u,x,r,\nu) = \frac{1}{r^{n-1}}\int_{Q_{r,\nu}(x)\cap J_u} 1- (\nu\cdot \nu_y)^2 d\cH^{n-1}(y), 
\]
where $\nu_y$ denotes the approximate jump vector at $y$. The \emph{energy} is
\[
 E(u,x,r,\nu)=\frac{1}{r^{n-1}}\int_{Q_{r,\nu}(x)}|\n u|^2 d\cL^n.
\]
If $\nu$ is suppressed, it is assumed that $\nu = e_n$. If $x$ is suppressed, it is assumed that $x=0$.
\end{definition}
From \eqref{eq:remball}, we have that all three of these quantities are uniformly bounded in terms of universal constants and $\L r^s$. The general structure of our argument is to show that, if they start small enough, then they (eventually) decay geometrically in $r$.

An alternative choice is an $L^{\infty}$ version of the flatness, which is more convenient in places:
\[
 f_{\8}(u,x,r,\nu) = \frac{1}{r}\sup_{y\in K_{u}\cap Q_{r,\nu}(x)} |(y-x)\cdot \nu|.
\]
Let us explain briefly how this is related to $f$.
\begin{proposition}\label{prop:flatsame} Let $u\in \sQ(Q_{2r},\L,4r)$. Then
\[
 \frac{1}{C}f_{\8}^{n+1}(u,r)\leq f(u,2r) \leq C f_{\8}^2 (u,2r),
\]
where the constant $C$ is universal.
\end{proposition}
\begin{proof}
 The estimate
\[
 f(u,2r) \leq C f_{\8}^2 (u,2r)
\]
follows from the definitions and the density upper bound in \eqref{eq:dregden}. For the other direction, choose $x\in \bar{Q}_r\cap K_u$ assuming the supremum in the definition of $f_\8(u,r)$, and notice that the ball $B_{\frac{r}{2} f_\8(u,r)}(x)\ss Q_{2r} \cap \{|x_n|\geq r f_\8 (x,r)/2\}$. This gives
\begin{align*}
 f(u,2r)&\geq \frac{1}{r^{n+1}}\int_{B_{\frac{r}{2} f_\8(u,r)}(x)\cap K_u}|y_n|^2 d\cH^{n-1} \\
 &\geq \frac{r^2 f_\8^2(u,r)}{4r^{n+1}}\cH^{n-1}(B_{\frac{r}{2} f_\8(u,r)}(x)\cap K_u) \\
 &\geq \frac{\d}{2^{n+1}}f_\8^{n+1}(u,r). 
\end{align*}
\end{proof}

Another variant of the flatness which we will use measures the distance to two different hyperplanes. We say that a point $y\in K_u$ is \emph{visible from above} on  $Q_{r,\nu}(x)$ if the line segment $Q_{r,\nu}(x) \cap \{y + t \nu:t>0\} $ does not intersect $K_u$. Similarly, a point is \emph{visible from below} if $Q_{r,\nu}(x) \cap \{y + t \nu:t<0\}$ avoids $K_u$. Given unit vectors $\nu$, $\nu_+$, define $\nu_-$ to be the reflection of $\nu_+$ about $\nu$: $\nu_- = 2(\nu_+\cdot \nu)\nu - \nu_+$. Then set
\begin{equation}\label{eq:ffdef}
ff(u,x,r,\nu) =\inf_{\nu_+\in S^{n-1},h\in [0,\frac{r}{2}], |\nu-\nu_+|\leq \frac{h}{100 r}} \3\frac{1}{r^{n+1}}\int_{Q_{r,\nu}(x)\cap K_u} d_*^2(y;\nu_+,h) d\cH^{n-1}(y)\4,
\end{equation}
where
\[
 d_*(y;\nu_+,h) = \begin{cases}
           |(y- x - h \nu)\cdot \nu_+|  & y \text{ is visible only from above}\\ &\qquad \text{ and } \ud(y)=0\\
           |(y - x + h \nu)\cdot \nu_-| & y \text{ is visible only from below}\\ & \qquad \text{ and } \ud(y)=0 \\
	   |(y- x - h \nu)\cdot \nu_+| + |(y- x + h \nu)\cdot \nu_-| & \text{otherwise.}
           \end{cases}
\]
To understand this definition (for simplicity taking $x=0$ and $\nu=e_n$), consider the two affine planes $\pi_+ = \{y:(y-h e_n)\cdot \nu_+ = 0\}$ and $\pi_- = \{y:(y + h e_n)\cdot \nu_= = 0\}$. These are reflections of one another across $\pi$, and from the conditions on $h$ and $\nu_+$ in the infimum in \eqref{eq:ffdef} they do not intersect $\pi$ or one another. The quantity $d_*(y;\nu_+,h)$ measures the distance between $y$ and one (or both) of these planes, depending on whether $y$ is visible from above or below. This is averaged over $K_u$, and the infimum is then taken to find the pair of such affine planes which minimizes the average. See Figure \ref{fig:1} for an illustration.  In some stages of our proof, we expect that $K_u$ is very close to a pair of such planes with comparatively large $h$: in this case, $f$ will be large, but $ff$ will remain small.

\begin{figure}
	\centering
	\def\svgwidth{15cm}
	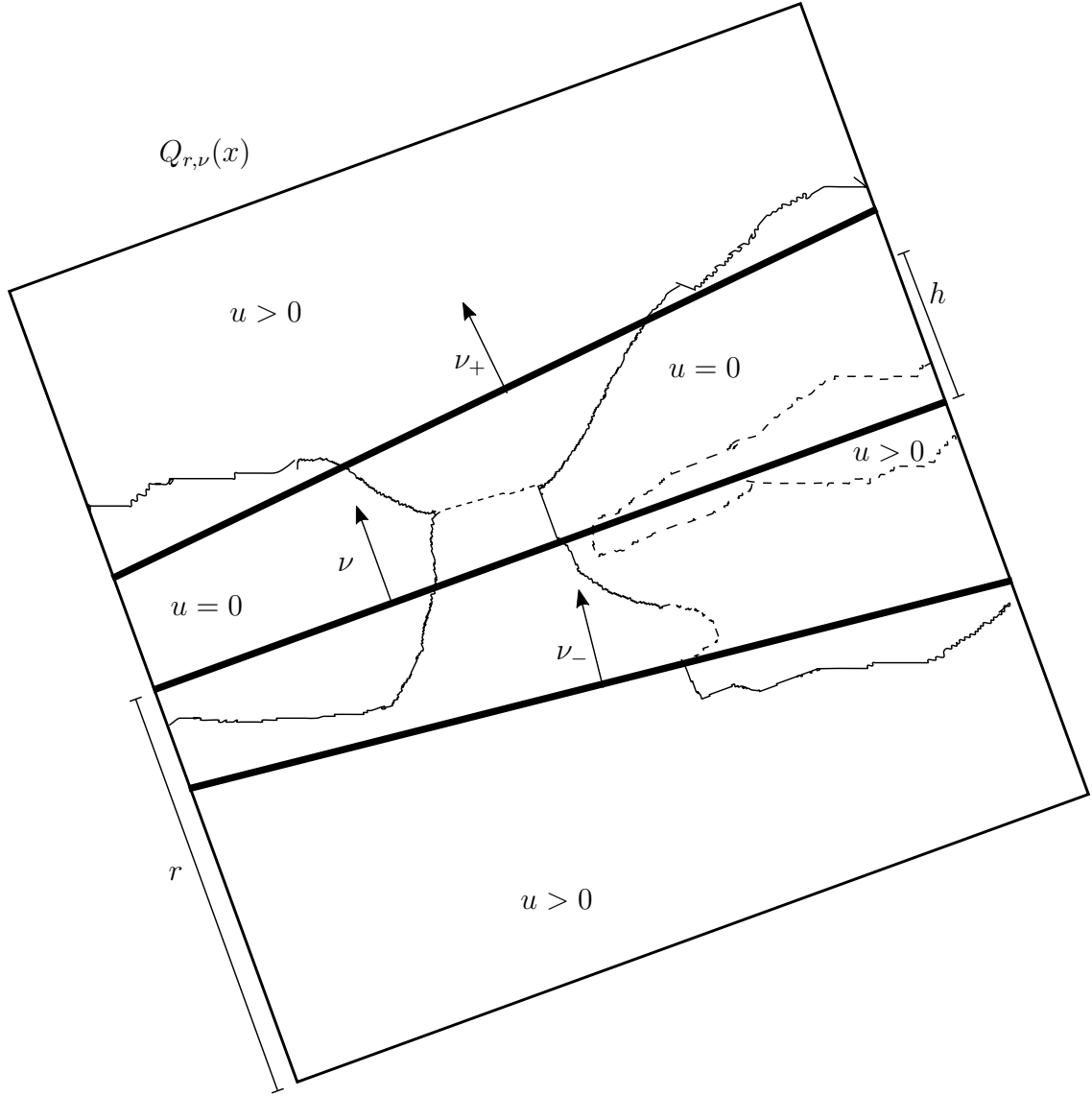
	\caption{The three bold lines are (from top to bottom) the planes $\pi_+,\pi,\pi_-$. The set $K_u$ of a function $u\in \cQ(Q_{r,\nu}(x),\L,2r)$ is represented by the thin irregular curves. The solid curves are those portions of $K_u$ which are visible only from above or below, and have $\uu=0$, while the dashed curves are the remaining portions. Note the different way in which $d_*(\cdot,\nu_+,h)$ is defined on the solid and dashed portions: on the solid portions, $d_*$ is the distance to whichever of $\pi_+,\pi_-$ the point is visible from, while on the dashed portions, it is the sum of the distances to the two planes.}\label{fig:1}
\end{figure}

When working with $ff$, we will frequently want to extract the normal vectors to the minimal planes $\nu_\pm$ and constant $h$ from the definition. Notice that the optimal $h/r \leq C f_\8$, and the infimum is always uniquely attained. We will use the notation $\nu_+(u,x,r,\nu)$ and $h(u,x,r,\nu)$ for these optimal choices, although when unambiguous we will suppress some, and indeed sometimes all, of the parameters. It will sometimes be convenient to express these planes as a graph over $\pi$ (when $\nu=e_n$, $x=0$), for which we use the notation $L_\pm [u,x,r,\nu]:\R^{n-1}\rightarrow \R$ (which we typically abbreviate this as $L_\pm$). In other words, for each point $y'\in D_r$, $(y',L_+(y'))$ is the unique point in $\{z:(z-h e_n)\cdot \nu_+=0\}$ with $z'=y'$. Note that 
\begin{equation}\label{eq:ffvf}
ff(u,r)\leq 4f(u,r)\leq C(n) [\frac{h^2}{r^2} + ff(u,r)].
\end{equation}

We also have a version of Proposition \ref{prop:flatsame} for $ff$, which states that $K_u$ is contained in a $C ff^{\frac{1}{n+1}}r$=neighborhood of the union of auxiliary planes $\pi_+,\pi_-$ with normal vectors $\nu_\pm$ and passing through $(0,\pm h)$.
\begin{proposition}\label{prop:fflatsame} Let $u\in \sQ(Q_{2r},\L,4r)$. Then
	\begin{align}
	K_u\cap Q_r &\ss \{ x\in Q_r: | (x - h e_n)\cdot \nu_+ | < C ff^{\frac{1}{n+1}}(u,2r)r \} \nonumber\\
	&\quad \quad \cup \{ x\in Q_r: | (x + h e_n)\cdot \nu_- | < C ff^{\frac{1}{n+1}}(u,2r)r \} \label{eq:fflatsamec}
	\end{align}
	where $\nu_+=\nu_+(u,x,r,\nu)$ and the constant $C$ is universal.
\end{proposition}
\begin{proof}
	Choose an $x\in Q_r\cap K_u$ which is not in the set to the right of \eqref{eq:fflatsamec}. Then
	\[
	 B:=B_{\frac{C r}{2} ff^{\frac{1}{n+1}}(u,2r)}(x)\ss Q_{2r} \cap \{d_*(y;\nu_+,h)\geq \frac{C r}{2} ff^{\frac{1}{n+1}}(u,2r)\}.
	\] 
	This gives
	\begin{align*}
	ff(u,2r)&\geq \frac{1}{r^{n+1}}\int_{B\cap K_u}d_*^2 (y;\nu_+,h) d\cH^{n-1} \\
	&\geq \frac{C^2 r^2 ff^{\frac{2}{n+1}}(u,2r)}{4r^{n+1}}\cH^{n-1}(B \cap K_u) \\
	&\geq C^{n+1}\frac{\d}{2^{n+1}} ff(u,2r), 
	\end{align*}
	with the last step from \eqref{eq:dregden}. For a large enough $C$, this is a contradiction. 
\end{proof}

The following is the version of the excess modeled after the two-plane flatness, where once again given $\nu_+,\nu$, we define $\nu_-$ as the reflection of $\nu_+$ across $\nu$:
\[
 ee(u,x,r,\nu,\nu_+) = \frac{1}{r^{n-1}} \int_{Q_{r,\nu}(x) \cap J_u}\min\{1-(\nu_- \cdot \nu_y)^2,1-(\nu_+ \cdot \nu_y)^2\} d\cH^{n-1}.
\]
If the last parameter is suppressed, it will be presumed that the $\nu_+$ being used is $\nu_+(u,x,4r,e_n)$, the one which is minimal for the flatness $ff$ on the cylinder four times the size.

\section{Quasiminimizers of Laplace's Equation}\label{sec:laplace}

In this section we consider the regularity of $u\in \sQ(\W,\L,r_0)$ away from the singular set. Set $\sQ_0(\W,\L,r_0,s)$ to be functions in $H^1(\W)$ satisfying
\[
 \int_{B_r(x)}|\n u|^2 d\cL^n \leq \int_{B_r(x)} |\n w|^2 d\cL^n + \L r^{n-1+s}
\]
for every $w\in H^1 (\W)$ with $u-w\in H^1_0 (B_r)$ and $B_r(x)\ss \W$ with $r\leq r_0$. If $u\in \sQ(\W,\L,r_0)$, then $u\in \sQ_0(\W\sm K_u,\L.r_0)$.

The following lemma establishes the optimal regularity of quasiminimizers of this sort. For our applications, we will always take $\r$ some fixed multiple of $r$. The proof is completely standard, and we include it for the reader's convenience. Note that the conclusion is not valid when $s=1$.

\begin{lemma}\label{lem:intreg} Let $s\in [0,1)$ and $u\in \sQ_0(Q_r,\L,2r,s)$. Then we have that $u\in C^{\frac{1+s}{2}}(Q_\r)$ for each $\r<r$, and the following three estimates
 \begin{equation}\label{eq:intreg1}
 [u]_{C^{\frac{1+s}{2}}(Q_\r)}^2 \leq C \1 \L  + \frac{r^{n-1}}{(r-\r)^{n-1+s}} E(u,r) \2
\end{equation}
 \begin{equation}\label{eq:intreg2}
 [u]_{C^{\frac{1}{2}}(Q_\r)}^2 \leq C \1 \L r^s  + \frac{r^{n-1}}{(r-\r)^{n-1}} E(u,r) \2
\end{equation}
\begin{equation}\label{eq:intreg3}
 \1\osc_{Q_{\r}} u\2^2 \leq C \r \1 \L r^s  + \frac{r^{n-1}}{(r-\r)^{n-1}} E(u,r) \2.
\end{equation}
The constants depend only on $n$ and $s$.
\end{lemma}

\begin{proof}
 Note that \eqref{eq:intreg2} follows by applying \eqref{eq:intreg1} with $s=0$, for 
 \[
 \sQ_0(Q_r,\L,2r,s)\ss \sQ_0(Q_r,\L (2r)^s,2r,0). 
 \]
 Also, \eqref{eq:intreg3} follows directly from \eqref{eq:intreg2}. We thus only show \eqref{eq:intreg1}.
 
 We will make use of Campanato's criterion \cite{Camp}, for which we must estimate the mean deviations of $u$ on every set $B_\s(x)\cap Q_\r$ with $x\in Q_\r$ and $\s \leq \r$:
 \begin{equation}\label{eq:intregi}
  [u]_{C^{\frac{1+s}{2}}(Q_\r)}^2\leq C \sup_{x\in Q_\r, \s>0} \3 \frac{1}{\s^{n+1+s}}\int_{B_\s(x)\cap Q_\r} | u- m(u;x,\s)|^2 d\cL^n   \4
 \end{equation}
holds for a universal constant $C=C(n)$. This is valid with $m(u;x,\s)$ any numbers, but we will use
\begin{equation}\label{eq:intregi2}
m(u;x,\s) = \fint_{B_\s(x)\cap Q_\r} u d\cL^n.
\end{equation}
 Note that, up to increasing the constant in \eqref{eq:intregi}, it suffices to only consider $\s_j= 2^{-j}(r-\r)$ for $j\in \Z$. The rest of the proof involves showing 
 \begin{equation}\label{eq:intregi3}
 \frac{1}{\s_j^{n+1+s}}\int_{B_{\s_j}(x)\cap Q_\r} | u- m(u;x,\s_j)|^2 d\cL^n \leq  C \1 \L  + \frac{r^{n-1}}{(r-\r)^{n-1+s}} E(u,r) \2
 \end{equation}
for every $x\in Q_\r$ and $j\in \Z$, which would then imply \eqref{eq:intreg1}.

First we deal with the balls $B_{\s_j}(x)$ which have $j\leq 0$. We have, from the Poincar\'{e} inequality (which has a uniform constant for domains of the type $B_{\s_j}(x)\cap Q_\r)$), that
\[
 \int_{B_{\s_j} (x)\cap Q_\r}|u - \fint_{B_{\s_j}(x)\cap Q_\r}u d\cL^n|^2 d\cL^n \leq C\s^{2}_j\int_{B_{\s_j}(x) Q_\r} |\n u|^2 d\cL^n=C E(x,r)\s^2_j r^{n-1}.
\]
This gives
\begin{equation}\label{eq:intregi4}
 \int_{B_{\s_j} (x)\cap Q_\r}|u - m(u;x,\s_j)|^2 d\cL^n \leq C \s^{n+1+s}_j \frac{r^{n-1}}{(r-\r)^{n-1+s}} E(u,r)
\end{equation}
with $m(u;x,\s_j)$ as in \eqref{eq:intregi2}. This implies \eqref{eq:intregi3} for $j\leq 0$.

The remaining balls $B_{\s_j}(x)$ with $j\geq 0$ are fully contained in $Q_r$. Set
\[
A_j(x)= \frac{1}{\s_j^{n-1+s}} \int_{B_{\s_j}(x)} |\n u|^2 d\cL^n .
\]
From the previous estimate \eqref{eq:intregi4}, we have that 
\begin{equation}\label{eq:intregi5}
A_0(x)\leq C  \frac{r^{n-1}}{(r-\r)^{n-1+s}} E(u,r).
\end{equation}
We will now obtain a recurrence between the quantities $A_j$.

First, fix a ball $B:=B_{\s_{j-1}}(x)$ with $j\geq 1$, and take any function $v\in H^1_0 (B)$. We may use $u+tv$ as a competitor for $u$, which gives
\[
 \int_{B}|\n u|^2 d\cL^n \leq \int_{B} |\n u|^2 + 2t \n u \cdot \n v + t^2 |\n v|^2 d\cL^n + \L \s_{j-1}^{n-1+s}.
\]
Applying to $\pm v$, we see that for each $t> 0$ we have the inequality
\begin{equation}\label{eq:hel}
 \left| \int_B \n u \cdot \n v d\cL^n \right| \leq \frac{t}{2} \int_{B}|\n v|^2 d\cL^n + \frac{\L \s_{j-1}^{n-1+s}}{2t}.
\end{equation}

Now let $h$ be the solution to the Dirichlet problem
\[
 \begin{cases}
  \triangle h = 0 & \text{ on } B\\
  h=u & \text{ on } \p B.
 \end{cases}
\]
Then $h\in H^1(B)$, and indeed
\[
 \int_B |\n h|^2 d\cL^n \leq \int_B |\n u|^2 d\cL^n.
\]
Set $v=u-h\in H^1_0 (B)$, observing that as $h$ is harmonic, we have
\[
 \int_B \n h \cdot \n v d\cL^n =0.
\]
Inserting $v$ into \eqref{eq:hel} with $t=1$, we obtain
\[
 \int_B |\n v|^2 d\cL^n \leq \L \s_{j-1}^{n-1+s}.
\]

Now, for any harmonic function, $\frac{1}{r^n}\int_{B_r} |\n h|^2 d\cL^n$ is a nondecreasing quantity in $r$, and hence
\[
 \int_{B_{\s_j}(x)}|\n h|^2 d\cL^n \leq \frac{1}{2^n} \int_{B}|\n h|^2 d\cL^n.
\]
We therefore obtain
\begin{align*}
 \|\n u\|_{L^2(B_{\s_j}(x))} &\leq \|\n h\|_{L^2(B_{\s_j}(x))}+\|\n v\|_{L^2(B_{\s_j}(x))}\\
    &\leq \frac{1}{2^{n/2}}\| \n h\|_{L^2(B)} + \sqrt{\L} \s_{j-1}^{\frac{n-1+s}{2}}\\
    &\leq \frac{1}{2^{n/2}}\| \n u\|_{L^2(B)} + \sqrt{\L} \s_{j-1}^{\frac{n-1+s}{2}}
\end{align*}
We divide both sides by $\s_j^{\frac{n-1+s}{2}}$, to get that
\[
 \sqrt{A_j(x)} \leq \frac{1}{2^{\frac{1-s}{2}}} \sqrt{A_{j-1}(x)} +\sqrt{\L}\ 2^{\frac{n-1+s}{2}}.
\]
It is now simple to check by induction that this implies that
\[
 A_j(x) \leq A_{0}(x) + C\L.
\]

Applying the Poincar\'{e} inequality as before, we obtain that
\[
 \frac{1}{\s_j^{n+1+s}} \int_{B_{\s_j}(x)\cap Q_\r} |u - m(u;x,\s_j)|^2 d\cL^n \leq C A_{0}(x) +C\L,
\]
with $m(u;x,\s_j)$ from \eqref{eq:intregi2}. Together with our bound on $A_0$ from \eqref{eq:intregi5}, this gives \eqref{eq:intregi3} for $j\geq 0$ and so completes the proof.
\end{proof}

In particular, any $F$ quasiminimizer $u$ is continuous away from $K_u$. We will also make use of the following lemma about the convergence of quasiminimizers:

\begin{lemma}\label{lem:harmacon}Let $s\in [0,1)$, $u_k\in \sQ_0(Q_r,\L_k,2r,s)$ with $\L_k\rightarrow 0$, and $u_k \rightharpoonup u$ weakly in $H^1(Q_r)$. Then $u$ is harmonic and $u_k\rightarrow u$ strongly in $H^1_{\text{loc}}(Q_r)$.
\end{lemma}

\begin{proof}
 Fix any $\r<r$, and take any $\phi\in H^1(Q_\r)$ with $\phi=u$ on $\p Q_\r$. Let $\eta$ be a smooth cutoff function with $\eta=1$ on $Q_\r$ and $\eta=0$ outside $Q_{\g\r}\cc Q_\r$. We use $v=(\phi -u_k)\eta^2 $ as the function in \eqref{eq:hel}; this gives
 \[
 | \int_{Q_r}\eta^2 \n u_k \cdot \n [\phi - u_k] + 2 \eta (\phi-u_k) \n u_k \cdot \n \eta d\cL^n |\leq \frac{t}{2} \int_{Q_r}|\n v|^2 d\cL^n + \frac{\L_k r^{n-1+s}}{2t}
 \]
Set $t=\L_k^{\frac{1}{2}}$ and take $k\rightarrow \8$. We obtain that
\[
 \lim \int_{Q_r}|\n u_k|^2 \eta^2 d\cL^n = \int_{Q_r} \eta^2 \n u \cdot \n \phi-2\eta (\phi-u) \n u \cdot \n \eta d\cL^n.
\]
Notice that the second term on the right vanishes, and we have
\[
\limsup \int_{Q_\r}|\n u_k|^2  d\cL^n \leq \lim \int_{Q_r}|\n u_k|^2 \eta^2 d\cL^n = \int_{Q_r} \eta^2 \n u \cdot \n \phi d\cL^n .
\]
 We may send $\eta \rightarrow 1_{Q_\r}$ to obtain
\begin{equation}\label{eq:harmacon}
 \limsup \int_{Q_\r}|\n u_k|^2 d\cL^n \leq \int_{Q_\r} \n u \cdot \n \phi d\cL^n. 
\end{equation}
Setting $\phi=u$ gives that
\[
 \limsup \int_{Q_\r}|\n u_k|^2 d\cL^n \leq  \int_{Q_\r}|\n u|^2 d\cL^n\leq \liminf \int_{Q_\r}|\n u_k|^2 d\cL^n ,
\]
(the second inequality from the weak convergence), and this implies strong convergence. Using $\pm \phi$ in \eqref{eq:harmacon} for any other $\phi$, we get
\[
\int_{Q_\r}|\n u|^2  d\cL^n = \int_{Q_\r} \n u \cdot \n \phi d\cL^n,
\]
which implies that $u$ is harmonic.
\end{proof}

\section{The First Variation and the First Excess-Flatness Bound}\label{sec:EL1}

In this section we derive the first variation formula for a quasiminimizer $u$, and use this to control the excess in terms of the flatness and energy. As we will see in future sections, the excess should be thought of as a higher-order quantity, and will be useful in controlling how much $K_u$ looks like the union of two Lipschitz graphs. 

\begin{lemma}\label{lem:EL} Let $u\in \sQ(Q_{2r},\L,4r)$. Then for every Lipschitz vector field $T\in C_c^{0,1}(Q_r;\R^n)$ with $\|T\|_{C^{0,1}}\leq 1$, and any $0<t<c_0$ with $c_0$ universal, we have
 \begin{align*}
  &\left| \int_{J_u}(\uu^2(x)+\ud^2(x))\dvg\nolimits^{\nu_x}T d\cH^{n-1}(x) + \int|\n u|^2 \dvg T - \n u \cdot \n T \n u d\cL^n\right| \\
  &\qquad\leq C(n) t  r^{n-1}\1E(u,r) + 1\2 + \frac{\L}{t}r^{n-1+s}.
 \end{align*}
 Here $\nu_x$ is the approximate jump vector. If $K\cap Q_{2r}$ is given by a pair of $C^{1}$ graphs $\G^+,\G^-$ of functions $g_+\geq g_-$ over $\pi$, with $u=0$ on $\{g_-<g_+\}$, we may also conclude that for $T$ as before with $T=\phi(x) e_n$ with $\phi$ a nonnegative function, we have
 \begin{align*}
  \int_{\G^+}&(\uu^2(x)+\ud^2(x))\dvg\nolimits^{\nu_x}T d\cH^{n-1}(x) + \int_{x_n> g_+(x')}|\n u|^2 \dvg T - \n u \cdot \n T \n u d\cL^n \\
  &\qquad\geq -C(n) t  r^{n-1}\1E(u,r) + 1\2 - \frac{\L}{t}r^{n-1+s}.
 \end{align*}
 \end{lemma}

 \begin{proof} We only sketch the proof. Consider the mapping $\phi_t(x)= x+tT(x)$ for small $t$. For $t<c_0$ small enough, the Lipschitz condition on $T$ ensures that this is a bilipschitz mapping $Q_{2r}\rightarrow Q_{2r}$, with $\phi_t(x)-x$ compactly supported on $Q_{2r}$. The function $u_t = u\circ \phi_t^{-1}$ is in $SBV$, and so we may use it as a competitor. We now compute the energy and surface terms.
 
 We know from the area formula for rectifiable sets (see \cite[Theorem 2.91]{AFP}) that $\phi_t$ is tangentially differentiable at $\cH^{n-1}$ -a.e. point of $J_u$, and that for a bounded Borel function $g$,
 \[
  \int_{\phi_t(J_u)} g\circ \phi^{-1}_t d\cH^{n-1} = \int_{J_u} g \text{\textbf{J}}_{n-1}(d^{\nu_x}\phi_t) d\cH^{n-1}, 
 \]
where $\text{\textbf{J}}_{n-1}(d^{\nu_x}\phi_t)$ is the $n-1$-dimensional Jacobian of the tangential derivative $d^{\nu_x}\phi_t$. At such a point of tangential differentiability, this Jacobian admits the approximation
\[
 |\text{\textbf{J}}_{n-1}(d^{\nu_x}\phi_t) - 1 - t\dvg\nolimits^{\nu_x}T|\leq Ct^2,
\]
with the constant controlled by the Lipschitz bound on $T$ (this is just a fact about matrices, and the computation may be found in \cite[Theorem 7.31]{AFP}). Setting $g=\uu^2+\ud^2$ gives
\[
 |\int_{J_u}(\uu^2+\ud^2) (1 + t\dvg\nolimits^{\nu_x} T) d\cH^{n-1} - \int_{J_{u_t}}\uu^2_t+\ud^2_t d\cH^{n-1}|\leq Cr^{n-1}t^2.
\]

On the other hand, we have that
\[
 \int |\n u_t|^2 d\cL^n = \int | \n\phi_t^{-1}\n u|^2 \circ \phi_t |\det\n \phi_t|d\cL^n
\]
from the area formula. Computing at differentiability points of $\phi_t$,
\[
 |\det (\n \phi_t) - I - t \dvg\nolimits T|\leq Ct^2
\]
and
\[
 |(\n \phi_t)^{-1} - 1 + t \n T|\leq Ct^2,
\]
so
\[
\left| \int |\n u_t|^2 - |\n u|^2 - t \dvg T |\n u|^2 + 2t \n u \cdot \n T\n u d\cL^n\right|\leq CE(u,2r)t^2.
\]

Combining these relations and using the quasiminimality of $u$ gives the conclusion. To obtain the second conclusion, use instead the competitors
\[
u_t(x)=\begin{cases}
 u\circ \phi_t^{-1}(x) & \{x_n \geq g_+(x')\}\\
 u(x) & \{x_n\leq g_-(x')\}\\
 0 & \{g_-(x')<x_n<g_+(x')\}.
\end{cases}
\]
The computation is then analogous.
 \end{proof}

\begin{theorem}\label{thm:exbyflat}Let $u\in \sQ(Q_{2r},\L,4r)$. Then
\[
 e(u,r/2)\leq C\1f(u,r)+E(u,r) + \sqrt{\L}r^{s/2}\2, 
\]
where the constant $C$ is universal.
\end{theorem}

\begin{proof}
 Define the vector field $T_0(x) = x_n e_n$; we have that $\n T_0 \cdot e = e\cdot e_n e_n$ for any vector $e$, so in particular $T_0$ has divergence 1. This implies that (for any $y\in J_u$ )
\[
 \dvg\nolimits^{\nu_y} T_0 =    1 - (\nu_y \cdot e_n)^2,  
\]
where $\nu_y$ is the approximate jump vector at $y$.

Now we select a smooth nonnegative cutoff function taking values in $[0,1]$, called $\phi$. Choose it so that it is supported on $Q_r$, is identically $1$ on $Q_{r/2}$, and has $|\n \phi_1|\leq C(n)/r$. Set $T = C_*\phi^2 T_0$, where $C_*$ is a universal constant chosen so that $T$ has derivative bounded by $1$.

Next, we estimate the tangential divergence of $T$. From the chain rule and the previous computation,
\begin{align*}
 \dvg\nolimits^{\nu_y} T &=  C_*\phi^2 ( 1 - (\nu_y \cdot e_n)^2) + 2C_*\phi y_n [e_n \cdot \n \phi - (e_n \cdot \nu_y)(\n \phi \cdot \nu_y)]\\
	&\geq C_*\phi^2 ( 1 - (\nu_y \cdot e_n)^2) - 2C_*\phi |y_n||\n \phi| |e_n  - (e_n \cdot \nu_y) \nu_y|\\
    &\geq C_*\phi^2 \frac{1}{2}( 1 - (\nu_y \cdot e_n)^2) -  2 C_*|\n \phi|^2 y_n^2. 
\end{align*}
The last step used Cauchy's inequality and the identity
\[
|e_n  - (e_n \cdot \nu_y) \nu_y|^2 = 1 - (e_n \cdot \nu_y)^2.
\] 
We may apply Lemma \ref{lem:EL} to $T$ with $\s = \sqrt{\L r^s}$ to obtain that
\[
 \left|\int_{Q_r \cap J_u} \dvg\nolimits^{\nu_y} T (\uu^2 +\ud^2)d\cH^{n-1}\right| \leq C[ E(u,r) + \sqrt{\L r^s}]r^{n-1};
\]
we moved all of the energy terms to the right. This gives
\[
 \int_{Q_r \cap J_u} \phi^2 ( 1 - (\nu_y \cdot e_n)^2) d\cH^{n-1} \leq C\1\int_{Q_r \cap J_u} \frac{x_n^2}{r^2} d\cH^{n-1} +[ E(u,r) + \sqrt{\L r^s}]r^{n-1}\2.
\]
We used the upper and lower bounds on $u$ from \eqref{eq:dregsize} to remove $\uu^2 + \ud^2$ from the integrands. The term on the left controls $e(u,r/2)$, while the first one on the right is bounded by $f(u,r)$. This gives
\[
 e(u,r/2)\leq C( E(u,r) + f(u,r) + \sqrt{\L r^s}).
\]
\end{proof}

We will later prove a version of this theorem, Theorem \ref{thm:exbyflat2}, which has $ff$ and $ee$ in place of $f$ and $e$.

\section{Measure-Theoretic Arguments}\label{sec:gmt}

We now use direct competitors and compactness arguments to prove several lemmas aimed at improving our density estimates \eqref{eq:dregden} in a flat regime. This requires more effort than is typically needed in, say, the theory of  stationary varifolds, largely because there is no monotonicity formula to give a ``perfect'' bound on the densities from below.

Some definitions will be helpful. Given $u\in \sQ(\W,\L,r_0)$, let $\mu_u  = (\uu^2 + \ud^2) d\cH^{n-1}\mres K_u$ be a Borel measure on $\W$ (we will drop the $u$ subscript where unambiguous). We would like to decompose $\mu_u$ into two pieces, one of which counts only the values of $u^2$ along the $\{x_n<0\}$-facing ``side'' of $K_u$, and the other along the $\{x_n>0\}$ ``side.'' Of course, $K_u$ does not have sides in any useful topological sense, so what we really want has to do with looking at the approximate tangent planes to $J_u$. It will be helpful to replace $u$ with its \emph{precise representative}, which is is defined to be $\frac{\uu+\ud}{2}$ outside of $\R^n \sm (S_u\sm J_u)$, and arbitrarily otherwise. Let $\nu_x$ be the approximate jump vector at $x\in J_u$. We make use of the following theorem of Vol'pert about one-dimensional restrictions of $BV$ functions, which may be found in \cite[Theorem 3.108]{AFP}.
\begin{proposition}\label{prop:1Drest} Let $u\in SBV(\R^n)$. Define $u_{x'}(t)=u(x',t):\R\rightarrow \R$; then $u_{x'}\in SBV(\R)$ for $\cL^{n-1}$-a.e. $x'\in \pi$, the measure-valued maps $x'\mapsto D^j u_{x'}, x'\mapsto \n u_{x'}d\cL^1$ are Borel measurable,
\[
 e_n \cdot \n u d\cL^n = \cL^{n-1}\otimes \n u_{x'}d\cL^1,
\]
and
\[
e_n \cdot D^j u = \cL^{n-1}\otimes D^j u_{x'}d\cL^1.
\]
Moreover, for $\cL^{n-1}$-a.e. $x'\in \pi$, we have that $J_{u_{x'}}=J_u \cap (\{x'\}\times \R)$, that $e_n \cdot \nu_{(x',t)}\neq 0$ at each $t\in J_{u_{x'}}$, and lastly that
\begin{equation}\label{eq:onesidelim}
 \begin{cases}
  \lim_{s\searrow t} u(x',s) = \uu(x',t) \qquad \lim_{s\nearrow t} u(x',s) = \ud(x',t) & e_n \cdot \nu_{(x',t)} >0 \\
  \lim_{s\searrow t} u(x',s) = \ud(x',t) \qquad \lim_{s\nearrow t} u(x',s) = \uu(x',t) & e_n \cdot \nu_{(x',t)} <0.
 \end{cases}
\end{equation}
\end{proposition}

We now perform the following construction: let $A\ss \pi \cap Q_r$ be the set of full $\cL^{n-1}$ measure on which \eqref{eq:onesidelim} holds. Then set
\[
 \begin{cases}
  A^{+}_{u,r}=\{(x',t)\in K\cap Q_r: x'\in A, (x',s)\notin K\cap B_r \forall s>t\}\\
  A^{-}_{u,r}=\{(x',t)\in K\cap Q_r: x'\in A, (x',s)\notin K\cap B_r \forall s<t\}
 \end{cases}
\]
to be the sets of all nice points in $K$ visible from above and below. The obvious but useful fact about these is that the line segment $\{x\pm e_n t,t>0\}\cap B_r$ does not touch $K$ for each $x\in A^{\pm}_{u,r}$. Also define
\[
 \m_{u,r}^- = \lim_{s\nearrow t}u^2 (x',s)d\cH^{n-1}\mres (J_u \cap (A\times \R))
\]
and
\[
 \m_{u,r}^+ = \lim_{s\searrow t}u^2 (x',s)d\cH^{n-1}\mres (J_u \cap (A\times \R)).
\]
These are Borel measures. The $r$ dependence will be suppressed from now on (it may happen that on a larger cylinder $Q_s$, the associated set $A$ gets smaller by an $\cL^{n-1}$ negligible set, but from Proposition \ref{prop:1Drest} it may be checked that this will not change the measures, in the sense that $\m^+_{u,r}=\m^+_{u,s}\mres Q_r$. Likewise, under the assumption that $f_\8(u,s)<r/s$, which will always hold for us, we have that  $\cH^{n-1}(A^\pm_{u,s}\sm A^\pm_{u,r})=0$.) Note that there is an implicit dependence in these definitions on the choice of vector $e_n$. 

From Proposition \ref{prop:1Drest} it is easy to deduce that $\m^+ + \m^- \leq \m$. Notice that it would be unreasonable to expect equality here, as there could be substantial parts of $K$ whose approximate tangent planes are orthogonal to $\pi$. We will show soon, however, that in the flat situations we are concerned with, they are individually bounded from below, and in fact so are $\m^\pm \mres A^{\pm}$. It is worth noting here that if, say, $f_\8(u,r)\leq \frac{1}{4}$ and $u(x,r/2)>0$, we have that $u\geq \d$ (the lower bound) along every line segment $x+te_n$ with $x\in A^+$ and $t>0$. This may be used to show that in this situation,
\begin{equation}\label{eq:mpposden}
 \m^+\mres (Q_r\cap A^+) \geq \d^2 \cH^{n-1}\mres (Q_r\cap A^+),
\end{equation}
which will be used frequently without explicit mention.

The next few lemmas are proved directly.

\begin{lemma}\label{lem:upperdensity} Let $u\in \sQ(Q_{2r},\L,4r)$. Then for each $\g>1$ there is a constant $\e_1(\g)$ such that if
\[
 M:= f(u,2r) + \L r^s + r \leq \e_1,
\]
then
\begin{equation}\label{eq:upperdensity}
 \m(Q_r)\leq \g (u^2(0,r/2)+u^2(0,-r/2)) \w_{n-1}r^{n-1}.
\end{equation}
\end{lemma}

To make sense of the conclusion, recall from Proposition \ref{prop:flatsame} that, up to assuming $\e_1$ is small enough, $f_\8(u, 3r/2)<<1/2$, and so $u$ is continuous in a neighborhood of the points $(0,\pm r/2)$. The reader may find it odd that we demand that $r$ is small, but this is because these estimates rely on $u$ being nearly constant on the two halves of $Q_r$, which is only the case if $r$ is small.

\begin{proof}
Fix $q$ to be small, and assume for the moment $f_\8(u,3r/2)\leq q/4$. Define two cylindrical regions $S^-=Q_r \cap \{x_n \leq - \frac{qr}{2}\}$ and $S^+=Q_r \cap \{x_n \geq  \frac{q r}{2}\}$. Applying Lemma \ref{lem:intreg} to cylinders covering the $q/4$ neighborhood of $S^\pm$, we obtain that
\[
 \osc_{S^+}u + \osc_{S^-}u \leq C\sqrt{r} (\L r^s + q^{-n-1}E(u,r))^{\frac{1}{2}}\leq C(q)\sqrt{r}.
\]
We used here that $E(u,r)$ is bounded from \eqref{eq:remball}.

Consider as a competitor the function $w = u 1_{\R^n \sm (D_r\times \{|x_n|\leq q r\})}$. We see from the minimization inequality that
\[
 \int_{D_r\times \{|x_n|\leq q r\}}|\n u|^2 d\cL^n + \int_{D_r\times \{|x_n|\leq q r\} \cap J_u}d\m \leq \int_{D_r \times \{\pm q\}}u^2 d\cH^{n-1} + \frac{2q}{\d^2}\w_{n-1}r^{n-1} +\L r^{n-1+s}.
\]
The second term on the right is the contribution from the lateral sides $\p D_r \times \{|x_n|\leq q r\}$ of the cylinder. Dropping the energy term on the left and using the assumption on $f_\8$, this gives
\[
 \frac{\m(Q_r)}{r^{n-1}}\leq \frac{1}{r^{n-1}}\int_{D_r \times \{\pm q\}}u^2 d\cH^{n-1}+ C(q+M).
\]
From the oscillation bound on $u$, we see that, furthermore,
\[
 \frac{\m(Q_r)}{r^{n-1}}\leq \w_{n-1}(u^2(0,r/2)+u^2(0,-r/2)) + C_0(q+M+C(q)\sqrt{r}).
\]
Now we choose $q = \frac{(\g-1) \d^2\w_{n-1}}{3C_0}$, and then $\e_1$ small enough to ensure that $f_\8(u,3 r/2)\leq \frac{q}{4}$ (by Proposition \ref{prop:flatsame}), $M\leq q$, and $r\leq \frac{q^2}{C^2(q)}$. This implies that
\[
 \frac{\m(Q_r)}{r^{n-1}}\leq \w_{n-1}(u^2(0,r/2)+u^2(0,-r/2)) + 3C_0 q \leq \g \w_{n-1}(u^2(0,r/2)+u^2(0,-r/2)).
\]
Note that we used here that at least one of $u(0,\pm r/2)$ is nonzero. If they were both zero, we have instead shown that
\[
 \frac{\m(Q_r)}{r^{n-1}}\leq (\g-1)\d^2.
\]
Selecting $\g$ close enough to $1$ and using the density lower bound from \ref{eq:dregden}, this implies that $K_u \cap Q_{7r/8}$ is empty. Applying this to cylinders $Q_{3r/2}(x)$ with $x\in D_{r/8}$ (on which we can ensure $f_\8(u,x,3r/2)<q/4$; observe that if $u(0,\pm r/2)=0$, then $u(x',\pm r/2)=0$ as well, as $u$ is either $0$ or at least $\d$ on each connected component of $Q_r\sm K_u$), we learn that $K_u \cap Q_{r}$ is empty, so \eqref{eq:upperdensity} holds trivially.
\end{proof}

\begin{lemma}\label{lem:lowerdensity} Let $u\in \sQ(Q_{2r},\L,4r)$. Assume $K_u\cap Q_{r/2} \neq \emptyset$. Then there is a constant $\e_2$ such that if
\[
 M:= E(u,2r)+ \L r^s + r \leq \e_2, \qquad f(u,2r)\leq \e_2,
\]
then there is a universal constant $c_0$ such that 
\begin{equation}\label{eq:ldc1}
|u(0,r/2)-u(0,-r/2)|^2 \geq c_0 r,
\end{equation}
and also
\begin{equation}\label{eq:ldc2}
 \frac{\cL^{n-1}(D_r \sm \pi(A^\pm\cap Q_r))}{r^{n-1}}\leq C M.
\end{equation}
Moreover, we have that
\begin{equation}\label{eq:ldc3}
 \m^+(Q_r\cap A^+)\geq (1-CM) u^2(0,r/2)\w_{n-1}r^{n-1}
\end{equation}
and
\begin{equation}\label{eq:ldc4}
 \m^-(Q_r\cap A^-)\geq (1-CM) u^2(0,-r/2)\w_{n-1}r^{n-1}.
\end{equation}
\end{lemma}

This lemma gives a very good estimate on the size of the ``holes''  in $K$, or rather the complement of its projection onto $\pi$. It also gives a lower bound on the density. This lower bound may be improved in certain cases, which will be an important idea in later sections.

\begin{proof}
We begin in the same way as in the proof of the previous lemma, by assuming that $f_\8(u, 3r/2)\leq t/4$, and showing that
\begin{align}
 \osc_{Q_{5r/4}\cap \{x_n> tr\}}u + \osc_{Q_{5r/4}\cap \{x_n<-tr\}}u & \leq Cr^{1/2}(\L r^s + t^{-n-1}E(u,2r))^{\frac{1}{2}} \nonumber \\
 &\leq C(t)r^{1/2}M^{1/2}\label{eq:ldi1} \\
 &\leq C(t)M. \label{eq:ldi2}
\end{align}

Using the bound on $E$, and taking $\s \in (q,1/8)$,
\[
 \frac{1}{r^{n-1}}\int_{\s r/2}^{\s r}\int_{D_r\times \{|x_n|=a\}}|\n u|^2 d\cL^n(x',a) \leq  E(u,2r)\leq M ,
\]
so applying Chebyshev's inequality lets us find a $t\in (\s,\s/2)$ such that
\begin{equation}\label{eq:ldi3}
 \frac{1}{r^{n-1}}\int_{D_r\times \{|x_n|=t r\}}|\n u|^2 d\cL^{n-1}\leq \frac{CM}{\s r} \leq \frac{CM}{tr}.
\end{equation}

We will now estimate quantitatively how close $u(0,-r/2)$ and $u(0,r/2)$ may be. Let $t$ be as above and consider as a competitor the following linear interpolation:
\[
 w(x)=\begin{cases}
       u(x) & x\notin D_r\times (-tr,tr)\\
       \frac{1}{2}\1u(x',-tr)(1-\frac{x_n}{t r}) + u(x',tr)(1+\frac{x_n}{t r})\2 & x\in D_r \times (-tr,tr).
      \end{cases}
\]
We know from the computation above that
\begin{align*}
 \int_{D_r \times (-tr,tr)}|\n w|^2 d\cL^n 
 & = 4tr \int_{D_r \times\{|x_n|=tr\} }|\n u|^2 d\cL^n + 2rt\int_{D_r} \frac{|u(x',tr)-u(x',-tr)|^2}{4t^2r^2}d\cL^{n-1} \\
 &\leq C M r^{n-1} + \int_{D_r} \frac{|u(x',tr)-u(x',-tr)|^2}{2tr}d\cL^{n-1} \\
& \leq CMr^{n-1} + C(t)Mr^{n-1} + Cr^{n-1} \frac{|u(0,r/2)-u(0,-r/2)|^2}{tr}.
\end{align*}
The second line used the bound \eqref{eq:ldi3}, and the third used the bound on the oscillation \eqref{eq:ldi1}. By plugging $w$ into the minimization inequality and using the density lower bound from \eqref{eq:dregden} at the point we know is in $K\cap Q_{r/2}$, we get
\[
 c(\d)  \leq C(t)M + \L r^s + \frac{|u(0,r/2)-u(0,-r/2)|^2}{2tr} + \frac{2t}{\d^2}\w_{n-1} .
\]
The last term comes from the lateral sides of the cylinder $D_r \times (-tr,tr)$. Now take $t$ small in terms of $\d$ only so that the last term may be reabsorbed, and then $\e_2$ small so that the first two terms are also reabsorbed. This gives
\[
 \frac{|u(0,r/2)-u(0,-r/2)|^2}{r}\geq c_0,
\]
where the constant is universal. This establishes \eqref{eq:ldc1}.

Next, we estimate how large the set of ``holes'' $H:=D_{5r/4} \sm \pi(K\cap Q_{5r/4})$ is. The point here is that whenever there is a hole, there is an enormous contribution to the gradient along the line segment going through it. The reason we care about the holes will become clear soon: the complement of the holes contributes directly to the lower bound.

Let $t$ be fixed, say $t=\frac{1}{8}$. For each point $x'\in H$, notice that
\begin{equation}\label{eq:ldi4}
 |u(x',tr)-u(x',-tr)|^2 \geq c_0 r - C M r\geq \frac{c_0 r}{2}
\end{equation}
from \eqref{eq:ldc1} and \eqref{eq:ldi2}, if $\e_2$ is chosen small enough. The line segment $\{x'\}\times (-tr,tr)$ does not intersect $K_u$, so $u$ is continuous along it. From the fundamental theorem of calculus,
\begin{equation}\label{eq:ldi5}
 \frac{c_0 r }{2}\leq \1\int_{-tr}^{tr} |\p_{e_n} u(x',a)|d\cL^1(a)\2^2 \leq tr \int_{-tr}^{tr} |\n u(x',a)|^2 d\cL^1(a),
\end{equation}
which holds for $\cL^{n-1}$-a.e. $x'\in H$. Integrating in $x'$ over $H$,
\[
 \cL^{n-1}(H) \leq C \int_{H\times (-tr,tr)}|\n u|^2 d\cL^n \leq C r^{n-1} M.
\]
The conclusion \eqref{eq:ldc2} now follows by observing that $H$ differs from $D_r \sm \pi(A^\pm \cap Q_r)$ only by the $\cL^{n-1}$-negligible set of Proposition \ref{prop:1Drest}.

Finally, we are in a position to estimate the densities of $\m^{\pm}$. First, let $A^0$ be $\pi(A^+\cap Q_r)=\pi(A^-\cap Q_r)$, and take a point $x'\in A^0$. Define $t^\pm(x')$ to be the unique numbers such that $(x',t^\pm(x'))\in A^{\pm}$, noting that $|t^\pm|\leq t r/4$. If we now let $t=\frac{1}{8}$ as before, we have
\[
 |u (x',t r) - \lim_{s\searrow t^{+}(x')} u(x',s)| \leq C\sqrt{2tr} \1\int_{t^{+}(x')}^{tr} |\p_{e_n} u(x',a)|^2 d\cL^1(a)\2^{1/2},
\]
so we get
\begin{equation}\label{eq:ldi6}
 |u (0,r/2) - \lim_{s\searrow t^{+}(x')}u(x',s)| \leq C \int_{-tr}^{tr} |\n u(x',a)|^2 d\cL^1(a) + C M + C r.
\end{equation}
We have that
\[
 \int_{Q_r \cap A^+}|e_n \cdot \nu_x| d\m^+ = \int_{A^0 \cap D_r}\lim_{s\searrow t^{+}(x')}u^2(x',s)d\cL^{n-1}
\]
from the area formula for rectifiable sets (see \cite[Theorem 2.91]{AFP}), where $\nu_x$ is the approximate jump vector at $x$. Noting that the Jacobian factor is less than one and applying \eqref{eq:ldi6}, we obtain that
\[
 \m^+(Q_r \cap A^+)\geq \cL^{n-1}(D_r \cap A^0)u^2(0,r/2) - CM r^{n-1}.
\]
Combining with \eqref{eq:ldc2}, we obtain the conclusion \eqref{eq:ldc3}. The final conclusion \eqref{eq:ldc4} is identical.
\end{proof}

We prove a slightly modified version of the lemma, which is less quantitative and will be useful much later.

\begin{lemma}\label{lem:lowerdensity2} Let $u\in \sQ(Q_{2r},\L,4r)$. Assume $K_u\cap Q_{r/2} \neq \emptyset$. Then for any $\g<1$, there is a constant $\e_{2'}$ such that if
\[
 M:=\frac{1}{r^{n-1}}\int_{Q_{2r}\cap \{|x_n|\geq 2\e_{2'}r\}}|\n u|^2 d\cL^n + f(u,2r) + r + \L r^s \leq \e_{2'},
\]
then 
\[
\frac{\m (Q_r)}{\w_{n-1} r^{n-1}} \geq \g (u^2 (0,r) + u^2(0,-r)).
\]
\end{lemma}

\begin{proof}
 The proof is the same as the previous lemma, just with worse estimates. First, we may still conclude that
 \[
  |u (0,r) - u(0,-r)|\geq c_0 \sqrt{r},
 \]
 proceeding in exactly the same way.
 
 For the estimate on the holes, we also use that for a.e. $x'\in H$,
 \[
   \frac{c_0 }{2}-C(t)M\leq  t \int_{-tr}^{tr} |\n u(x',a)|^2 d\cL^1(a),
 \]
from \eqref{eq:ldi4} and \eqref{eq:ldi5}, but now carefully preserve the $t$, which may be chosen as small as $f_\8(u,r)$. This gives
\[
 \cL^{n-1}(H) (1-C(t)M) \leq C t r^{n-1}.
\]
We may first choose $t$ small, and then $\e_{2'}$ even smaller so that $C(t)M\leq \frac{1}{2}$, and
\[
 \cL^{n-1}(H) \leq C t r^{n-1}\leq (1-\g)\w_{n-1}\frac{r^{n-1}}{2}.
\]

Similarly, in the final estimate \eqref{eq:ldi6}, we get
\begin{align*}
 |u (0,r/2) - \lim_{s\searrow t^{+}(x')}u(x',s)| &\leq C \sqrt{tr}  \sqrt{\int_{-tr}^{tr} |\n u(x',a)|^2 d\cL^1(a)} + C(t)M\\
 & \leq C(t)M + t\int_{-tr}^{tr} |\n u(x',a)|^2 d\cL^1(a).
\end{align*}
Integrating along the complement of $H$ and using the crude energy bound from \eqref{eq:remball} gives
\[
 |\cL^{n-1}(D_r\sm H)u^2(0,r/2) - \int_{D_r\sm H} \lim_{s\searrow t^{+}(x')}u^2(x',s)|d\cL^{n-1}\leq [C(t)M+Ct]r^{n-1}.
\]
By first choosing $t$ small and then $M$ small, the right-hand side may be controlled by $(1-\g)\w_{n-1}\frac{\d^2 r^{n-1}}{2}$. We may now conclude as before.
\end{proof}

Now we prove a lemma about what happens when the excess and the energy are small. For the Lipschitz approximation theorem of Section \ref{sec:lip}, we require some way of saying that small excess implies flatness. This is because flatness is not additive in the same way excess and energy are, and so it is poorly suited to some covering arguments. On the other hand, excess and energy do not by themselves control the flatness, as you might have several parallel planes as $K$. This example has overly large density, however, and hence is not minimal. Up to technical points involving how to tell what the correct value of $u$ is when computing the density, that is the purpose of the next set of lemmas.

Let us emphasize here that while the conclusion in the following lemma is a lower bound on the densities of $\mu^\pm$, not unlike in Lemma \ref{lem:lowerdensity}, the hypotheses are weaker: the whole point is that it features $e$ rather than $f$. In Theorem \ref{thm:exbyflat}, we showed that $f$ controls $e$, but the converse is more difficult to show. This explains the additional steps in the proof and the softer blow-up argument used here.

\begin{lemma}\label{lem:exlb}Let $u\in \sQ(Q_{2r},\L,4r)$. Assume that $0\in K$. Then for every $\g<1$ there is a constant $\e_3(\g)$ such that if
	\begin{equation}\label{eq:exlbh}
	E(u,2r) + e(u,2r) + \L r^s + r \leq \e_3,
	\end{equation}
	then the following three properties hold:
\begin{enumerate}[(A)]
	\item If
	\begin{equation}\label{eq:exlbc1}
	\frac{\m^+(Q_{ r})}{\w_{n-1}r^{n-1}}< \g \d^2,
	\end{equation}
	where $\d$ is the constant from the $\d$-regularity condition, then
	\begin{equation}\label{eq:exlbc2}
	 f(u, \g r)<1-\g \text{ and } u(0,r/2)=0.
	\end{equation}
	\item At least one of the following holds:
	\begin{equation}\label{eq:exlbc3}
	\frac{\m^+(Q_{ r})}{\w_{n-1}r^{n-1}}\geq \g \d^2\qquad \text{ or }\qquad \frac{\m^-(Q_{r})}{\w_{n-1}r^{n-1}}\geq \g \d^2.
	\end{equation}
	\item If we also assume that $0\in A^+$ and
	\[
	\sup_{t\in (0,r/2)}|a - u(0,t)|\leq \e_3
	\]
	for some $a$, then
	\begin{equation}\label{eq:exlbc4}
	\frac{\m^+(Q_{ r}\cap \{|x_n|\leq (1-\g)r\})}{\w_{n-1}r^{n-1}}\geq \g a^2.
	\end{equation}
\end{enumerate}	
 \end{lemma}
%
%

\begin{proof}
 The argument will be by contradiction. If one of (A),(B), or (C) fail to hold, then we may find sequences $u_k,r_k,\L_k$, with
\begin{equation}\label{eq:exlbi1}
   E(u_k,2r_k) + e(u_k,2r_k) + \L_k r_k^s + r_k \rightarrow 0.
\end{equation}
We will first analyze this sequence, extract a suitable limit, and examine its structure. Afterwards, we will see how failing one of the lemma's conclusions then results in a contradiction.

Letting $\tilde{u}_k(x)=u(r_k x)$, we see from the energy bound in \eqref{eq:exlbi1} that
\begin{equation}\label{eq:exlbi2}
 \frac{1}{r_k}\int_{Q_2}|\n \tilde{u}_k|^2 d\cL^n = E(u_k,2r_k) \rightarrow 0.
\end{equation}
It follows from the $SBV$ compactness theorem of Proposition \ref{prop:SBVcpt} that $\tilde{u}_k \rightarrow u_\8 \in SBV$ in $L^1(Q_2)$ along a subsequence, and $u_\8$ is locally constant on $Q_2\sm K_{u_\8}$ (having $\n u_\8=0$ from \eqref{eq:exlbi2}). We may also take a further subsequence for which $r_k^{1-n}\m_{u_k}(r_k \cdot) \rightarrow \m$ and $r_k^{1-n}\m^+_{u_k}(r_k \cdot)\rightarrow \m^+$ weakly-* as measures to some Borel measures $\mu,\mu^+$. Recall that for any $r_k x_k\in K_{u_k} = \supp \mu_{u_k} $ and every ball $B_s(x_k)\ss Q_2$, we have that the density estimate \eqref{eq:dregden} gives
\[
C^{-1}\leq \frac{\mu_{u_k}(B_{sr_k}(x_k))}{s^{n-1}r_k^{n-1}} \leq C,
\] 
where we used \eqref{eq:dregsize} to pass from $\cH^{n-1}\mres K_{u_k}$ to $\mu_{u_k}$. For almost every $s$, this passes to the limit to give that if $x = \lim x_k$, then
\begin{equation}\label{eq:exlbi5}
C^{-1}\leq \frac{\mu(B_s(x))}{s^{n-1}}\leq C. 
\end{equation}
 As a consequence, it follows that if we set $K :=\supp \mu$, then $K$ is the limit of $K_{u_k}/r_k$ in Hausdorff topology (locally in $Q_2$). Indeed, \eqref{eq:exlbi5} implies that if $x_k\in K_{u_k}/r_k$ converge to $x\in Q_2$, then $x\in K$. If, on the other hand, $x \in K$, then
\[
0< \mu(B_s(x))\leq \liminf r^{1-n}_k\mu_{u_k} (B_{s r_k}(x))
\]
for all $s$ sufficiently small, so we may find $x_k \in K_{u_k}/r_k $ converging to $x$. 
Another consequence of \eqref{eq:exlbi5} is that we may express $\mu = \theta(x)d\cH^{n-1}\mres K$, with $C^{-1}\leq \theta \leq C$ at $\cH^{n-1}$-a.e. point (from \cite[Theorem 2.56]{AFP} and the Radon-Nikodym theorem). 

We also have that $K_{u_\8} \ss K$: one way to see this is that for any open connected $V \cc Q_2 \sm K$, we have that $V \ss Q_2 \sm K_{\tilde{u}_k}$ for $k$ large enough. Applying Lemma \ref{lem:intreg}, we have that $\tilde{u}_k$ are equicontinuous on $V$, and so converge uniformly to $u_\8$. This implies that $u_\8$ is continuous, and so constant, on $V$, and so $V\cap K_{u_\8} = \emptyset$. Note that we do not claim the opposite inclusion $K\ss K_{u_\8}$.

We now make use of the excess assumption. Consider the product space $Q_2 \times S^{n-1}$, and the measures $\s_k =r_k^{1-n}(\overline{\tilde{u}_k}^2(x)+\tilde{\ud}_k^2(x))d\cH^{n-1}\mres J_{\tilde{u}_k} \otimes \d_{\nu_x}$ on it, where $\d_{\nu_x}$ is a delta at the approximate jump vector $\nu_x$, which is normal to to $J_{\tilde{u}_k}$ at $x$. These measures have bounded total variation, so we may extract a weak-* convergent subsequence $\s_k \rightarrow \s$. Then we have that
\begin{align*}
 \int_{Q_2\times S^{n-1}} 1- (e_n,\nu')^2 d\s(x,\nu') &= \lim_k \int_{Q_2\times S^{n-1}}1- (e_n,\nu')^2 d\s_k(x,\nu')\\
 & = \lim_k r_k^{1-n}\int_{Q_2\cap J_{\tilde{u}_k}} (1- (e_n,\nu_x)^2 ) (\overline{\tilde{u}}_k^2(x)+\tilde{\ud}_k^2(x))\cH^{n-1}\\
 & \leq C \lim_k e(u_k,2r_k)\\
 & = 0,
\end{align*}
which implies that $\s$ is supported on $Q_2\times \{e_n,-e_n\}$. We also have that (for any $\phi:Q_2\rightarrow \R$ continuous and compactly supported on $Q_2$)
\begin{align*}
\int_{Q_2\times S^{n-1}} \phi(x) d\s(x,\nu') &= \lim_k \int_{Q_2\times S^{n-1}} \phi(x) d\s_k(x,\nu')\\
& = \lim_k r^{1-n}_k \int_{Q_{2 r_k}}\phi(x/r_k) d\mu_{u_k}(x)\\
& = \int \phi d\mu,
\end{align*}
which means that 
\begin{equation}\label{eq:exlbi3}
\s = \mu \otimes (s \d_{e_n} + (1-s)\d_{-e_n})
\end{equation}
for some $s\in [0,1]$ (using here that $\s,\s_k\geq 0$).

We apply this observation to the first variation formula from Lemma \ref{lem:EL}. Fix any vector field $T\in C_c^{1}(Q_{3/2};\R^n)$, with $\|T\|_{C^1}\leq 1$. Plugging $u_k$ into the formula with the scaled vector field $r_k T(x/r_k)$, we get
\[
  |\int_{Q_2\times S^{n-1}}\dvg\nolimits^{\nu'}T(x) d\s_k(x,\nu') | \leq C(n) t \1E(u_k,2r_k) + 1\2 + \frac{\L_k r_k^{s}}{t} + CE(u_k,2r_k).
\]
for $|t|<c_0$. The energy terms were lumped together and moved to the right, the entire expression was divided by $r_k^{n-1}$, and the left integral was suggestively rewritten. The right-hand side goes to $0$ upon first taking $k\rightarrow \8$ and then $t\rightarrow 0$. This gives
\begin{align}
 0 &= \lim_k \int_{Q_2\times S^{n-1}}\dvg\nolimits^{\nu'}T(x) d\s_k(x,\nu') \nonumber \\
   &= \lim_k \int_{Q_2\times S^{n-1}}\dvg T(x) - \nu' \cdot \n T(x)\nu'  d\s_k(x,\nu') \nonumber \\
   &= \int_{Q_2\times S^{n-1}}\dvg T(x) - \nu' \cdot \n T(x)\nu'  d\s(x,\nu') \nonumber\\ 
   &=\int_{Q_2}\dvg\nolimits^{e_n}T(x) d\m.  \label{eq:exlbi4}
\end{align}
The third line used that the integrand is just a continuous function of $x$ and $\nu'$ (and the weak-* convergence), while the fourth line follows directly from \eqref{eq:exlbi3}.

Choose as the vector field $T$ any function $\eta\in C_c^1(Q_{3/2})$ multiplied by a vector $e\in \pi$. After computing the tangential divergence, \eqref{eq:exlbi4} gives
\[
 0 = \int_{Q_2} \dvg\nolimits^{e_n}T d\m   \int_{Q_2}\p_e \eta d\m,
\]
which implies that $\m\mres Q_{3/2}$ is invariant under translations in $\pi$. It follows from this and the density estimate \eqref{eq:exlbi5} that $K\cap Q_{3/2}$ must be contained in a finite union of affine hyperplanes  $\{\pi_i\}_{i=-N_-}^{N_+}$ parallel to $\pi$, indexed so that if $\pi_i = \pi + \a_i e_n$, then $\a_i$ is increasing and $\pi_0 = \pi$. Also set $\a_{N_+ + 1}=3/2$ and $\a_{N_- - 1}=-3/2$. We know that as $0\in K_{\tilde{u}_k}$ for every $k$, $0\in K$, and so $\pi_0=\pi$ is consistent. Recall that $u_\8$ is locally constant away from $K_{u_\8}\ss K$, so we may set $u_{\8,i}$ to be this constant value of $u_\8$ on the region $\{\a_i<x_n<\a_{i+1}\}$.

Take $x\in \pi_i$ any cylinder $Q_t(x)$ intersecting only one of these hyperplanes $\pi_i$: we have
\[
 f(u_k, r_k x,t r_k,e_n)\rightarrow 0
\]
from the Hausdorff convergence of $K_{u_k}/r_k$ to $\pi_i$ on $Q_t(x)$. Applying Lemma \ref{lem:upperdensity} to $Q_{t r_k}(r_k x)$ with progressively smaller $\g$, we see that
\[
\mu(Q_{t/2}(x)) \leq \liminf_k r_k^{1-n}\mu_{u_k}(Q_{tr_k/2}(r_k x)) \leq (u_{\8,i}^2 + u_{\8,i-1}^2) \w_{n-1} \1\frac{t}{2} \2^{n-1}.
\]
If we apply \eqref{eq:ldc3} and \eqref{eq:ldc4} of Lemma \ref{lem:lowerdensity} instead, we get (using $\mu^+ (\p Q_{t/2}(x))\leq \mu(\p Q_{t/2}(x)) =0$) that
\[
\mu^+(Q_{t/2}(x)) \geq \limsup_k r_k^{1-n}\mu^+_{u_k}(Q_{tr_k/2}(r_k x)) \geq u_{\8,i}^2 \w_{n-1} \1\frac{t}{2} \2^{n-1}
\]
and
\[
\mu^-(Q_{t/2}(x)) \geq \limsup_k r_k^{1-n}\mu^-_{u_k}(Q_{tr_k/2}(r_k x)) \geq u_{\8,i-1}^2 \w_{n-1} \1\frac{t}{2} \2^{n-1}
\]
for almost every $t$. Combining with $\mu^+ + \mu^- \leq \mu$, it follows that all of the inequalities in these three estimates must be equalities, and moreover 
\begin{equation}\label{eq:exlbi7}
\mu^+ \mres \pi_i = u_{\8,i}^2 d\cH^{n-1}\mres \pi_i \qquad \mu = \m^+ + \m^-.
\end{equation}
Furthermore, we have that for each $i\in [N_-,N_+]$ the values $u_{\8,i},u_{\8,i+1}$ may not both be $0$: if $\tilde{u}_k$ converges to $0$ uniformly, it actually is $0$ for sufficiently large $k$ (using \eqref{eq:dregsize}), and we get a contradiction to \eqref{eq:ldc1} in Lemma \ref{lem:lowerdensity}.

We are now in a position to show that if one of the conclusions fails along the entire sequence $u_k$, we obtain a contradiction. First, assume that \eqref{eq:exlbc1} holds, but \eqref{eq:exlbc2} does not. This means that from 
\[
\m^+(Q_1)\leq \liminf \frac{\m^+_{u_k}(Q_{r_k})}{r^{1-k}} \leq \g \w_{n-1} \d^2.
\]
On the other hand, we have shown in \eqref{eq:exlbi7} that
\begin{equation}\label{eq:exlbi6}
\m^+(Q_1) = \w_{n-1}\sum_{\{i:-1<\a_i<1\}} u_{\8,i}^2.
\end{equation}
As $u_{\8,i}\in \{0\}\cup [\d,\d^{-1}]$ for each $i$ (from \eqref{eq:dregsize} again), we must have that $u_{\8,i}=0$ for every $i$ with $Q_1\cap \pi_i \neq \emptyset$. We also saw that $u_{\8,i}$ may not equal $0$ for any two consecutive $i$, which implies that $K\cap Q_1 = \pi_0 (=\pi)$ and that $u_{\8,0}=0$. It follows from the Hausdorff convergence of $K_{\tilde{u}_k} \rightarrow K = \pi$ on $Q_1$ that for large $k$, $f(u_k,\g r_k) < 1-\g$. We also have $u_k(0,r_k/2)=0$ for large $k$, and this means \eqref{eq:exlbc2}, holds; this is a contradiction.

Next, assume that (B) does not hold. In this case, both
\[
\m^\pm (Q_1)\leq \liminf \frac{\m^\pm_{u_k}(Q_{r_k})}{r^{1-k}} \leq \g \w_{n-1} \d^2.
\]
Using \eqref{eq:exlbi6}, we see that $\m^+(Q_1)=0$. Likewise, $\m^-(Q_1)=0$, and so from \eqref{eq:exlbi7} $\m(Q_1) = \m^+(Q_1)+\m^-(Q_1) = 0$. This contradicts that $0\in K = \supp \m$.

Finally, assume that (C) fails. In this case, from the local uniform convergence of $\tilde{u}_k\rightarrow u_\8$ away from $K$, we have that $u_\8 = a $ on $(Q_r\sm K) \cap \{x_n \in (0,r/2)\}$. In particular, $u_{\8,0}=a$. On the other hand, if \eqref{eq:exlbc4} we also have
\begin{align*}
u_{\8,0}^2 \w_{n-1} &= \m^+(Q_1\cap \pi)\\
	&\leq \m^+(Q_1\cap \{|x_n|<(1-\g)\}) \\
	&\leq \liminf \frac{\m^+_{u_k}(Q_{ r}\cap \{|x_n|\leq (1-\g)r_k\})}{r^{n-1}_k}\\
	&\leq \g a^2 \w_{n-1}. 
\end{align*}
This is a contradiction.
\end{proof}

We remark that the only reason to assume that $x\in A^+$ in the lemma above was so that the condition on the oscillation along the line segment made sense.

\begin{lemma}\label{lem:flatbyex} Let $u\in \sQ(Q_{2r},\L,4r)$. Assume $0\in A^+$, and $x=(x',x_n)\in Q_r \cap K_u$ with
\[
 |x_n|>\frac{1}{10}|x'|.
\]
Then there is an $\e_4>0$ universal such that if for some number $a$,
\[
 \sup_{\r<2r}\1e(u,\r) + E(u,\r) + |a - u(0,\r)|\2  + \L r^s +f(u,2r) \leq \e_4, 
\]
and also
\[
 \sup_{\r<r}\1 e(u,x,\r) + E(u,x,\r)\2\leq \e_4,
\]
then we have that
\begin{equation}\label{eq:flatbyexc}
 \frac{\mu^+(Q_{\frac{3|x_n|}{4}}(x))}{\w_{n-1}\1\frac{3|x_n|}{4}\2^{n-1}}\leq \frac{3\d^2}{4}.
\end{equation}
If $u(0,-r/2)=0$, there are no $x\in Q_r \cap K_u$ for which the above hypotheses hold. 
\end{lemma}

\begin{proof} We argue by contradiction. Assume first that \eqref{eq:flatbyexc} does not follow: we have that there is a sequence of $u_k,a_k,r_k,\L_k,x_k$ such that
\begin{equation}\label{eq:flatbyexi1}
 \sup_{\r<2}\1e(u_k,\r r_k) + E(u_k,\r r_k) + |a_k - u_k(0,\r r_k)|\2  +\L_k r_k^s +f(u_k,2r_k)\rightarrow 0,
\end{equation}
as well as 
\begin{equation}\label{eq:exlipint}
|(x_k)_n|> \frac{1}{10}|x_k'|,
\end{equation}
 and
\[
 \sup_{\r<1}\1e(u_k,x_k,\r r_k) + E(u_k,x_k,\r r_k)\2 \rightarrow 0,
\]
but nevertheless
\begin{equation}\label{eq:flatbyexi2}
 \frac{\mu^+_{u_k}(Q_{\frac{3|(x_k)_n|}{4}}(x_k))}{\w_{n-1}\1\frac{3|(x_k)_n|}{4}\2^{n-1}}\geq \frac{3\d^2}{4}.
\end{equation}
for every $k$.

Let us first agree that we may extract a subsequence $a_k \rightarrow a$, and then replacing $a_k$ with $a$ in the above does not change the fact that the expression goes to $0$. From \eqref{eq:flatbyexi1}, we may apply Lemma \ref{lem:upperdensity} on $Q_{\g r_k}$ for sufficiently large $k$ to get that
\begin{equation}\label{eq:flatbyexi3}
 \frac{\m_{u_k}(Q_{\g r_k})}{\w_{n-1}r_k^{n-1}} \leq (a^2 + u_k^2(0,-\frac{r_k}{2})+ \frac{\d^2}{16}),
\end{equation}
is valid for some $\g=\g(\d)>1$.  Now also using the smallness of $E(u_k,2r_k)$ and Lemma \ref{lem:lowerdensity} on $\m^+$ and $\m^-$, we obtain that
\[
 \frac{\mu^+_{u_k}(Q_{\g r_k})}{\w_{n-1}r^{n-1}}\geq a^2 - \frac{\d^2}{16} \qquad \frac{\mu^-_{u_k}(Q_{\g r_k})}{\w_{n-1}r^{n-1}}\geq u_k^2(0,-\frac{r_k}{2}) - \frac{\d^2}{16}.
\]
Combining with \eqref{eq:flatbyexi3} gives
\begin{equation}\label{eq:flatbyexi5}
 \frac{\mu^+_{u_k}(Q_{\g r_k})}{\w_{n-1}r^{n-1}}\leq a^2+\frac{\d^2}{4}.
\end{equation}
Let us also guarantee that $|x_k|\leq (\g-1)r_k$ by choosing the flatness small enough and using \ref{eq:exlipint}.

Next, by Lemma \ref{lem:exlb} (part (C) with $1-\g<\frac{1}{4}$) we have that for $\r\leq r_k$,
\begin{equation}\label{eq:flatbyexi4}
 \frac{\mu^+_{u_k}(Q_{\r }\cap \{|y_n|\leq \frac{\r}{4}\})}{\w_{n-1}\r^{n-1}}\geq a^2 - \frac{\d^2}{4}.
\end{equation}
Consider the quantity
\[
 \k_k(\r)=\frac{\mu^{+}_{u_k}(Q_{\r} \cup Q_{\r }(x_k))}{\w_{n-1}\r^{n-1}}.
\]
For $\r = \r^*_k:=\frac{3|(x_k)_n|}{4}$, the two regions $Q_{\r^*_k}(x_k)$ and $Q_{\r^*_k} \cap \{|y_n|\leq \frac{\r^*_k}{4}\}$ are disjoint, and so we have that $\k_k(\r^*_k)$ is at least $\d^2/2 +a^2$ from combining \eqref{eq:flatbyexi2} and \eqref{eq:flatbyexi4}. As $|x_k|\leq (\g-1)r_k$, when $\r=r_k$ we have that the union of the two cylinders $Q_\r,Q_\r(x_k)$ is inside $Q_{\g r_k}$, and so $\k_k(\r_k)$ is at most $a^2 + \d^2/4$ by \eqref{eq:flatbyexi5}. Let $\r_k$ be the smallest $\r\in [\r_k^*,r_k]$ for which 
\[
 \k_k(\r) \leq a^2 + \d^2/4;
\]
our discussion guarantees that such a $\r_k$ exists and lies inside this interval. As $\k_k$ is left-continuous in $\r$ (and only increases at each jump), we will have that
\[
 \k_k(\r_k) = a^2 + \d^2/4.
\]

Observe that the two cylinders $Q_{\r_k},Q_{\r_k}(x_k)$ may not be disjoint, for otherwise we would have $\k_k(\r_k) \geq a^2 +\d^2/2$ from \eqref{eq:flatbyexi2} and \eqref{eq:flatbyexi4}. It follows that $|x_k|< 2\sqrt{2}\r_k$, and that $Q_{\r_k},Q_{\r_k}(x_k)\ss Q_{8 \r_k}$. If we set
\[
S = \begin{cases}
	8 & \lim_{k\rightarrow \8} \frac{r_k}{\r_k} = \8\\
	2 & \lim_{k\rightarrow \8} \frac{r_k}{\r_k} < \8
\end{cases}
\]
(we may always assume the limit exists in $[1,\8]$ by passing to a subsequence), then for all $k$ large enough we have $Q_{\r_k},Q_{\r_k}(x_k)\ss Q_{S \r_k}\ss Q_{2 r_k}$. Indeed, the second inclusion is clear, while we just checked the first one when $S=8$. If $S=2$, then form \eqref{eq:flatbyexi1} $|x_n|/r_k\rightarrow 0$, so by \eqref{eq:exlipint} we have $|x|/r_k\rightarrow 0$. As in this case $r_k/\r_k$ remains bounded, it follows that $|x|/\r_k\rightarrow 0$, meaning that for large $k$ we have $Q_{\r_k}(x)\ss Q_{2\r_k}$.

Next, we let $\tilde{u}_k$ be the rescaled functions $\tilde{u}_k(x) = u(\r_k x)$, which are defined on $Q_{2 r_k/\r_k}$, which contains the  cylinder $Q_{S}$. The energy bound gives that
\[
 \frac{1}{\r_k}\int_{Q_{S}}|\n \tilde{u}_k|^2 d\cL^n \rightarrow 0.
\]
Along a subsequence, we have $x_k\rightarrow x\in Q_{S}$ with $|x_n|\geq \frac{1}{10}|x'|$ and $Q_1(x)\ss Q_S$. We may also take $K_{\tilde{u}_k} \rightarrow K$ locally in Hausdorff topology, as well as $\r_k^{1-n}\mu_{u_k}(\r_k \cdot) \rightarrow \m$ and $\r_k^{1-n}\mu^{+}_{u_k}(\r_k \cdot) \rightarrow \mu^+$ in the weak-* sense as measures, all on $Q_{S}$.  The functions $\tilde{u}_k$ satisfy the hypotheses for the $SBV$ compactness theorem, and so converge in $L^1$ to a function $u_\8 \in SBV$. The energy bound tells us that $u_\8$ is locally constant on $Q_{S} \sm K_{u_\8}$.

Also, using the weak-* convergence and the definition of $\r_k$,
\begin{equation}\label{eq:exlipint2}
 \mu^{+}(Q_1 \cup Q_1(x))\leq \liminf \r_k^{1-n} \m^+_{u_k}(Q_{\r_k} \cup Q_{\r_k}(x_k)) = \w_{n-1}(a^2+ \d^2/4),
\end{equation}
while
\begin{equation}\label{eq:exlipint3}
 \mu^{+}(\bar{Q}_t \cup \bar{Q}_t(x))\geq \limsup \r_k^{1-n} \m^+_{u_k}(\bar{Q}_{t \r_k} \cup \bar{Q}_{ t \r_k}(x_k)) \geq \w_{n-1}(a^2+ \d^2/4)t^{n-1}
\end{equation}
for each $t\in(\frac{3}{4}|x_n|,1)$.

Arguing as in the proof of Lemma \ref{lem:exlb}, we have that $K$ is contained in a finite union of hyperplanes $\pi_i$ parallel to $\pi$, with $\pi_0=\pi$ and $\pi_i=\pi + \a_i e_n$ with $\a_i$ increasing. As in \eqref{eq:exlbi7},
\[
\mu^+ \mres \pi_i=u_{\8,i}^2 \cH^{n-1}\mres\pi_i,
\]
\[
\mu^- \mres \pi_i=u_{\8,i-1}^2 \cH^{n-1}\mres\pi_i,
\]
and $\mu = \mu^+ + \mu^-$, with $u_{\8,i}$ being the constant value of $u_\8$ in the region between $\pi_i$ and $\pi_{i+1}$. It follows from the oscillation condition on $\tilde{u}_k$ along the line segment $\{(0,t):t\in (0,S)\}$, from \eqref{eq:flatbyexi1}, that $u_{\8,0}=a$ and so $\m^+\mres \pi = a^2 \cH^{n-1}\mres \pi$.

First consider the case of $x\in \pi$. This, coupled with the bound $|x_n|\geq \frac{1}{10}|x'|$, gives that $x=0$. Choosing a small cylinder $Q_t$ which has $Q_t \cap K \ss \pi$ and evaluating $\m^+$, we get that
\[
 \m^+(Q_t)=a^2 \w_{n-1} t^{n-1},
\]
which contradicts \eqref{eq:exlipint3}. On the other hand, say $x\notin \pi$. Then $x\in \pi_i$ for another hyperplane, and we have by assumption that
\[
 \m^+(\bar{Q}_{\frac{3|x_n|}{4}}(x))\geq \limsup_k \r_k^{1-n}\m_{u_k}(\bar{Q}_{\r^*_k}(x_k))=\1\frac{3|x_n|}{4}\2^{n-1}\frac{3\d^2}{4}>0.
\]
As this cylinder does not intersect $\pi$, this guarantees the existence of a plane $\pi_j\ss K$, $\pi_j\neq \pi$ which intersects $Q_1(x)$ and has $u_{\8,j}\geq \d$ (note that we do not claim $x\in \pi_j$). We thus have
\[
 \m^{+}(Q_1(x)\cap \pi_j) \geq \d^2 \w_{n-1}.
\]
As the two hyperplanes $\pi_j,\pi$ are disjoint,
\[
 \m^{+}(Q_1(x)\cup Q_1) \geq (\d^2 +a^2) \w_{n-1},
\]
which contradicts \eqref{eq:exlipint2}.

If $u(0,-r/2)=0$, the same argument goes through with $\m$ in place of $\m^+$ everywhere, the key point being that we know
\[
 \m_{u_k}(Q_{\g r_k}) \leq (a^2 + \frac{\d^2}{16})\w_{n-1}
\]
simply from \eqref{eq:flatbyexi3}. When we select $\r_k$, we now choose the smallest $\r\in (0,r_k)$ for which $\k_k(\r)\leq a^2+\d^2/4$. For $\r<|x_k|/4$, the two cylinders $Q_\r,Q_\r(x_k)$ are disjoint, so applying Lemma \ref{lem:exlb} to them individually gives $\k_k(\r)\geq a^2 +3\d^2/4$. Hence $\r_k\geq|x_k|/4$. We then argue as above, and obtain the same contradiction if $x\in \pi$.  Having a limit point $x\notin \pi$ gives a contradiction as well, as we know the density of $\m$ at $x$ (and hence along the entire affine plane $\pi_i$ containing it) will be at least $\d^2$. 
\end{proof}

Notice that we may apply Lemma \ref{lem:exlb} (for $1-\g$ small enough) to the cylinder $Q_{3|x_n|/2}(x)$ of Lemma \ref{lem:flatbyex}, and we will then conclude that $f(u,x,7|x_n|/10)<1-\g$ and $u(x',x_n + 3 |x_n|/8)=0$. This is indeed the application we have in mind.

\section{Lipschitz Approximation}\label{sec:lip}

We are now in a position to show that flat quasiminimizers have $K_u$ well approximated by a pair of Lipschitz graphs. We will use the notation introduced in the previous section.

\begin{theorem}\label{thm:lip} Let $u\in \sQ(Q_{6r},\L,12r)$. Then there are universal constants $\e_L,C$ such that if for some unit vector $\nu_* \in S^{n-1}$
\begin{equation}\label{eq:liph}
 M := ee(u,6r,e_n,\nu_*)+E(u,6r)+ r + \L r^s\leq \e_L, \qquad f(u,6r) + 1-(e_n\cdot \nu_*)^2 \leq \e_L
\end{equation}
then there is a pair of $1/10$-Lipschitz functions $g_-,g_+:\pi \rightarrow \R$ with graphs $\G^-,\G^+$, with $g_-\leq g_+$, for which the following hold:
\begin{enumerate}
 \item $|g_\pm|\leq Cf^{\frac{1}{n+1}}(u,6r)$.
 \item $\cH^{n-1}(Q_r \cap K \sm (\G^+ \cup \G^-))\leq C Mr^{n-1}$.
 \item $\cH^{n-1}(Q_r \cap \G^{\pm}\sm A^{\pm})\leq C M r^{n-1}$.
\end{enumerate}
If either of $u(x,\pm2r)=0$, we may take $g_- = g_+$.
\end{theorem}

\begin{proof}
 First, by choosing $\e_L$ small enough, we may guarantee $f_\8(u, 5 r) \leq \frac{1}{100}$. Set $a_\pm=u(0,\pm 2r)$; we will work with the more difficult case when both are nonzero, and then explain what to do if one vanishes.
 
We define the set $G$ in the following way, for $\l$ to be chosen later:
\[
 G=\3x \in Q_{2r}\cap K : \sup_{\r\leq 4r} e(u,x,\r,e_n) + E(u,x,\r,e_n) \leq \l \4.
\]
Also define the following subset $G^+$ (the idea being that one point in $G$ and another point in $G^+$ are susceptible to Lemma \ref{lem:flatbyex}):
\[
 G^{+}=\3x \in G\cap A^+ : \sup_{0<\r\leq 4 r}|u(x + \r e_n)-a_+|\leq \l \4.
\]
Define $G^{-}$ analogously. 

Take any two points $x,y\in G^{+}$. By taking $\l$ and $\e_L$ small enough that Lemma \ref{lem:flatbyex} may be applied centered at every such point $x\in G^+$, with the other point at $y$, we see that either 
\[
 |x_n-y_n|\leq \frac{1}{10}|x'-y'|,
\]
or else 
\[
 \m^+(Q_{3|x_n-y_n|/4}(y))\leq \frac{3\d^2}{4} \w_{n-1}(\frac{3|x_n-y_n|}{4})^{n-1},
\]
with the latter a contradiction to Lemma \ref{lem:exlb}, part (A), the fact that $y\in G^+$, and the fact that $a_+$ is nonzero. 

A consequence is that for each $x'\in D_{2r}$, there is at most one point $z\in G^+$ with $z'=x'$. Defining $g_+^0(x')$ so that $z=(x',g_+^0(x'))$, we have that $g_+^0$ is a $\frac{1}{10}$-Lipschitz function defined on $\pi(G^+)$, with $|g_+^0|\leq C f^{n+1}$. Extending $g_+^0$ to $\pi$ in a manner preserving these two properties, we have constructed a function satisfying $(1)$. We may likewise construct $g_-^0$. We will now prove several additional estimates on the graphs of $g_\pm^0$  (we will later use $g_\pm^0$ to construct the actual $g_\pm$ promised in the theorem, which will in addition satisfy $g_-\leq g_+$). To begin with, we tackle the set $B_E=K\cap Q_{2r} \sm G$, which contains points whose excess or energy are large at some scale.

By making sure that $\e_L$ is small enough, we may guarantee that $1-(e_n \cdot \nu_*)^2 \leq \l/10$, where $\nu_*$ is the secondary unit normal vector in the hypothesis \eqref{eq:liph}.  If $x\in B_{E}$, then there is a radius $\r_x$ with $\r_x<4r$ and $e(u,x,\r_x,e_n)+E(u,x,\r_x,e_n)>\l$, which in particular means that
\[
\l \r_x^{n-1} \leq \int_{B_{\sqrt{2}\r_x}(x)}|\n u|^2 d\cL^n + \int_{B_{\sqrt{2}\r_x}(x)\cap J_u}\min\{1-(\nu_*,\nu_x)^2,1-(\nu'_*,\nu_x)^2\}^2 d\cH^{n-1}:=\a(x).
\]
Here $\nu_*' = 2 (e_n\cdot \nu_*) e_n - \nu_*$ is the reflection of $\nu_*$ about $e_n$, and $\nu_x$ is the approximate jump vector. The collection of all such balls covers $B_{E}$, and by the Besicovitch covering lemma we may find a finite-overlapping subcover $\{B_{\r_i}(x_i)\}_{i}$. Then using the density upper bound from \eqref{eq:dregden},
\begin{align}
 \cH^{n-1}(B_{E}) &\leq \cH^{n-1}(\cup_{i}B_{\r_i}(x_i) \cap K)\nonumber\\
 &\leq C\sum_i \r_i^{n-1}\nonumber\\
 &\leq \sum_i \frac{C}{\l}\a(x_i)\nonumber\\
 &\leq C r^{n-1} \1ee(u,6r,e_n,\nu_*)+E(u,6r)\2\nonumber\\
 &\leq CM r^{n-1}. \label{eq:lipi1}
\end{align}
It follows that  
\begin{equation}\label{eq:lipi2}
\cL^{n-1}(\pi(B_{E}))\leq \cH^{n-1}(B_E)\leq CM r^{n-1}
\end{equation}
as well.


We now estimate the size of the projection of $B_{OS}=G\cap A^+ \sm G^+$ onto the plane $\pi$. First of all, we have that along the disk $D_{2r}\times \{r\}$, $|u(x)-a_+|\leq C(n)M^{1/2}\leq \frac{\l}{2}$ if $\e_L$ is chosen sufficiently small, from Lemma \ref{lem:intreg}. For each point $x\in B_{OS}$, the line segment $\{(x',t):t \in (x_n,r]\}
$ does not intersect $K_u$, and contains a point $(x',t(x'))$ with $|u(x',t(x'))-a_+|\geq \l$. Applying the fundamental theorem of calculus,
\[
 \frac{\l^2}{4}\leq |u(x',t(x'))-u(x',r)|^2\leq Cr \int_{t(x')}^r |\n u (x',\t)|^2 d\cL^1 (\t).
\]
Integrating over $\pi(B_{OS})$, we get
\begin{equation}\label{eq:lipi3}
 \cL^{n-1}(\pi(B_{OS})) \leq C r \int_{D_{2r}\times (-r,r)}|\n u|^2 d\cL^n \leq C r E(u,6r) r^{n-1} \leq CM r^{n-1}.
\end{equation}
Note that we used the crude bound $r\leq 1$, as we will not require finer estimates on this error.

Recalling the estimate \eqref{eq:ldc2} on the ``holes'' $D_{2r}\sm \pi(A_{4r}^+)$ from Lemma \ref{lem:lowerdensity}. we may conclude that
\begin{equation}\label{eq:lipi8}
 \cL^{n-1}(D_{2r} \sm \pi(G^+) ) \leq CMr^{n-1}.
\end{equation}
Indeed, we have estimated the measure of the set $D_{2r}\sm \pi (A^+)$ from \eqref{eq:ldc2}, $\pi(K\sm G)$ in \eqref{eq:lipi2}, and $\pi(G \cap A^+ \sm G^+)$ in \eqref{eq:lipi3}.

The same argument applies to $G^-$, and so we have
\begin{equation}\label{eq:lipi9}
\cL^{n-1}(D_{2r} \sm \overline{\pi(G^+)\cap \pi(G^-)})\leq  \cL^{n-1}(D_{2r} \sm (\pi(G^+)\cap \pi(G^-))) \leq CMr^{n-1}.
\end{equation}
Let $Z$ be the cylinder $\pi^{-1}(D_{2r}\cap \overline{\pi(G^+)\cap \pi(G^-)})$; we will now show that $G\cap Z$ consists entirely of the union of $\bar{G}^+$ and $\bar{G}^-$. In fact, let us prove a stronger statement:

\textbf{Claim:} There is a $T>0$ universal such that for any point $z\in G$ and $x'\in \pi(G^+)\cap \pi(G^-)$, if $r_+=|g^0_+(z')-z_n|$, $r_-=|g^0_-(z')-z_n|$, and $d=|z'-x'|$, then $\min\{r_+,r_-\}\leq T d$. 

Notice this would imply our original statement, as for any $z\in Z\cap G$ there is a sequence of $x'\in \pi(G^+)\cap \pi(G^-)$ with $d\rightarrow 0$; if $z\notin \bar{G}^+\cup \bar{G}^-$, then both of $r_+,r_-$ would remain bounded from below along the sequence and we would arrive at a contradiction.

We now prove the claim, assuming for contradiction that $r_+\geq T d$ and $r_- \geq T d$ for some $x'$ and $z$. We may assume, without loss of generality, that $r_+\leq r_-$. Set $x^\pm=(x',g^0_\pm(x'))$. There will be two cases: either $r_+\leq \frac{3 r_-}{8}$ or $r_+\in (3 r_-/8,r_-]$; we begin with the latter. 

Apply Lemma \ref{lem:flatbyex} to the cylinder $Q_{2r}(x^-)$ and the point $z$. Provided $T\geq \frac{1}{10}$, this gives that
\[
 \m^-(Q_{3 r_-/4}(z))\leq \frac{3\d^2}{4}\w_{n-1}\1\frac{3r_-}{4}\2^{n-1}.
\]
We now apply part (A) of Lemma \ref{lem:exlb} to $Q_{3 r_-/2}(z)$ with $\g$ chosen close to one. This gives us
\begin{equation}\label{eq:lipi4}
 f(u,z,3 \g r_-/4)<1-\g
\end{equation}
and also that $u(z - \frac{3r_-}{8} e_n)=0$. Now do the same for the point $x^+$, to get the weaker flatness property
\[
 f(u,z,3 \g r_+/4)<1-\g,
\]
but also that $u(z + \frac{3r_+}{8} e_n)=0$. Now, observe that as $r_+\in (3 r_-/8,r_-]$, if we ensure that $1-\g<\frac{3}{8}$, we have from the first flatness property \eqref{eq:lipi4} that $u=0$ at $z + \frac{3r_-}{8} e_n$ as well. Now choosing $1-\g<\e_2$ will lead to a contradiction to the minimum jump size estimate \eqref{eq:ldc1} of Lemma \ref{lem:lowerdensity} (using that $z\in G$ and choosing $\l\leq \e_2$). 

If $r_+$ is small, and in particular $r_+<3r_-/8$, we proceed differently. Start by applying Lemma \ref{lem:flatbyex} to the cylinder $Q_{2r}(x^-)$ and the point $x^+$, to obtain
\[
 \m^-(Q_{3 (r_--r_+)/4}(z))\leq \frac{3\d^2}{4}\w_{n-1}\1 \frac{3(r_--r_+)}{4} \2^{n-1}.
\]
Again apply part (A) of Lemma \ref{lem:exlb} to this cylinder, this time making sure that $1-\g<\min\{\e_4,\frac{1}{5}\}$. This gives
\begin{equation}\label{eq:lipi5}
 f(u,x^+, \frac{3 \g}{4} (r_--r_+))<1-\g<\e_4
\end{equation}
and $u(x^+ - \frac{3}{8}(r_- -r_+) e_n)=0$. Now, form the assumptions on $r_+,r_-$, and $\g$, we have
\[
 \frac{3 \g}{4} (r_--r_+) \geq \frac{3}{4}\cdot \frac{4}{5}\cdot \frac{5}{8}r_- = \frac{3}{8}r_- \geq r_+.
\]
Provided $T\geq 1$, we then have that $z\in Q_{\frac{3}{4}(r_--r_+)}(x^+)$, and so we apply Lemma \ref{lem:flatbyex} \emph{to the smaller cylinder} $Q_{\frac{3\g}{4}(r_--r_+)}(x^+)$ and the point $z$. Note carefully that we are using our new flatness bound \eqref{eq:lipi5} here, not the global one. As $u(z - \frac{3\g}{8}(r_- -r_+) e_n)=0$, the extra assumption in that lemma is verified, which means we attain a contradiction via the secondary, stronger, conclusion. We have successfully proved the claim.

We have shown that $Q_{2r}\cap G \cap Z$ is contained fully in $\bar{G}^+\cup \bar{G}^-$. It remains to show that the $\cH^{n-1}$ measure of $Q_{r}\cap G\sm Z$ is small; we currently only know that the measure of the projection of this set onto $\pi$ is controlled by $CM r^{n-1}$.  Let $x=(x',x_n)\in Q_r \cap  G$, and assume $x\notin Z$. Then there is a small disk $D_{\r}(x')$ with the property that $D_{\r}(x')\times (-r,r)$ does not touch $Z$; this is simply because $Z$ is closed and $x'$ is not in its projection. Let $\r_x$ be the largest $\r$ for which this remains true (we certainly have $\r_x<r$, as the projection of $Z$ has nearly full $\cL^{n-1}$ measure in $D_{2r}$). We may then find two points on $\p D_{\r}(x')\times (-r,r)$ which have the same projection: a point $z_1\in \bar{G}^+$ and a point $z_2\in \bar{G}^-$. By the claim, we have then that for any $y\in G\cap D_{\r_x}(x')\times (-r,r)$, we must have either $|(y-z_1)_n|\leq 2 T \r_x$ or $|(y-z_2)_n|\leq 2 T \r_x$; in particular, by using the upper density estimate on $K$ from \eqref{eq:dregden}  on the balls $B_{2\sqrt{1+T^2}\r_x}(z_i)$, this means that
\[
 \cH^{n-1}(G\cap D_{\r_x}(x')\times (-r,r)) \leq \sum_{i=1}^2 \cH^{n-1}(K\cap B_{2\sqrt{1+T^2}\r_x}(z_i)) \leq C \r_x^{n-1}.
\]

We now find a finite-overlapping cover of $D_r\sm \pi(Z)$ by disks $D_{\r_i}(x'_i)$ of this sort. This gives
\begin{align*}
 \cH^{n-1}&(Q_r \cap G\sm Z) \leq \sum_i \cH^{n-1}(G\cap (D_{\r_i}(x_i')\times (-r,r))) \\
 &\leq C \sum_i \cL^{n-1}(D_{\r_i}(x_i'))\\
 &\leq C \cL^{n-1}(\cup_i D_{\r_i}(x_i'))\\
 &\leq C\cL^{n-1}(D_{2r} \sm \pi(Z))\\
 &\leq CMr^{n-1}.
\end{align*}

Combining with our previous estimate \eqref{eq:lipi1} of $B_E$, this means
\begin{equation}\label{eq:lipi6}
 \cH^{n-1}( Q_r \cap K \sm (\bar{G}^+ \cup \bar{G}^-))\leq CM r^{n-1}.
\end{equation}

Define $g_+$ to be the maximum of $g_{+}^0,g_{-}^0$, and $g_-$ to be the minimum; these satisfy (1). Note that as on $D_r\cap \pi(Z)$, $g_-^0\leq g_+^0$ already, the support of $g_\pm - g_\pm^0$ is contained in $D_r\sm \pi(Z)$.

If we denote the graph of $g_\pm$ by $\G^\pm$, and recall that by construction $\bar{G}^+ \cup \bar{G}^- \ss \G^+ \cup \G^-$, \eqref{eq:lipi6} implies
\begin{equation}\label{eq:lipi7}
 \cH^{n-1}( Q_r \cap K \sm (\G^+ \cup \G^-))  \leq CM r^{n-1}.
\end{equation}
We have completed the argument for $(2)$. 

Now $(3)$ is an immediate consequence of two observations: first, the $\cL^{n-1}$ measure of the projection $\pi(Q_r \cap \G^+ \sm A^+)$ controls $\cH^{n-1}(Q_r \cap \G^+ \sm A^+)$, simply because $g^+$ is Lipschitz. Second, any $x'\in \pi(Q_r \cap \G^+ \sm A^+)$ either has no point $x=(x',x_n)$ in $G^+$, or else this point $x\notin \G^+$. As $G^+ \ss \G^+$ on $Z$, this means that
\[
\cL^{n-1}(\pi(Q_r \cap \G^+ \sm A^+)) \leq \cL^{n-1}(D_r \sm \pi(G^+)) + \cL^{n-1}(D_r \sm \pi(Z)) \leq CMr^{n-1}.
\] 
The first term was estimated using \eqref{eq:lipi8}, while \eqref{eq:lipi9} was used for the second.

If, say, $a_-$ was $0$, define $g_+$ in the same way as above. Then the estimate on $B_E$ goes through in the same way. The only difference is that to prove the claim, we may instead use the final statement in Lemma \ref{lem:flatbyex}, which now automatically ensures that for any point $x$ in $G^+$ and $y\in G$, we have $|x_n-y_n|\leq \frac{1}{10}|x'-y'|$, with no extra complications. The rest of the argument is unchanged.
\end{proof}

\begin{remark}\label{rem:lipest}
 We point out that in the process of proving the theorem above, we have actually obtained estimates superior to those promised. Let $\Theta^+$ be those points in $K\cap \G^+$ which are in $A^+$ and also have either that $x\in \G^-\cap A^-$ or $\lim_{s\nearrow x_n} u(x',s)=0$, and $\Theta^-$ be defined analogously; these will be particularly desirable to work with later. Then 
 \begin{equation}\label{eq:lipest1}
  \cH^{n-1}(Q_r \cap (\Theta^+ \cup \Theta^-) \triangle K)\leq CM r^{n-1}.
 \end{equation}
Indeed, as $\Theta^+ \ss \G^+$, the estimate on $\Theta^+ \sm K$ follows directly from item (3) in the theorem. For the other estimate, 
we note that it would suffice to show that $ S:=\pi(Z) \cap \pi(\bar{G}^+\sm \Theta^+) \sm \pi(K\sm G) $ has $\cL^{n-1}(S)\leq CMr^{n-1}$: indeed,
\[
K \sm (\Theta^+ \cup \Theta^-) \ss (\bar{G}^+ \sm \Theta^+) \cup (\bar{G}^-\sm \Theta^-) \cup (K\sm (\bar{G}^+\cup \bar{G}^-),
\]
with the $\cH^{n-1}$ measure of the last piece controlled by $CMr^{n-1}$ from \eqref{eq:lipi6}. Then
\begin{align*}
\cH^{n-1}(\bar{G}^+ \sm \Theta^+) &\leq \cH^{n-1}(K\sm Z) + \cH^{n-1}(Z\cap \bar{G}^+ \sm \Theta^+) \\
& \leq CMr^{n-1} + C \cL^{n-1}(\pi(Z)\cap \pi(\bar{G}^+ \sm \Theta^+))\\
& \leq CMr^{n-1} + C \cL^{n-1}(S) + C \cL^{n-1}(\pi(K\sm G))\\
& \leq CMr^{n-1} + C \cL^{n-1}(S) \sm \pi(K\sm G)),
\end{align*}
where the second step used that $Z \cap \bar{G^+} \ss \G^+$ and \eqref{eq:lipi6}, while the final step used \eqref{eq:lipi2}. 

Let us assume that $x'\in S$, and moreover that $(x',g_+(x'))\in G^+$ and $(x',g_-(x'))\in G^-$: note that if $g_-(x')=g_+(x')$, then $(x',g_+(x'))\in G^+ \cap \Theta^+$, contradicting that $x'\in S$. The point now is that we may apply Lemma \ref{lem:flatbyex} and then part (A) of Lemma \ref{lem:exlb} to the cylinder $Q_{2r}(x',g_+(x'))$ and the point $(x',g_-(x'))$. This implies (arguing as in the proof of the claim) that $u(x',t)=0$ for $t=g_-(x')+\frac{3}{8}(g_+(x')-g_-(x'))$. However, the line segment $\{(x',s):s\in (g_-(x'),g_+(x'))\}$ cannot intersect $K$: it cannot intersect $K\sm G$, because $x'\notin \pi(K\sm G)$, and it cannot intersect $G$, from the fact that $x'\in \pi(Z)$ and applying the claim. Thus $u$ vanishes on the whole line segment, and so $(x',g_+(x'))\in \Theta^+$, again contradicting $x'\in S$. We have shown that $S \ss D_r \sm (\pi(G^+)\cap \pi(G^-))$, so applying \eqref{eq:lipi8} gives $\cL^{n-1}(S)\leq CMr^{n-1}$, as desired. This proves \eqref{eq:lipest1}.

We also have
\[
 |\m\mres(\Theta^+ \cup \Theta^-) - u^2(0,r/2)\cH^{n-1}\mres \Theta^+ - u^2(0,-r/2)\cH^{n-1}\mres \Theta^-|(Q_r)\leq CMr^{n-1}.
\]
To check this, notice that it suffices to estimate
\[
 \int_{Q_r \cap \Theta^+} |\lim_{t\searrow x_n}u^2(x',t)- u^2 (0,2r)|d\cH^{n-1} +  \int_{Q_r \cap \Theta^-} |\lim_{t\nearrow x_n}u^2(x',t)- u^2 (0,-2r)|d\cH^{n-1}, 
\]
which may be done exactly as in the proof of Lemma \ref{lem:lowerdensity}.
\end{remark}

\section{The Second Excess-Flatness Bound}\label{sec:efbound2}

We begin with a simple cleanup property which deals with the case when $ff^{\frac{1}{n+1}}\ll \frac{h}{r}$, with $h=h(u,r)$ the one obtaining the infimum in the definition of $ff$. In this case, from Proposition \ref{prop:fflatsame} we have that $K$ lies in the union of two disjoint regions, one around the upper plane in the definition of $ff$ and one around the lower one. We show now that in this configuration $u$ must vanish in the region between them, for otherwise we would obtain a contradiction to the density upper bound in Lemma \ref{lem:upperdensity}.

\begin{lemma}\label{lem:zerobetween}Assume that  $u\in \sQ(Q_{r},\L,2r)$. Then there is a universal $\e_5>0$ such that if $K\cap Q_{r/2}\neq \emptyset$ and
\[
 f_\8(u,r)+E(u,r)+r + \L r^s \leq \e_5 \qquad ff^{\frac{1}{n+1}}(u,r)\leq \e_5 \frac{h (u,r)}{r},
\]
then $u$ vanishes on $\pi \cap Q_{r/2}$. Here $h(u,r)$ is the one attaining the infimum in the definition of $ff(u,r)$.
\end{lemma}

\begin{proof}
First, note that from Proposition \ref{prop:fflatsame} we have that $K\cap Q_{r/2}$ is contained in a $Cff^{\frac{1}{n+1}}(u,r)r$ neighborhood of the union of the two planes $\pi_+ = \{x: (x - h e_n)\nu_+ = 0 \}$ and $\pi_-=\{x: (x + h e_n)\nu_- = 0 \}$, where $\nu_+$ and $h$ are the ones attaining the infimum in the definition of $ff(u,r)$. If $\e_5$ is small enough, this neighborhood does not intersect $\pi$, and is the disjoint union of $U_+$ (which contains $\pi_+$) and $U_-$ (which contains $\pi_-$). If the conclusion of the lemma was to fail, then, we would  have that by \eqref{eq:dregsize}, $u\geq \d>0$ on the entirety of $\pi\cap Q_{r/2}$.

We apply Lemma \ref{lem:lowerdensity} to give that
 \[
\frac{\mu_\pm(A^\pm \cap Q_{r/2})}{\w_{n-1}\1\frac{r}{2}\2^{n-1}} \geq u^2(0,\pm r/2)-C\e_5,
 \]
and moreover from \eqref{eq:ldc2}
\[
 \frac{\cL^{n-1}(D_{r/2}\sm \pi(A^+))}{r^{n-1}}\leq C\e_5.
\]
We now claim that 
\begin{equation}\label{eq:zerobetweeni1}
\frac{\mu_+(A^+ \cap U_+)}{\w_{n-1}\1\frac{r}{2}\2^{n-1}} \geq u^2(0,\pm r/2)-C\e_5.
\end{equation}
Indeed, this follows, as for every point $x\in A^+\cap U_-$, we have $d_*(x;\nu_+,h)\geq d(x,\pi_+)\geq \frac{h}{2}$ (recall that $d_*$ is the integrand in the definition of $ff$). Integrating,
\begin{equation}\label{eq:zerpbetweeni2}
 \cH^{n-1}(A^+ \cap U_-) \leq 4 h^{-2} ff(u,r)r^{n+1} \leq C \e_5^2 r^{n-1}.
\end{equation}
Hence
\[
\frac{\mu_+(A^+ \cap U_+)}{\w_{n-1}\1\frac{r}{2}\2^{n-1}} \geq \frac{\mu_+(A^+ \cap Q_{r/2})}{\w_{n-1}\1\frac{r}{2}\2^{n-1}} - \frac{\mu_+(A^+ \cap U_-)}{\w_{n-1}\1\frac{r}{2}\2^{n-1}} \geq u^2(0,\pm r/2)-C\e_5,
\]
implying \eqref{eq:zerobetweeni1}.

As another consequence of \eqref{eq:zerpbetweeni2}, we have that
\[
 \frac{\cL^{n-1}(D_{r/2}\sm \pi(A^-\cap U_-))+\cL^{n-1}(D_{r/2}\sm \pi(A^+\cap U_+))}{r^{n-1}}\leq C\e_5. 
\]
Set $S=\pi(A^-\cap U_-)\cap \pi(A^+\cap U_+)$, and for every $x'\in S$ let $I_{x'}$ be the largest vertical line segment $\{(x',t):t\in (a(x'),b(x'))\}$ containing $(x',0)$ and avoiding $K$ (as $x'\in \pi(A^+)\cap \pi(A^-)$, this is compactly contained in $Q_{r/2}$). We have that $u\geq \d$ on $I_{x'}$. Therefore, arguing as in the proof of Lemma \ref{lem:lowerdensity}, we have that
\[
 \m_-(U_+) \geq \int_S \lim_{t \nearrow b(x')} u^2(t) d\cL^{n-1}(x') \geq \d^2 \cL^{n-1}(S)\geq (\w_{n-1}-C\e_5)\d^2 (\frac{r}{2})^{n-1}.
\]
This gives a density lower bound
 \[
  \frac{\m( Q_{r/2})}{\w_{n-1}\1\frac{r}{2}\2^{n-1}} \geq u^2(0, r/2) +u^2(0,-r/2) +\d^2 -C\e_5.
 \]
Choosing $\e_5$ small enough, this contradicts Lemma \ref{lem:upperdensity}.
\end{proof}

Define
\[
 J(u,r)= \frac{\sqrt{r}}{|u(0,r/2)-u(0,-r/2)|}.
\]
We know from Lemma \ref{lem:lowerdensity} that if $E(u,r)+f(u,r)+r$ is small enough, then $J(u,r)\leq C$. We will see later that this is essentially a lower-order quantity that decays geometrically under mild assumptions. 

In order to obtain more information about $K$, it will be necessary to use perturbations where the part of $Q_r$ above $A^+$ and the part below $A^-$ are deformed separately from one another as competitors. This will enable us to prove an excess-flatness bound involving $ee$ and $ff$, and later on in Section \ref{sec:EL2} to obtain a stronger first variation formula which implies that $g_-$ and $g_+$ both ``almost solve a PDE." However, this kind of competitor is problematic, as \emph{a priori} $K$ need not fully separate the upper and lower halves of $Q_r$ into distinct connected components. Moreover, some mechanism is needed to take care of the part of $\{u>0\}$ in between $A^+$ and $A^-$. The workaround used here, largely contained in the following ``cutting lemma," involves enlarging the set $K$ to a closed, countably $\cH^{n-1}$ rectifiable set $Z$. This set $Z$ fully separates the cylinder $Q_r$ into two connected components, and also encloses the leftover part of $\{u>0\}\cap Q_r$ which lies between $A^+$ and $A^-$ in connected components of $Q_{2r}\sm Z$ which are compact (this facilitates dealing with these leftovers later). 

At the same time, note that as we plan on using this in first variation arguments to build competitors $v$ with $K_v$ a smooth deformation of $Z$, expanding $K$ to $Z$ is dangerous: it introduces \emph{zero order} error in a computation where the key terms we are trying to estimate are \emph{first order}, proportional to the size of the deformation. The redeeming characteristic of $Z$ that will allow our arguments to proceed is that $Z\sm K$ is extremely small, with $\cH^{n-1}$ measure bounded by higher powers of $ff$ and $E$. This is different, and much stronger, than any of our previous estimates playing similar roles (notably the ``holes" estimate \eqref{eq:ldc2} and the Lipschitz parametrization estimates of Theorem \ref{thm:lip}), which have all been linear in $ff$ and $E$. The construction of this $Z$ is based partly on cutting $u$ along several of its level sets in order to fill in the gaps in $K$, and partly on a finite nonlinear iteration argument that allows us to isolate the leftover part of $\{u>0\}\cap Q_r$ with the addition of only a small amount of extra measure. 

\begin{lemma}\label{lem:cut} Let $u\in \sQ(Q_{2r},\L,4r)$. Then there is a universal constant $\e_6$ such that if
\begin{equation}\label{eq:cuth}
 f_\8(u,2r) + E(u,2r)+\L r^s + J(u,2r)+ r\leq \e_6,
\end{equation}
then for any $\xi \in [J(u,2r),\e_6]$ and $p>0$, we may find a closed, countably $\cH^{n-1}$ rectifiable set $Z\cc Q_{2r}$ satisfying the following:
 \begin{enumerate}
  \item $Z$ partitions $Q_r$ into at least two connected components, one of which (called $S_+$) contains $\{x_n\geq r/2\}$ and another (called $S_-$) contains $\{x_n\leq -r/2\}$.
  \item On $S_\pm$, we have that $|u-u(0,\pm r)|\leq C\xi^{-1}\sqrt{r}$.
  \item There is a constant $C(p)$ such that 
  \[
  \cH^{n-1}(Z \sm K)\leq C(p) r^{n-1} [ff^{p}(u,2r) +\xi\sqrt{E(u,r)} +\L r^s]. 
  \]
  \item The union of the components of $Q_{2r}\sm Z$ which contain $Q_r\cap \{u>0\} \sm (S_+ \cup S_-)$ is compactly contained in $Q_{2r}$.
  \item $K\cap Q_r \ss Z$.
 \end{enumerate}
\end{lemma}

The application of this we have in mind is when $J(u,2r) + \L r^s+r$ is smaller than some power of $ff(u,2r)$, and $\xi$ is set to that power of $ff$. In particular, $p$ will always be some fixed value.

\begin{proof}
Let $u_\pm=u(0,\pm r)$; up to a reflection we may take $u_+>u_-$ (as $J$ is small from \eqref{eq:cuth}, $u_+\neq u_-$). Note that $K\cap Q_{3r/2}$ is contained in the region
\begin{equation}\label{eq:cuti1}
 \{|x_n -L_-(x')|\leq C ff^{\frac{1}{n+1}}(u,2r) r\}\cup \{|x_n -L_+(x')|\leq C ff^{\frac{1}{n+1}}(u,2r) r\}
\end{equation}
by Proposition \ref{prop:fflatsame}; here $L_\pm=L_\pm[u,2r]$ are the affine functions parametrizing the optimal planes in the definition of $ff(u,2r)$, and the region might or might not be connected. Let us check that
\begin{equation}\label{eq:cuti2}
 \sup_{Q_{3r/2} \cap \{x_n \geq L_+(x') + C ff^{\frac{1}{n+1}}(u,2r) r\}} |u_+ - u(x)| \leq C_0\sqrt{r},
\end{equation}
for some constant $C_0$. Indeed, this follows from the Campanato embedding \cite{Camp} applied to the region in question, together with the estimate from the energy upper bound:
\[
 \int_{B_t(x)} |\n u|^2 d\cL^n \leq Ct^{n-1},
\]
which is valid for any ball from \eqref{eq:remball}. An analogous estimate holds on $\{x_n\leq L_-(x')-C ff^{\frac{1}{n+1}}(u,2r)r\}$.

We begin by finding a set $Z_0$ which is the union of four level sets of $u$. From the coarea formula \cite[Theorem 2.93]{AFP}, we have that
\[
 \int_{\a}^{\b}\cH^{n-1}(\{u=t\}\cap Q_{2r}\sm K) d\cL^1(t) = \int_{\{x\in Q_{2r}:u(x)\in (\a,\b)\}\sm K}|\n u| d\cL^n
\]
and that the set $\{u=t\}\cap Q_{2r}\sm K$ is countably $\cH^{n-1}$-rectifiable for $\cL^1$-a.e. $t$. Thus we may extract a relatively closed, countably $\cH^{n-1}$-rectifiable level set $\{u=t\}$ with $t\in (\a,\b)$ and
\[
 \frac{\cH^{n-1}(\{u=t\}\cap Q_{2r})}{r^{n-1}}\leq C\frac{\sqrt{r}}{\b-\a} \sqrt{E(u,2r)}.
\]
We do this four times, with the following choices of $\a,\b$:
\begin{center}
\begin{tabular}{c c c}
 $t$ & $\a$ & $\b$ \\ \hline
 $t_1$ & $u_- - \frac{1}{5 \xi} \sqrt{r} $ & $u_- - \frac{1}{10 \xi} \sqrt{r} $ \\
 $t_2$ & $u_- + \frac{1}{10 \xi} \sqrt{r} $ & $u_- + \frac{1}{5 \xi} \sqrt{r} $ \\
 $t_3$ & $u_+ - \frac{1}{5 \xi} \sqrt{r} $ & $u_+ - \frac{1}{10 \xi} \sqrt{r} $ \\
 $t_4$ & $u_+ + \frac{1}{10 \xi} \sqrt{r} $ & $u_+ + \frac{1}{5 \xi} \sqrt{r} $ 
\end{tabular}  
\end{center}
Letting $Z_0$  be the union of these four level sets, we see that
\[
 \frac{\cH^{n-1}(Z_0)}{r^{n-1}} \leq C\xi \sqrt{E(u,2r)}.
\]

Now, $K\cup Z_0$ is relatively closed in $Q_{2r}$ and partitions that cylinder into a countable union of connected components. By choosing $\e_6$ small enough that $\frac{1}{10\xi}$ is larger than $C_0$ in \eqref{eq:cuti2}, we have that $Q_{3r/2} \cap \{x_n\geq L_+(x')+ Cff^{\frac{1}{n+1}}(u,2r)r\}$ is contained in a single component of the open set $Q_{2r}\sm (K\cup Z_0)$, which we denote $S_+$. The component $S_-$ is chosen analogously. Then we have that $Q_{2r}\sm (Z_0\cup K)$ is made up of the components $S_-\ss\{x\in Q_{2r}\sm K: u(x)\in (t_1,t_2)\}$, $S_+\ss\{x\in Q_{2r}\sm K: u(x)\in (t_3,t_4)\}$, the region where $u=0$, and everything else, which we denote by $S_B$:
 \[
  S_B:= Q_{2r}\sm (K\cup Z_0 \cup S_+\cup S_-\cup \{u=0\}) .
 \]
Note that as $\xi\geq J(u,2r)$, we have $t_2<t_3$ and $S_-$ and $S_+$ are distinct. 

We will now estimate the sizes of the intersections $S_B\cap \p Q_{\s r}$, for some choices of $\s\in (1,\frac{3}{2})$. First, let us establish the weak estimate that
\begin{equation}\label{eq:cuti3}
 \frac{\cL^n(S_B \cap Q_{\frac{3r}{2}})}{r^n} \leq C ff^{\frac{1}{n+1}}(u,2r).
\end{equation}
Indeed, if $ff^{\frac{1}{n+1}}(u,2r)\geq \e_5 \frac{h(u,2r)}{r}$, we have that $S_B$ is contained in the set in \eqref{eq:cuti1}. This, in turn, is contained in
\[
 \{(x',x_n): L_-(x')- Cff^{\frac{1}{n+1}}(u,2r)r \leq x_n\leq L_+(x')+ Cff^{\frac{1}{n+1}}(u,2r)r\},
\]
whose volume is bounded by $(h+ Cff^{\frac{1}{n+1}}(u,2r)r)r^{n-1}$. If, instead, $ff^{\frac{1}{n+1}}(u,2r)< \e_5 \frac{h(u,2r)}{r}$, we apply Lemma \ref{lem:zerobetween} to see that the two sets in \eqref{eq:cuti1} are disjoint, and $u=0$ in the region between them. In particular, $S_B$ is contained in their union, which has Lebesgue measure controlled by $C r ff^{\frac{1}{n+1}}(u,2r) r^{n-1}$. This gives \eqref{eq:cuti3}.

Fix $p$, and let $\s_k = 1+2^{-k-1}$. We have just shown in \eqref{eq:cuti3} that if
\[
 V_k:=\frac{\cL^{n}(Q_{\s_k r}\cap S_B)}{r^n}, 
\]
then
\[
  V_0 \leq C ff^{\frac{1}{n+1}}.
\]
We will show that
\begin{equation}\label{eq:cuti4}
 V_{k+1} \leq C^k [V_k + \xi \sqrt{E(u,2r)} + \L r^s]^{\frac{n}{n-1}}.
\end{equation}
Indeed, from the relative isoperimetric inequality on $Q_{\s_{k+1}}$, which has a constant depending only on $n$, we have that
\begin{equation}\label{eq:cuti5}
 V_{k+1} \leq C(n) \1\frac{\cH^{n-1}(\p S_B \cap Q_{\s_{k+1} r})}{r^{n-1}}\2^{\frac{n}{n+1}},
\end{equation}
where we use that $V_{k+1}\leq V_0 \ll 1$. Now choose a value $\s \in (\s_{k+1},\s_k)$ such that
\begin{equation}\label{eq:cuti7}
 \frac{\cH^{n-1}(\p Q_{\s r} \cap S_B)}{r^{n-1}} \leq 4^k V_k, 
\end{equation}
which is always possible from Chebyshev's inequality. We now use $u1_{\R^n \sm (S_B \cap Q_{\s r})}$ as a competitor for $u$. This, along with the lower bound on $u$ from \eqref{eq:dregsize}, leads to the inequality
\begin{align}
 \cH^{n-1}(\p S_B \cap Q_{\s r}) &\leq C\1 \cH^{n-1}(\p Q_{\s r} \cap S_B) + \cH^{n-1}(Z_0) + \L r^s     \2 \nonumber\\
 &\leq C^k [V_k +\xi \sqrt{E(u,2r)}+ \L r^s]. \label{eq:cuti6}
\end{align}
Combining \eqref{eq:cuti5} and \eqref{eq:cuti6} gives the desired inequality \eqref{eq:cuti4}.

Finally, set $k\geq \frac{\log p +\log (n+1)}{\log n - \log (n-1)}$ and iterate the recurrence \eqref{eq:cuti4} $k$ times to obtain
\[
 V_k \leq C(p) [ff^{p}(u,2r) + \xi \sqrt{E(u,2r)} + \L r^s].
\]
Then select $\s$ as in \eqref{eq:cuti7} one final time, to be in $(\s_{k+1},\s_k)$. Let $Z=Z_0 \cup (\p Q_{\s r}\cap S_B) \cup K$; we have that
\[
 \frac{\cH^{n-1}(Z\sm K)}{r^{n-1}} \leq C(p) [ff^{p}(u,2r) + \xi \sqrt{E(u,2r)} + \L r^s],
\]
and so property (3) is satisfied. Properties (1) and (2) follow from the construction of $Z_0$, (5) is automatic, and (4) follows from $S_B \cap \p Q_{\s r} \ss Z$. 
\end{proof}

The following theorem applies this lemma to a two-plane variant of the proof of Theorem \ref{thm:exbyflat}. This will allow us to show that the modified excess $ee$ is bounded in terms of $E,ff$, and some power of $J$. This will be a more useful estimate than Theorem \ref{thm:exbyflat}, as $J$ is a lower-order quantity, while (unlike $f$) $ff$ will be shown to decay geometrically. The assumption that $u$ does not vanish above and below will be typical of statements involving $ff$, as the case where $u$ does vanish will be covered by using just $f$.

The idea of the argument is straightforward: we wish to perform a perturbation like that in the proof of Theorem \ref{thm:exbyflat} to the upper and lower portions of $K$ separately, pushing them toward the upper and lower auxiliary hyperplanes in the definition of $ff$, respectively. In practice, this turns out to be rather involved, and will rely heavily on Lemma \ref{lem:cut}.

\begin{theorem}\label{thm:exbyflat2}Let $u\in \sQ(Q_{4r},\L,8r)$. Assume that neither of $u(0,\pm 2r)=0$. Then there is a universal constant $\e_7>0$ such that if $f_\8(u,4r)\leq \e_7$, we have
\begin{equation}\label{eq:ebf2c1}
 ee(u,r)\leq C M \text{ where }  M:= E(u,4r) + ff(u,4r) + J^{\frac{2}{3}}(u,4r) + r^{\frac{1}{7}} + \sqrt{\L} r^{s/2}.
\end{equation}
Moreover, we have that
\begin{equation}\label{eq:ebf2c2}
 \int_{Q_r\cap A^+}1-(\nu_+ \cdot \nu_x)^2 d\cH^{n-1}+\int_{Q_r\cap A^-}1 - (\nu_- \cdot \nu_x)^2 d\cH^{n-1}\leq CM r^{n-1}
\end{equation}
as well. Here $\nu_\pm$ are the vectors attaining the infimum in the definition of $ff(u,4r)$.
\end{theorem}

\begin{proof} 
We will use $h=h(u,4r)$ to be the optimal number in the definition of $ff(u,4r)$, and $\pi_+ = \{x: (x-he_n)\cdot \nu_+=0\}$, 	$\pi_- = \{x: (x+he_n)\cdot \nu_-=0\}$ be the optimal upper and lower auxiliary planes.
	
We begin by applying Lemma \ref{lem:cut} to $u$ on $Q_{4r}$ with $\xi \geq J(u,4r)$ to be selected and $p=2$, to obtain the sets $Z$, $S_+$,  $S_-$, and $S_B$. Recall that $S_B$ is the union of the connected components of $Q_{4r}\cap \{u>0\}\sm Z$, other than $S_\pm$, which intersect $Q_{2r}$, and $S_B$ is compactly contained in $Q_{4r}$.  We set $u_\pm = u(0,\pm 2r)$.

Let $W_\s = Q_\s \cap [(A^+ \sm \p S_+)\cup (A^- \sm \p S_-)]$, with $W^\pm_\s=Q_\s \cap (A^\pm \sm \p S_\pm)$. Let us consider a point $x\in W_\s$; say $x\in A^+\sm \p S^+$. Then clearly the vertical line segment $\{x+ t e_n:t>0\}$ must intersect $Z\sm K$ (as $x\in A^+$, it cannot intersect $K$, and yet it must intersect $\p S^+$), and so $x'\in \pi(Z\sm K)$. We will estimate the size of this, iteratively, later in the proof. For now, let us just record that
\begin{equation}\label{eq:ebf2i1}
 \cH^{n-1}(W_\s)\leq 2\cH^{n-1}(K\cap \pi^{-1}(\pi(Z\sm K)) \cap Q_\s).
\end{equation}
Note that if $ff^{\frac{1}{n+1}}(u,4r) \leq \e_5 \frac{h}{r}$, it follows from Lemma \ref{lem:zerobetween} that $A^-\cap \p S_+$ is empty, and so we may instead use
\begin{equation}\label{eq:ebf2i2}
\cH^{n-1}(W_\s)\leq 2\cH^{n-1} (K\sm (\p S_+ \cup \p S_-)\cap Q_\s). 
\end{equation}

Let us note that for $\cH^{n-1}$-a.e. point $x$ in $J_u\cap \p S_+\sm W_{2r}$, we have $d_*(x;\nu_+,h)\geq d(x,\pi_+)$, where $d_*$ is the integrand in the definition of $ff$. Indeed, this would only fail if $x\in A^-$ and $\ud(x)=0$, and as $A^-\sm W_{2r} \ss \p S_-$, this means $x\in \p S_-\cap \p S_+$. Up to ignoring a set of $\cH^{n-1}$-measure $0$, we may assume $x\in \p^* S_+\cap \p^* S_-$, the reduced boundaries. This immediately contradicts $\ud(x)=0$, however, as $S_+\cup S_-$ has Lebesgue density $1$ at $x$, and $u\geq \d$ on this set.

Let $\t_\pm$ be the signed distance function to $\pi_\pm$, chosen so that $\n \t_\pm = -\nu_\pm $. Set $T^0_\pm (x) = \t_\pm(x) \nu_\pm$. We have that
\[
 \n T^0_+ e = -(\nu_+\cdot e)\nu_+ , 
\]
from which we may compute
\[
 \dvg\nolimits^{\nu_x} T_+^0 =  (\nu_+ \cdot \nu_-)^2-1,
\]
where $\nu^x$ is an approximate jump vector. 

Let $\eta: Q_{4r}\rightarrow [0,1]$ be a smooth cutoff function which vanishes outside $Q_{3r/2}$ is equal to one on $Q_r$, and has $|\n \eta|\leq 2/r$. Then the vector fields $T_\pm = T^0_\pm \eta^2$ are Lipschitz (indeed, smooth), with Lipschitz constants universally bounded. Let $\phi_\pm(x) = x +tT_\pm(x)$; these are smooth diffeomorphisms which are the identity outside of $Q_{2r}$ for $t$ universally small enough. We intend on building competitors for $u$ by using $u\circ \phi_+^{-1}$ for $x\in \phi_+(S_+)$ and $u\circ \phi_-^{-1}$ for $x\in \phi_-(S_-)$. That this is not actually a well-defined function (i.e. not single-valued) is an obstruction which will be resolved shortly. Note that $\phi_+(S_+)\ss Q_{2r}$ is an open set with $\phi_+(S_+)\triangle S_+ \cc Q_{2r}$, and likewise for $\phi_-(S_-)$.

Our construction will depend on whether
\begin{equation}\label{eq:ebf2i3}
 \cH^{n-1}(\p \phi_+(S_+) \cap \phi_-(S_-)^{(1)})\leq \cH^{n-1}(\p \phi_-(S_-) \cap \phi_+(S_+)^{(1)})
\end{equation}
or not. Recall that the superscripts $E^{(a)}$ refer to the points of Lebesgue density $a$ for $E$.  Define the function $v$ (implicitly depending on $t$) as follows:
\[
 v(x)=\begin{cases}
       u(x) & x \in Q_{4r}\sm (Q_{2r}\cup S_B)\\
       0 & x \in S_B \sm Q_{2r}\\
       u(\phi^{-1}_+(x)) & x\in \phi_+(S_+)\sm \phi_-(S_-)\\
       u(\phi^{-1}_-(x)) & x\in \phi_-(S_-)\sm \phi_+(S_+)\\
       u(\phi^{-1}_+(x)) & x\in \phi_-(S_-)\cap \phi_+(S_+) \text{ if \eqref{eq:ebf2i3} holds} \\
       u(\phi^{-1}_-(x)) & x\in \phi_-(S_-)\cap \phi_+(S_+) \text{ if \eqref{eq:ebf2i3} does not hold}\\
       0 &x\in Q_{2r} \sm (\phi_+(S_+) \cup \phi_-(S_-)).
      \end{cases}
\]
This specifies $v$ for $\cL^{n}$-a.e. $x$, and it is easy to see that $v\in SBV$. We will assume that \eqref{eq:ebf2i3} holds, for if this is not the case, we may use the reflection of $u$ across $\pi$ in place of $u$. Note that $v - u1_{\R^n\sm S_B}$ is compactly supported in $Q_{2r}$, and so $v$ is an admissible competitor for $u$ in $Q_{4r}$. We will now estimate the energy and surface terms of $v$.

For the energy terms, we proceed as in Lemma \ref{lem:EL}, using a cruder estimate on $\n \phi_\pm$:
\begin{align}
 \int_{Q_{2r}}|\n v|^2 d\cL^n &= \int 1_{S_+}| \n \phi_+^{-1}\n u|^2 |\det \n \phi_+| + 1_{\phi_-^{-1}(\phi_-(S_-)\sm \phi_+(S_+))}| \n \phi_-^{-1}\n u|^2 |\det \n \phi_-| d\cL^n \nonumber \\
 &\leq \int_{S_+\cup S_-} |\n u|^2 (1+Ct(|\n T_+|+|\n T_-|))d\cL^n \nonumber\\
 &\leq \int_{Q_{2r}} |\n u|^2 d\cL^n + Ct E(u,4r)r^{n-1}.\label{eq:ebf2i4}
\end{align}
Here the last step used the Lipschitz bounds on $T_\pm$.

Now for the surface terms. As $\phi_+(S_+),\phi_-(S_-)\sm \phi_+(S_+)$ are sets whose boundaries are countably $\cH^{n-1}$-rectifiable, and $v$ vanishes in the complement of their union,  we have that
\begin{align}
  \label{eq:exbyflat1}
 \m_v(Q_{2r}) &= \m_v(\phi_+(S_+)^{(1)}) + \m_v([\phi_-(S_-)\sm \phi_+(S_+)]^{(1)}) \\
 &\quad+ \int_{\p^* \phi_+(S_+) \cap Q_{2r}} \overline{ v 1_{\phi_+(S_+)}}^2d\cH^{n-1} \nonumber\\  
  \nonumber
 &\quad+ \int_{\p^* (\phi_-(S_-)\sm \phi_+(S_+)) \cap Q_{2r}} \overline{ v 1_{\phi_-(S_-)\sm \phi_+(S_+)}}^2d\cH^{n-1}.
\end{align}
We will be using the fact that smooth diffeomorphisms preserve the Lebesgue density of a set at each point.  The first contribution may be estimated exactly as in Lemma \ref{lem:EL}, to give
\[
 |\m_v(\phi_+(S_+)^{(1)}) - \m_u(S_+^{(1)}) - t\int_{S_+^{(1)}\cap J_u}\dvg\nolimits^{\nu_y} T_+ d\m_u(y)|\leq Ct^2 r^{n-1}.
\]
Let us at this point compute this tangential divergence:
\begin{align}
  \dvg\nolimits^{\nu_y} T_+ &= - \eta^2(y) (1-(\nu_+\cdot \nu_y)^2) + 2\t_+(y)\eta(y) (\n \eta \cdot \nu_+ - (\n \eta \cdot \nu_y)(\nu_+\cdot \nu_y) ) \nonumber\\
  &= - \eta^2(y) (1-(\nu_+\cdot \nu_y)^2) + 2\t_+(y)\eta(y) \n \eta \cdot (\nu_+ - \nu_y(\nu_+\cdot \nu_y) )\nonumber\\
  &\leq - \eta^2 \frac{1}{2}(1-(\nu_+\cdot \nu_y)^2) + C|\n \eta|^2 \t_+^2 \label{eq:ebf2i5}.
\end{align}
The last step used the Cauchy inequality and that 
\[
1-(\nu_+\cdot \nu_y)^2 = |\nu_+ - \nu_y(\nu_+\cdot \nu_y)|^2.
\]
Thus we have that
\begin{align}
 \m_v(\phi_+(S_+)^{(1)})&\leq \m_u(S_+^{(1)}) - \frac{t\d^2}{4}\int_{J_u\cap S_+^{(1)}} \eta^2(1-(\nu_+\cdot \nu_y)^2) d\cH^{n-1} \nonumber\\
&\qquad  + \frac{Ct}{r^2}\int_{J_u\cap S_+^{(1)}} \t_+^2(y)d\cH^{n-1}(y) + Ct^2 r^{n-1}\label{eq:ebf2i6}.
\end{align}
Note that $\t_+^2(y) = d^2(y,\pi_+)\leq d_*^2(y;\nu_+,h)$ outside of $W_{2r}$, as observed at the beginning of the proof, so we have 
\begin{align}
 \m_v(\phi_+(S_+)^{(1)})&\leq \m_u(S_+^{(1)}) - \frac{t\d^2}{4}\int_{J_u\cap S_+^{(1)}} \eta^2(1-(\nu_+\cdot \nu_y)^2) d\cH^{n-1}\nonumber\\
&\qquad  + \frac{Ct}{r^2}\int_{J_u\cap S_+^{(1)}} d_*^2(y;\nu_+,h)d\cH^{n-1}(y) + Ct^2 r^{n-1} + C t \1\frac{h}{r}\2^2 \cH^{n-1}(W_{2r}) \label{eq:ebf2i7}.
\end{align}
We used that at any point of $K$ (and in particular at those in $W_{2r}$), we have $\t_+^2 \leq C(d_*^2 +h^2)$.

The same argument may be applied to the second term in \eqref{eq:exbyflat1} to give
\begin{align}
 \m_v&([\phi_-(S_-)\sm\phi_+(S_+)]^{(1)})\leq \m_u([S_+\sm \phi_-^{-1}(\phi_+(S_+))]^{(1)}) \nonumber\\
 &\qquad- \frac{t \d^2}{4}\int_{J_u\cap [S_+\sm \phi_-^{-1}(\phi_+(S_+))]^{(1)}}\eta^2 (1-(\nu_-\cdot \nu_y)^2) d\cH^{n-1} \nonumber\\
&\qquad  + \frac{Ct}{r^2}\int_{J_u\cap [S_+\sm \phi_-^{-1}(\phi_+(S_+))]^{(1)}} d_*^2(y;\nu_+,h)d\cH^{n-1}(y) + Ct^2 r^{n-1} + C t \1\frac{h}{r}\2^2 \cH^{n-1}(W_{2r}). \label{eq:ebf2i8}
\end{align}
Here again $\t_-\leq d_*$ away from $W_{2r}$. 

We proceed to estimate the third term in \eqref{eq:exbyflat1}; this works in the same way, except on the right the integrals over $J_u$ are now over $\p^* S_+$ which may contain points on $Z\sm K$. Rather than include them in the integration, we estimate them by the size of $Z\sm K$.
\begin{align}
 \int_{\p^* \phi_+(S_+) \cap Q_{2r}} &\overline{ v 1_{\phi_+(S_+)}}^2 d\cH^{n-1}\leq \int_{\p^* S_+ \cap Q_{2r}} \overline{ u 1_{S_+}}^2 - \frac{t \d^2}{4}\eta^2(1-(\nu_+\cdot \nu_y)^2) d\cH^{n-1} \nonumber\\
&\qquad  + \frac{Ct}{r^2}\int_{\p^* S_+ \cap Q_{2r}} \t_+^2(y)d\cH^{n-1}(y) + Ct^2 r^{n-1}\nonumber\\
&\leq \int_{\p^* S_+ \cap J_u \cap Q_{2r}} \overline{ u 1_{S_+}}^2 - \frac{t\d^2}{4}\eta^2(1-(\nu_+\cdot \nu_y)^2) d\cH^{n-1} \nonumber\\
&\qquad  + \frac{Ct}{r^2}\int_{\p^* S_+ \cap J_u \cap Q_{2r}} d_*^2(y;\nu_+,h)d\cH^{n-1}(y) + Ct^2 r^{n-1}\nonumber\\
&\qquad + C\cH^{n-1}(Z\sm K) +C t \1\frac{h}{r}\2^2 \cH^{n-1}(W_{2r}) \label{eq:ebf2i9}.
\end{align}
Note carefully that the first term on the last line is being used to estimate the size of $\p^* S_+ \sm J_u$ in all of the integrals (including the very first, zero-order one), and so appears without any multiple of $t$.

For the fourth and final integral in \eqref{eq:exbyflat1}, we split the domain $\p^* (\phi_{-}(S_-)\sm \phi_+(S_+))$ into two parts: the first is $R_1=\p^* \phi_+(S_+)\cap \phi_-(S_-)^{(1)}$, while the second is $R_2 = \p^* \phi_-(S_-)\sm \phi_+(S_+)^{(1)}$. Indeed, the union of these two disjoint countably $\cH^{n-1}$-rectifiable sets contains $\p^* (\phi_-(S_-) \sm \phi_+(S_+))$ up to an $\cH^{n-1}$-negligible set. For $R_2$, we proceed exactly as above:
\begin{align}
 \int_{R_2 \cap Q_{2r}} \overline{ v 1_{\phi_-(S_-)}}^2 d\cH^{n-1}&\leq \int_{\phi_-^{-1}(R_2) \cap Q_{2r}} \overline{ u 1_{S_-}}^2 - \frac{t \d^2}{4}\eta^2(1-(\nu_-\cdot \nu_y)^2) d\cH^{n-1} \nonumber\\
&\qquad  + \frac{Ct}{r^2}\int_{\phi_-^{-1}(R_2) \cap Q_{2r}} \t_-^2(y)d\cH^{n-1}(y) + Ct^2 r^{n-1}.\label{eq:ebf2i10}
\end{align}
For the integral over $R_1$, note that by assumption \eqref{eq:ebf2i3} we have that 
\[
\cH^{n-1}(R_1 \cap Q_{2r})\leq \cH^{n-1}(R_3 \cap Q_{2r}), 
\]
where $R_3:=\phi_+(S_+)^{(1)}\cap\p^* \phi_-(S_-)$, and by definition $R_3\cap R_2$ is empty. Moreover, as $v\in (u_- - C\frac{\sqrt{r}}{\xi},u_- + + C\frac{\sqrt{r}}{\xi})$ on $\phi_-(S_-)$, we have that
\[
 \int_{R_1 \cap Q_{2r}} \overline{ v 1_{\phi_-(S_-)}}^2 d\cH^{n-1} \leq \int_{R_3 \cap Q_{2r}} \overline{ v 1_{\phi_-(S_-)}}^2 d\cH^{n-1} + C\frac{r^{n-\frac{1}{2}}}{\xi}.
\]
Now we argue as before to obtain
\begin{align}
 \int_{R_3 \cap Q_{2r}} \overline{ v 1_{\phi_-(S_-)}}^2 d\cH^{n-1}&\leq \int_{\phi_-^{-1}(R_3) \cap Q_{2r}} \overline{ u 1_{S_-}}^2 - \frac{ t \d^2}{4}\eta^2(1-(\nu_-\cdot \nu_y)^2) d\cH^{n-1} \nonumber\\
&\qquad  + \frac{Ct}{r^2} \int_{\phi_-^{-1}(R_3) \cap Q_{2r}} \t_-^2(y)d\cH^{n-1}(y) + Ct^2 r^{n-1}. \label{eq:ebf2i11}
\end{align}
Using that $\phi^{-1}_-(R_2)$ and $\phi_-^{-1}(R_3)$ cover $Q_{2r}\cap \p^* S_-$, we may add \eqref{eq:ebf2i10} and \eqref{eq:ebf2i11} to get
\begin{align}
 \int_{\p^* (\phi_-(S_-)\sm \phi_+(S_+)) \cap Q_{2r}}& \overline{ v 1_{\phi_-(S_-)}}^2 d\cH^{n-1}\leq \int_{\p^* S_-\cap J_u \cap Q_{2r}} \overline{ u 1_{S_-}}^2 - \frac{t \d^2}{4}\eta^2(1-(\nu_-\cdot \nu_y)^2) d\cH^{n-1} \nonumber \\
&\qquad  + \frac{Ct}{r^2}\int_{\p^* S_-\cap J_u \cap Q_{2r}} d_*^2(y;\nu_+,h)d\cH^{n-1}(y) + C[t^2 + \frac{\sqrt{r}}{\xi}] r^{n-1} \nonumber\\
&+ C\cH^{n-1}(Z\sm K)+C t \1\frac{h}{r}\2^2 \cH^{n-1}(W_{2r}).\label{eq:ebf2i12}
\end{align}

Putting together the estimates \eqref{eq:ebf2i7}, \eqref{eq:ebf2i8}, \eqref{eq:ebf2i9}, and \eqref{eq:ebf2i12} on all four terms of \eqref{eq:exbyflat1}, we obtain that
\begin{align}
 \m_v(Q_{2r})&\leq \m_{u}(\bar{S}_+\cup \bar{S}_-) - \frac{t\d^2}{4}\1\int_{J_u \cap \bar{S}_+}\eta^2(1-(\nu_+\cdot \nu_y)^2) d\cH^{n-1} + \int_{J_u \cap \bar{S}_-}\eta^2 (1-(\nu_-\cdot \nu_y)^2) d\cH^{n-1}\2\nonumber\\
 &\qquad + \frac{Ct}{r^2}\int_{J_u\cap (\bar{S}_+ \cup \bar{S}_-)}  d_*^2(y;\nu_+,h) d\cH^{n-1}\nonumber\\
 &\qquad + C[t^2 + \frac{\sqrt{r}}{\xi}] r^{n-1} + C\cH^{n-1}(Z\sm K)+C t \1\frac{h}{r}\2^2 \cH^{n-1}(W_{2r}).\label{eq:ebf2i13}
\end{align}
The term on the second line is controlled by $ff(u,4r)$. We now use $v$ as a competitor for $u$. Using \eqref{eq:ebf2i4} and \eqref{eq:ebf2i13}, we obtain that
\begin{align}
 \m_u&(Q_{2r}\sm (\bar{S}_+\cup \bar{S}_-)) + \frac{t\d^2}{4}\1\int_{J_u \cap \bar{S}_+}\eta^2(1-(\nu_+\cdot \nu_x)^2) d\cH^{n-1} + \int_{J_u \cap \bar{S}_-}\eta^2 (1-(\nu_-\cdot \nu_x)^2)  d\cH^{n-1}\2 \nonumber\\
&\leq C[t^2 + t ff(u,4r) + \frac{\sqrt{r}}{\xi}+t E(u,4r)] r^{n-1} \nonumber\\
&\quad + C\cH^{n-1}(Z\sm K) + C\L r^s +C t \1\frac{h}{r}\2^2 \cH^{n-1}(W_{2r}).\label{eq:ebf2i14}
\end{align}

Let us first assume that \eqref{eq:ebf2i2} applies, in which case we may reabsorb the final term in the left into the first term on the right, if $\frac{h^2}{r^2}\leq e_7^2$ is small enough. Note that now the left-hand side controls $ ee(u,r)t$. We choose $\xi = J(u,4r) + r^{\frac{3}{14}}$, which leads to
\begin{equation}\label{eq:exbyflat2}
 ee(u,r)\leq C[t + ff(u,4r) + E(u,4r) + \frac{r^{\frac{2}{7}}+\L r^s}{t}] + C\frac{\cH^{n-1}(Z\sm K)}{tr^{n-1}}.
\end{equation}
We estimate
\begin{align*}
 \frac{\cH^{n-1}(Z\sm K)}{r^{n-1}} &\leq C[ff^2 (u,4r) + \xi \sqrt{E(u,4r)}+\L r^s]\\
 & \leq C[ff^2 (u,4r) + \xi^{\frac{4}{3}} + E^2(u,4r)+\L r^s]\\
 & \leq C[ff^2 (u,4r) + E^2(u,4r) + J^{4/3}(u,4r) + r^{\frac{2}{7}}+\L r^s]\\
 & \leq CM^2.
\end{align*}
Now choose $t= M$ to obtain that
\begin{equation}\label{eq:ebf2i15}
 ee(u,r) \leq CM.
\end{equation}
This implies the conclusion \eqref{eq:ebf2c1}.

Now assume that \eqref{eq:ebf2i2} does not apply, so $ff^{\frac{1}{n+1}}(u,4r) > \e_5 \frac{h}{r}$. In this case, we cannot reabsorb the final term in \eqref{eq:ebf2i14}, and must content ourselves with just applying \eqref{eq:ebf2i1} instead. Proceeding as above, this results in
\begin{align*}
ee(u,r)&\leq CM +C\frac{h^2}{r^{n+1}}\cH^{n-1}(W'_{2r})\\
& \leq CM +C ff^{\frac{2}{n+1}} \frac{\cH^{n-1}(W'_{2r})}{r^{n-1}},
\end{align*}
where $W'_\s = K\cap Q_{\s}\cap \pi^{-1}(\pi(Z\sm K))$, in place of \eqref{eq:ebf2i15}. We show how to remove this final term with the aid of Theorem \ref{thm:lip}.

Indeed, apply that theorem on $Q_r$ with $\nu_*=\nu_+$ and $h_*=h$. We have just shown that
\begin{equation}\label{eq:ebf2i17}
 ee(u,r)\leq CM +C\frac{h^2}{r^{n+1}}\cH^{n-1}(W'_{2r}),
\end{equation}
from which we deduce that
\[
 \frac{\cH^{n-1}(Q_{r/6}\cap [(\G^+\cup \G^-)\triangle K)}{r^{n-1}}\leq CM + Cff^{\frac{2}{n+1}} \frac{\cH^{n-1}(W'_{2r})}{r^{n-1}}.
\]
Now, clearly $\cH^{n-1}(W'_{2r})\leq Cr^{n-1}$. On the other hand, $W_{r/6}'\sm (\G^+\cup\G^-)$ is contained in the set just estimated, while
\[
 \frac{\cH^{n-1}(W'_{r/6}\cap (\G^+\cup\G^-))}{r^{n-1}} \leq 4 \cL^{n-1}(D_{r/6}\cap \pi(Z\sm K))\leq CM;
\]
this gives that
\[
 \frac{\cH^{n-1}(W'_{r/6})}{r^{n-1}}\leq CM + ff^{\frac{2}{n+1}}(u,4r).
\]

We now carry out the excess bound construction repeatedly, on smaller cylinders, to obtain for any $k\in [0,\lceil\frac{n+1}{2}\rceil]$ that
\[
 ee(u,12^{-k}r)\leq CM +Cff^{\frac{2}{n+1}} \frac{\cH^{n-1}(W'_{2\cdot 12^{-k}r})}{r^{n-1}}.
\]
Notice that we do not find a new set $Z$ each time, but rather use the same one in every step. Applying the same argument with the Lipschitz approximation theorem iteratively gives
\[
 \cH^{n-1}(W'_{2\cdot 12^{-k}r})\leq CM + ff^{\frac{2k}{n+1}}(u,4r),
\]
and so after $\lceil\frac{n+1}{2}\rceil$ iterations, we obtain that
\[
 ee(u,12^{-k}r) + \cH^{n-1}(W_{12^{-k}})\leq CM.
\]
Now an elementary covering and scaling argument gives that for some larger constant,
\begin{equation}\label{eq:ebf2i16}
 ee(u,r) + \cH^{n-1}(W_{2r})\leq CM.
\end{equation}
This implies \eqref{eq:ebf2c1}

Finally, the other conclusion \eqref{eq:ebf2c2} of the theorem follows as well, as \eqref{eq:exbyflat2} is still valid with the integrals
\[
 \int_{Q_r\cap A^+}1-(\nu_+ \cdot \nu_x)^2 d\cH^{n-1}+\int_{Q_r\cap A^-}1 - (\nu_- \cdot \nu_x)^2 d\cH^{n-1}
\]
 in place of $ee$ on the left. Arguing in the same way as to get \eqref{eq:ebf2i15} gives the conclusion if \eqref{eq:ebf2i2} holds; if it does not hold, we instead get
\[
 \int_{Q_r\cap A^+}1-(\nu_+ \cdot \nu_x)^2 d\cH^{n-1}+\int_{Q_r\cap A^-}1 - (\nu_- \cdot \nu_x)^2 d\cH^{n-1} \leq CMr^{n-1}+C\frac{h^2}{r^{2}}\cH^{n-1}(W'_{2r})
\]
as in \eqref{eq:ebf2i17}, and apply \eqref{eq:ebf2i16} to conclude.
\end{proof}

\section{Estimates on the Graphs}\label{sec:graphs}

It will be used later that our various scale-invariant quantities may be rewritten in terms of $g_{\pm}$. In this section, $L_\pm=L_\pm [u,0,24r,e_n]$, which we recall are the affine maps $\pi\rightarrow \R$ parametrizing $\pi_\pm$ over $\pi$, where $\pi_\pm,\nu_\pm , h$ are the ones optimal for $f(u,24r)$.

First, we prove a lemma which is the equivalent of Proposition \ref{prop:fflatsame} for the Lipschitz graphs $\G^\pm$ of Theorem \ref{thm:lip}.

\begin{lemma}\label{lem:gflat} Let $u\in \sQ(Q_{24r},\L,48r)$. Assume that $u$ satisfies the hypotheses in Theorem \ref{thm:lip} on $Q_{6r}$, with $\nu_*=\nu_+$, and that neither of $u(0,\pm r)=0$. Then there is an $\e'_L>0$ such that if 
	\begin{equation}\label{eq:gflath1}
	M:= ff(u,24r) + E(u,24r) + J^{\frac{2}{3}}(u,24r)+\sqrt{\L r^s} + r^{\frac{1}{7}}\leq \e'_L, 
	\end{equation}
	then
	\[
	| L_\pm - g_\pm| \leq C M^{\frac{1}{n+1}} r
	\]
	on $Q_r$, where $g_\pm$ as in the conclusion of Theorem \ref{thm:lip}.
\end{lemma}

\begin{proof} First, the quantity $M$ in \eqref{eq:gflath1} controls the $M$ in Theorem \ref{thm:lip}, from Theorem \ref{thm:exbyflat2}.
	
Let $U^+_s =\{x:|L_{+}(x') - x_n|\leq s\}$, and similarly for $U^-_s$. From Proposition \ref{prop:fflatsame}, we have that $K\cap Q_r$ is fully contained in $U^+_{C_0 M^{\frac{1}{n+1}} r} \cup U^+_{C_0 M^{\frac{1}{n+1}} r}$. From the definition of $d_*$, the integrand in $ff$, and Remark \ref{rem:lipest}, we may deduce that $\G^+ \cap U^+_{C_0 M^{\frac{1}{n+1}} r}$ is nonempty.
	
There are now two cases: either $U^+_{2 C_0 M^{\frac{1}{n+1}} r}\cap \pi$ is empty, or not. If it is empty, then take any $x\in \G^+\sm U^+_{2C_0 M^{\frac{1}{n+1}} r}$. If $x_n<0$, we may then find another $x$ in this set with $x_n = 0$ (simply because $g_+$ is continuous, has a point $z'$ with $(z, g_+(z'))\in U^+_{C_0 M^{\frac{1}{n+1}} r}$, so with $g_+(z')>0$, and $g_+(x')<0$). Now consider the ball $B = B_{\frac{1}{2}C_0 M^{\frac{1}{n+1}} r} (x)$: by construction, we have that $B \cap (U^+_{C_0 M^{\frac{1}{n+1}} r} \cup U^-_{C_0 M^{\frac{1}{n+1}} r})$ is empty, so $B \cap Q_r \cap K$ is empty. On the other hand, from the fact that $g_+$ is Lipschitz-1, we have
\begin{equation}\label{eq:gflati1}
c M^{\frac{n-1}{n+1}}r^{n-1} \leq \cH^{n-1}(\G^+ \cap B \cap Q_r) \leq \cH^{n-1}(Q_r \cap \G^+ \sm K) \leq C M r^{n-1}.
\end{equation}
Here the last step used (3) in Theorem \ref{thm:lip}. For a sufficiently small $M$, this is a contradiction.

If instead $U^+_{2 C_0 M^{\frac{1}{n+1}} r}\cap \pi$ is nonempty, then $U^+_{8 C_0 M^{\frac{1}{n+1}} r}$ contains $U^-_{C_0 M^{\frac{1}{n+1}} r}$, and hence all of $K\cap Q_r$. Choose any $x$ in $\G\sm U^+_{10 C_0 M^{\frac{1}{n+1}} r}$, and define $B$ as before: we have that $B\cap Q_r \cap K$ is empty, so \eqref{eq:gflati1} leads to the same contradiction.
\end{proof}

\begin{proposition}\label{prop:liptoplane} Let $u\in \sQ(Q_{24r},\L,48r)$. Let $u$ satisfy the assumptions of Theorem \ref{thm:lip} (on $Q_{6r}$, with $\nu_*=\nu_+$)  and that neither of $u(0,\pm r)=0$. If $g_{\pm}$ are the Lipschitz functions produced there and
\[
M:= ee(u,24r) + E(u,24r) + J^{\frac{2}{3}}(u,24r)+\sqrt{\L r^s} + r^{\frac{1}{7}}\leq \e'_L, 
\]
then the following estimates hold:
\begin{enumerate}
 \item For a universal constant $C$,
\begin{equation}\label{eq:liptoplanec1}
  \frac{1}{r^{n-1}} \int_{D_r}|\n g_{+} - \n L_+|^2 + |\n g_{-} - \n L_-|^2 d\cL^{n-1}\leq C M.
\end{equation}
\item The following holds unless $\n L_+ = \n L_-$:
\begin{equation}\label{eq:liptoplanec4}
   \cL^{n-1}(D_r \cap \{g_+ \leq 0\}) + \cL^{n-1}(D_r \cap \{g_- \geq 0\}) \leq C\frac{[M + ff(u,24 r)] r^{n-1}}{|\n L_+- \n L_-|^2}.
\end{equation}
\item We have
\begin{equation}\label{eq:liptoplanec2}
 \frac{1}{r^{n+1}}\int_{D_{r}} |g_{+}-L_+|^2 + |g_- -L_-|^2 d\cL^{n-1} \leq C[ff(u,24r) + M^{1+\frac{2}{n+1}}].
\end{equation}
\end{enumerate}
If one of $u(0,\pm r)=0$, then the same holds with $\nu_*=e_n$, $L_\pm$ replaced by $0$, and $ff$ by $f$.
\end{proposition}

\begin{proof}
Let us assume that $u(x,\pm r)\neq 0$; if one of them does equal $0$ the same proof applies by using Theorem \ref{thm:exbyflat} in place of Theorem \ref{thm:exbyflat2}. We begin with item $(3)$.

Note that we always have $|L_+ - g_+|\leq C[M + ff(u,24 r)]^{\frac{1}{n+1}}r$ on $Q_r$, by Lemma \ref{lem:gflat}.

We then have the following estimate, simply by changing coordinates and using the fact that $g_+$ is Lipschitz-$1$:
\begin{align*}
 \frac{1}{r^{n+1}}&\int_{D_r} |g_+ - L_+|^2 d\cL^{n-1}\\
 &\leq \frac{C}{r^{n+1}}\1 \int_{\Theta^+ \cap Q_r}d^2(x,\pi_+) d\cH^{n-1} + [M+ff(u,24 r)]^{\frac{2}{n+1}}r^2 \cH^{n-1}(Q_r \cap \G^+ \sm \Theta^+)\2,
\end{align*}
where $\Theta^+$ is the set of Remark \ref{rem:lipest}. Applying that remark, as well as the fact that on $A^+$ we always have $d^2(x,\pi_+ )\leq d^2_*(u;\nu_+,h)$, we obtain that
\begin{align*}
 \frac{1}{r^{n+1}}\int_{D_r} |g_+ - L_+|^2 d\cL^{n-1}
 &\leq C\1 ff (u,24r) + M [M + ff(u,24 r)]^{\frac{2}{n+1}}\2\\
 &\leq C\1 ff (u,24r) + M^{1+\frac{2}{n+1}}\2.
\end{align*}

We proceed similarly for (1); we claim that
\begin{align*}
 \frac{1}{r^{n-1}}&\int_{D_r} |\n g_+ - \n L_+|^2 d\cL^{n-1}
 &\leq \frac{C}{r^{n-1}}\1 \int_{\Theta^+ \cap Q_r} 1-(\nu_+\cdot \nu_x)^2 d\cH^{n-1} + \cH^{n-1}(Q_r \cap \G^+ \sm \Theta^+)\2.
\end{align*}
Here $\nu_x$ is the unit normal to $\G^+$, which coincides with the approximate jump vector at $\cH^{n-1}$-a.e. point in $\G^+\cap J_u$. Let us check this change-of-variables formula. We are using that
\begin{align*}
 \int_{\G^+\cap Q_r} 1-(\nu_+\cdot \nu_x)^2 d\cH^{n-1} &= \int_{D_r}\sqrt{1+|\n g_+|^2}\left[1 - \1\frac{(-\n g_+,1)}{\sqrt{1+|\n g_+|^2}}\cdot \frac{(-\n L_+,1)}{\sqrt{1+|\n L_+|^2}} \2^2\right] d\cL^{n-1}\\
 &=\int_{D_r} \frac{|\n g_+ - \n L_+|^2 + |\n g_+|^2 |\n L_+|^2 - |\n g_+\cdot \n L_+|^2}{\sqrt{1+|\n g_+|^2}(1+|\n L_+|^2)} d\cL^{n-1}.
\end{align*}
Here in the first step, we changed variables (incurring a Jacobian factor), while the second step combined terms in the numerator. Notice that the denominator is always between $1$ and $4$, as both $g_+$ and $L_+$ are Lipschitz-$1$. The term $|\n g_+|^2 |\n L_+|^2 - |\n g_+ \cdot \n L_+|^2$ is nonnegative, and so we obtain
\[
 \int_{\G^+\cap Q_r} 1-(\nu_+\cdot \nu_x)^2 d\cH^{n-1}\geq \frac{1}{4}\int_{D_r}|\n g_+ - \n L_+|^2 d\cL^{n-1}.
\]
We note that it is elementary to check that the reverse inequality is equally valid, with a different constant.

We now apply Theorem \ref{thm:exbyflat2} to obtain
\[
 \frac{1}{r^{n-1}}\int_{D_r} |\n g_+ - \n L_+|^2 d\cL^{n-1}\leq CM.
\]

Finally, the inequality in (2) follows by applying (3):
\[
 \cL^{n-1}(D_r \cap \{g_+ < L_+ - h/10\})h^2 \leq  C\int_{D_r} |g_+ - L_+|^2 d\cL^{n-1}\leq C[M + ff(u,24 r)] r^{n+1}.
\]
As $|\n L_+ -\n L_-|\leq h/r$, this is a stronger estimate.
\end{proof}

\section{Estimating the Energy}\label{sec:energy}

This section mirrors the corresponding one in the flat-implies-smooth theorem of \cite{AFP}. We recall the following minimal surface estimate, which may be found in \cite[Lemma 8.16]{AFP}.

\begin{lemma}\label{lem:minareaest}
 Let $r>0$, and assume that $g$ is a $1/2$-Lipschitz function defined on $D_{r}$. Then there exists a $3/4$-Lipschitz function $w: D_r\rightarrow \R$ which has $w=g$ on $\p D_r$ and if
 \[
 \cA(w):=\int_{D_r}\sqrt{1+|\n w|^2} d\cL^{n-1},
 \]
then $w$ admits the following estimate, for every $\a>1$ provided that $\a \leq \frac{r}{4 \sup{|g|}}$:
\begin{equation}\label{eq:maec1}
 \cA(w)-\w_{n-1}r^{n-1} \leq C\frac{r}{\a}\int_{\p D_{r}}\frac{\a^2}{r^2} g^2 +|\n g|^2 d\cL^{n-2}.
\end{equation}
The constant is independent of $r,\a,$ or $g$. Furthermore, for any $\b \leq \sup |g|$, we have
\begin{equation}\label{eq:maec2}
\cL^{n-1}(\{|w|>\b\}) \leq \frac{r}{\b^2} \int_{\p D_r} g^2 d\cL^{n-2}.
\end{equation}
\end{lemma}

\begin{proof}
Define $w(x) = \1\frac{|x|}{r}\2^\a g(\frac{r x}{|x|})$. Then
\[
|\n w| \leq \frac{\a}{r} \1\frac{|x|}{r}\2^{\a - 1} |g(\frac{r x}{|x|})| + \1\frac{|x|}{r}\2^{\a - 1} |\n g(\frac{r x}{|x|})| \leq \frac{1}{4} + \frac{1}{2}\leq \frac{3}{4}, 
\]
so $w$ is a $\frac{3}{4}$-Lipschitz function satisfying $w=g$ on $\p D_r$. To check \eqref{eq:maec1}, we use
\begin{align*}
\cA(w) &\leq \int_{D_r}1+\frac{1}{2}|\n w|^2 d\cL^{n-1}\\
& \leq \w_{n-1}r^{n-1} + \int_{D_r} \1\frac{|x|}{r}\2^{2\a - 2}\1 \frac{\a^2}{r^2}g^2(\frac{r x}{|x|}) + |\n g(\frac{r x}{|x|})|^2 \2 d\cL^{n-1}\\
& \leq \w_{n-1}r^{n-1} + \frac{r}{2\a+n-3}\int_{\p D_{r}}\frac{\a^2}{r^2} g^2 +|\n g|^2 d\cL^{n-2}.
\end{align*}
For \eqref{eq:maec2}, we have that for any $x$, if $|w(x)|>\b$, then $|g(\frac{r x}{|x|})|>\b$ as well. It follows that
\[
\cL^{n-1}(\{|w|>\b\}) \leq r \cL^{n-2}(\p D_r \cap \{|g|>\b\}) \leq \frac{r}{\b^2}\int_{\p D_r}|g|^2 d\cL^{n-1}, 
\]
with the last step by Chebyshev's inequality.
\end{proof}

The next theorem gives a geometric decay of the energy, provided the energy was substantially larger than our other normalized quantities (the flatness, the quasiminimality factor $\L r^s$, the jump $J$, and $r$ itself).

\begin{theorem}\label{thm:enimp} Let $u\in \sQ(Q_{2r},\L,4r)$, and neither of $u(0,\pm r)=0$. Then for every $\t\in(0,\t_0)$, there is a number $\e'_E(\t)$ such that if
\begin{equation}\label{eq:enimph1}
 E(u,2r) + f(u,2r) + \L r^s + r\leq \e'_E,
\end{equation}
\begin{equation}\label{eq:enimph2}
\mu(Q_{2r})=0
\end{equation}
and
\begin{equation}\label{eq:enimph3}
 ff(u,\t r) +\sqrt{\L r^s} + r^{\frac{1}{7}} + J^{\frac{2}{3}}(u,2r) \leq \e'_E E(u,2r) 
\end{equation}
then
\begin{equation}\label{eq:enimpc1}
 E(u,\t r/24) \leq C'_E\t E(u, r).
\end{equation}
The constant $C'_E$ is universal. If one of $u(0,\pm r)=0$, the same holds with $ff$ replaced by $f$.
\end{theorem}

Before proving the theorem, we present an easy corollary that removes the extra assumption \eqref{eq:enimph2} and re-expresses the hypotheses and conclusion in a way which will be more useful later.

\begin{corollary}\label{cor:enimp} Let $u\in \sQ(Q_{3r},\L,4r)$, and neither of $u(0,\pm r)=0$. Then for every $\t\in(0,\t_0)$, there is a number $\e_E(\t)$ such that if
	\begin{equation}\label{eq:enimp2h1}
	E(u,3r) + f(u,3r) + \L r^s + r\leq \e_E
	\end{equation}
	and
	\begin{equation}\label{eq:enimp2h2}
	\inf_{|e_n-\nu|\leq \e_E,|x|\leq \e_E}ff(u,x,\t r,\nu) +\sqrt{\L r^s} + r^{\frac{1}{7}} + J^{\frac{2}{3}}(u,2r) \leq \e_E E(u,2r) 
	\end{equation}
	then
	\begin{equation}\label{eq:enimp2c1}
	E(u,\t r/50) \leq C_E\t E(u, 2r).
	\end{equation}
	The constant $C_E$ is universal. If one of $u(0,\pm r)=0$, the same holds with $ff$ replaced by $f$.
\end{corollary}

\begin{proof} Fix $\nu, x$ which attain the infimum in \eqref{eq:enimp2h2}. We wish to apply Theorem \ref{thm:enimp} to the cylinder $Q_{s r,\nu}(x+rt\nu)$ with the value used as $\t$ in there set to $\t' = \frac{1}{s} \t - t>0$, where $s>1$ is close to $1$ and $t$ is very small. Indeed, we select $t$ and $s$
\begin{align}
&Q_{2r}\ss Q_{2s r,\nu}(x+rt\nu) \ss Q_{3 r} \label{eq:enimp2i1}\\
&Q_{s r,\nu}(x+rt\nu) \ss Q_{2r} \label{eq:enimp2i2}\\
&Q_{\t' s r /25} \ss Q_{\t' s r/24,\nu}(x+rt\nu) \label{eq:enimp2i3}\\
&\frac{1}{s}-\frac{t}{\t}\geq \frac{1}{2} \label{eq:enimp2i4}\\
&\mu(Q_{2s r,\nu}(x+rt\nu))=0. \label{eq:enimp2i5}
\end{align}
Note that the first four may be ensured by selecting $\e_E$ small enough (depending on $\t$ for \eqref{eq:enimp2i3}), while the final property is true for $\cL^1$-a.e. choice of $t$, and ensures that \eqref{eq:enimph2} holds. Then \eqref{eq:enimp2h1} and \eqref{eq:enimp2i1} directly imply \eqref{eq:enimph1}, if $\e_E$ is chosen small enough: this is nontrivial only for the $f$ term, in which case
\begin{align*}
f(u,x+t\nu,2sr,\nu) &= \frac{1}{(s r)^{n+1}}\int_{K \cap Q_{2sr,\nu}(x+rt\nu)} | x \cdot \nu|^2 d\cH^{n-1}\\
	& \leq \frac{1}{(s r)^{n+1}}\int_{K \cap Q_{3r}} | x \cdot \nu|^2 d\cH^{n-1} \\
	& \leq \frac{1}{(s r)^{n+1}}\int_{K \cap Q_{3r}} | x_n|^2 + |x|^2 |\nu - e_n|^2 d\cH^{n-1} \\
	& \leq C [f(u,3r) + \e_E^2]\\
	& \leq C \e_E^2.
\end{align*}
The penultimate step used the condition on $\nu$ in the infimum in \eqref{eq:enimp2h2}.

Now we check \eqref{eq:enimph3}, using \eqref{eq:enimp2h2}. Indeed, for the first three terms,
\begin{align*}
ff(u,x+rt\nu,\t' s r, \nu) + \L (2sr)^s + (2sr)^{\frac{1}{7}} &\leq C[ff(u,x,\t r,\nu) + \L r^s + r^{\frac{1}{7}}] \\
&\leq C \e_E E(u,2r)\\
&\leq C \e_E E(u,x+rt\nu,2sr,\nu).
\end{align*}
The first line used that $\t'sr + |rt\nu| \leq \t r - rst +rt <\t r$, while the last line used \eqref{eq:enimp2i1}. For the last term, we may use Lemma \ref{lem:intreg} to say
\[
 |u(r e_n) - u(x+r(t+s)\nu)| \leq C \sqrt{r}(\L r^s + E(u,3r))\leq C\e_E \sqrt{r}.
\] 
It follows that
\[
J(u,x+rt\nu, 2sr,\nu) \leq \1\frac{1}{1-C\e_E}\2 J(u,2r) \leq C \e_E E(u,x+rt\nu,2sr,\nu)
\]
as for the other terms.

We then learn from \eqref{eq:enimpc1} that
\[
E(u,x+rt\nu,\t' sr/24,\nu) \leq C'_E\t' E(u,x+rt\nu, sr,\nu) \leq C C'_E \t E(u,2r),
\]
with the last step from \eqref{eq:enimp2i2}. Using \eqref{eq:enimp2i3} and \eqref{eq:enimp2i4},
\[
E(u,x+rt\nu,\t' sr/24,\nu) \geq c E(u,\t' sr/25) \geq c E(u,\t r /50).
\]
This gives the conclusion \eqref{eq:enimp2c1}.
\end{proof}

We turn to the proof of Theorem \ref{thm:enimp}. The proof is rather involved, partially because we will need to use very sharp lower perimeter density estimates in the argument.

\begin{proof}[Proof of Theorem \ref{thm:enimp}.]
We argue by contradiction. Assume this estimate fails for some value of $\t$; then there are sequences $u_k,r_k, \e_k,\L_k$ with $\e_k\searrow 0$ and
\[
 E(u_k,2 r_k) + f(u_k,2 r_k) + \L_k r^s_k + r_k \rightarrow 0,
\]
\[
 ff(u_k, \t r_k) +\sqrt{\L_k}r_k^{s/2} +r_k^{\frac{1}{7}}+J^{\frac{2}{3}}(u_k,2 r_k) \leq \e_k E(u_k,2 r_k):= \e_k \xi_k^2,
\]
and yet
\[
  E(u_k,\frac{\t}{24} r_k) > C_E \t E(u_k, r_k)
\]
for a $C_E$ to be chosen later independently of $\t$. For convenience, choose $u_k(0,r_k)>u_k(0,-r_k)$; we will give the proof in the case when both are nonzero, but the argument proceeds analogously if one vanishes (up to replacing $ff$ by $f$). Any modifications needed in this case will be described in comments inside [square brackets].

Define the rescaled sets $\tilde{K}_k = K_{u_k}/r_k$, rescaled functions $\tilde{u}_k(x)=u_k(r_k x)$, and rescaled measures $\tilde{\m}_k(E) = r_k^{1-n}\m_k (r_k E)$. Along a subsequence, Lemmas \ref{lem:upperdensity} and \ref{lem:lowerdensity} imply that $\tilde{K}_k \rightarrow \pi \cap Q_2$ locally in Hausdorff topology, $\tilde{u}_k \rightarrow u_\8$ in $L^1 (Q_2)$, with $u_\8$ a locally constant function on $Q_2\sm \pi$, and $\tilde{\m}_k \rightarrow \m_\8 = [u_\8^2(e_n) + u_\8^2 (-e_n)] d\cH^{n-1}\mres \pi$ in weak-$*$ sense.

Consider the functions $v_k =\frac{\tilde{u}_k}{\sqrt{r_k} \xi_k}$. We have that
\[
 \int_{Q_2} |\n v_k|^2 d\cL^n = 1,
\]
while
\[
 \int_{Q_{\frac{\t}{24}}} |\n v_k|^2 d\cL^n > C_E \t^n \int_{Q_1}|\n v_k|^2 d\cL^n.
\]
Up to extracting a subsequence, we have that there are numbers $c^\pm_k$ such that $v_k - c^\pm_k \rightarrow v$ weakly in $H^1(Q_2 \cap \{x_n>\s (<-\s)\})$, where $v\in H^1(Q_2 \sm \pi)$ has
\[
 \int_{Q_2\sm \pi} |\n v|^2 d\cL^n \leq 1.
\]
For future reference, we note that $c_k^\pm$ may be taken as the mean value of $v_k$ on $Q_2 \cap \{x_n>\frac{1}{2}(<-\frac{1}{2})\}$. However, we have that on those regions the oscillation of $\tilde{u}_k$ is controlled by
\begin{equation}\label{eq:enimpi4}
 \osc_{Q_2 \cap \{x_n\geq 1/2\}} \tilde{u}_k \leq C r_k^{\frac{1}{2}} [\xi_k + \sqrt{\L_k r_k^{s}}]\leq C\xi_k r_k^{\frac{1}{2}},
\end{equation}
so the oscillation of $v_k$ over this region is bounded by a universal constant. We may therefore choose $c_k^\pm=v_k(\pm e_n)$ and obtain the same conclusion. With this choice, the estimate \eqref{eq:ldc1} of Lemma \ref{lem:lowerdensity} rescales to give 
\begin{equation}\label{eq:bigjump}
\xi_k (c_k^+ - c_k^-)\geq c_0>0, 
\end{equation}
where $c_0$ is universal. The same argument gives that $v_k -c_k^\pm \rightarrow v$ uniformly on compact subsets of $Q_2\sm \pi$.

We will now show that $v$ is harmonic and satisfies a Neumann condition on each half-cylinder. Take any test function $\phi \in C^1_c (Q_2)$, and set $\phi_k(x)= \xi_k \sqrt{r_k} \phi (x/r_k)$. Let $\psi: \R \rightarrow (0,1)$ be a bijective increasing function with $|\psi'|\leq t$. Then use $w_k = u_k + \s \phi_k \psi (\frac{u_k - c_k}{\sqrt{r_k}\xi_k})$ as a competitor for $u_k$, where $c_k = \frac{c_k^+ + c_k^-}{2}$. Note that $|\phi_k|\leq \|\phi\|_{C^{1}} \sqrt{r_k}\xi_k$, so $|\bar{w}_k -\uu_k| + |\underline{w}_k - \ud|\leq C\s \xi_k \sqrt{r_k} \|\phi\|_{C^1}$ at $\cH^{n-1}-$a.e. point. This gives
\begin{align*}
 \int_{Q_{2 r_k}} &|\n u_k|^2 d \cL^n \leq \int_{Q_{2 r_k}} |\n u_k|^2 + \s \n u_k \cdot [\n \phi_k \psi(\frac{u_k - c_k}{\sqrt{r_k}\xi_k}) + \frac{\phi_k \psi'}{\sqrt{r_k} \xi_k} \n u_k]d\cL^n\\
 &+ \int_{Q_{2 r_k}}2\s^2 \frac{\phi_k^2}{r_k \xi_k^2}|\psi'|^2 |\n u_k|^2 + 2\s^2 \psi^2 |\n \phi_k|^2 d\cL^n\\
  &+  C \s r_k^{n-1+1/2}\xi_k \|\phi\|_{C^1} + \L_k r_k^{n-1+s}.
\end{align*}
Dividing by $r_k^{n-1}\xi_k^2 \s$ and rescaling the integrals gives
\begin{align*}
 \int_{Q_2}& [\n v_k \cdot \n \phi] \psi(v_k-c_k) + |\n v_k|^2 \phi \psi'(v_k-c_k) d\cL^n\\
 &\geq - C\s \int_{Q_2} \phi^2 |\psi'|^2 |\n v_k|^2 + \psi^2|\n \phi|^2 d\cL^n - C\frac{\sqrt{r_k}}{\xi_k}\|\phi\|_{C^1} - \frac{\L_k r_k^s}{\s \xi_k^2}.
\end{align*}
The choice of $\s=\xi_k^2$ makes the entire right-hand side converge to $0$ as $k\rightarrow \8$, so we have that
\[
 \liminf \int_{Q_2} \n v_k \cdot \n \phi \psi(v_k-c_k) d\cL^n \geq - Ct \|\phi\|_{C^1},
\]
On the cylinder $Q_2$, $\n v_k \rightharpoonup \n v$ weakly in $L^2$, while from \eqref{eq:bigjump} and dominated convergence theorem $\n \phi \psi(v_k-c_k) \rightarrow \n \phi 1_{\{x_n>0\}}$ strongly. We conclude that
\[
 \int_{Q_2\cap \{x_n>0\}} \n v \cdot \n \phi \geq - Ct \|\phi\|_{C^1},
\]
and then send $t$ to $0$. After repeating on the other half-cylinder (using $1-\psi$ in place of $\psi$), this implies that $v$ is harmonic on $Q_2 \sm \pi$ and satisfies the Neumann condition on either side of $\pi$. In particular, for some universal constant $C(n)$,
\[
 \int_{Q_{\t/24}}|\n v|^2 d\cL^n \leq C(n)\t^n\int_{Q_1}|\n v|^2 \cL^n \leq C(n)\r^n \liminf \int_{Q_1}|\n v_k|^2.
\]
Notice that if we knew that $\n v_k$ converged to $\n v$ strongly in $L^2$ on $Q_{\t/24}$, this would be a contradiction. We now set out to prove this. We do know that
\[
 \int_{E} |\n v_k|^2 d\cL^n \rightarrow \int_{E} |\n v|^2 d\cL^n,
\]
for  $E\cc Q_2 \sm \pi$ (from Lemma \ref{lem:harmacon}).

Up to extracting a further subsequence, we have that $|\n v_k|^2 d\cL^n \rightharpoonup \nu$ locally on $Q_2$ to some measure $\nu$ in the weak-$*$ sense, and may find the Lebesgue decomposition $\nu = |\n v|^2 d\cL^n + \nu^s$. As we have just concluded that $\n v_k \rightarrow \n v$ strongly in $L^2$ on compact subsets of $Q_2 \sm \pi$, we have that $\nu^s$ is concentrated on $\pi$. We will now show that for any $x\in Q_{\t/24}\cap \pi$ and $\s$ small enough,
\[
 \nu^s(Q_\s(x)) \leq C\s^n,
\]
which implies that $\nu^s=0$ and that $\n v_k \rightarrow \n v$ strongly. We note that from now on, constants will depend on $\t$; as $\e_H$ is allowed to depend on $\t$, this is not a problem.

Let $L_{\pm,k}=L_{\pm}[u_k,\t r_k]$ be the affine functions parametrizing the minimal planes in the definition of $ff(u_k,\t r_k)$, and $h_k = L_{+,k}(0)$. We have
\begin{equation}\label{eq:enimpi12}
h_k/r_k \rightarrow 0 \text{ and } \max_{D_{\t r}}\frac{|L_{\pm,k}|}{r_k}\rightarrow 0
\end{equation}
as $k\rightarrow \8$. Extract $g_{\pm,k}$, the Lipschitz functions from Theorem \ref{thm:lip} applied to $Q_{\frac{\t}{4} r_k}$, and let $\G^\pm_k$ be their graphs. Using also Theorem \ref{thm:exbyflat2}, these graphs obey
\[
 \cH^{n-1}(Q_{\frac{\t}{24} r_k} \cap [K_{u_k} \triangle (\G^+_k \cup \G^-_k)]) \leq C \xi_k^2.
\]
We know from Proposition \ref{prop:liptoplane} that
\begin{equation}\label{eq:enimpi6}
 \frac{1}{ (\t r_k)^{n-1}}\int_{Q_{\t r_k/24}} |\n g_{+,k} - \n L_{+,k}|^2 + |\n g_{-,k} - \n L_{-,k}|^2 d\cL^{n-1} \leq C \xi_k^2
\end{equation}
and
\begin{equation}\label{eq:enimpi7}
 \frac{1}{(\t r_k)^{n+1}}\int_{Q_{\t r_k/24}} |g_{+,k} - L_{+,k}|^2 + |g_{-,k} - L_{-,k}|^2 d\cL^{n-1} \leq o_k(1)\xi_k^2.
\end{equation}
We will also use $\nu_{\pm,k}$ for the optimal vector and its reflection in the definition of $ff(u,\t r)$.

Given a point $x\in \pi$ and a radius $\s$, define the following sets:
\[
 Z_{+,k}(x,\s) :=  Q_{\s,\nu_{+,k}}(x',L_{+,k}(x')) \cap \{x_n \geq 0\},
\]
\[
 Z_{-,k}(x,\s) :=  Q_{\s,\nu_{-,k}}(x',L_{-,k}(x')) \cap \{x_n\leq 0\},
\]
and
\[
 Z_k(x,\s) := Z_{+,k}(x,\s)\cup Z_{-,k}(x,\s).
\]
We have glued two cut-off cylinders together; this construction will only be used when $|L_{+,k}|\ll \s$, and so $\pi$ only intersects the lateral sides of each cylinder, not their faces. Let $Y_{+,k}(x,\s)$ be the upper face of the cylinder $Q_{\s,\nu_{+,k}}(x',L_{+,k}(x'))$, i.e. the set $\{x:(x-h_k e_n)\cdot \nu_{+,k} = \s \}\cap \bar{Q}_{\s,\nu_{+,k}}(x',L_{+,k}(x'))$.  We will use the following coordinates to identify points in $y\in Z_{+,k}(x.\s)$: $y=(q_y,t_y)_*$, where $q_y\in D_\s$ is found by fixing an isometry $D_\s \rightarrow Y_{+,k}(x,\s)$ and pulling back the unique point in $ Y_{+,k}(x,\s) \cap \{y+\R \nu_+\}$ via this isometry, while $t_y \in [0,\8)$ is the unique number which allows the expression $y = t_\s \nu_+ + z'$ for some $z'\in \pi$. We likewise identify $Z_{-,k}(x,\s)$ with a subset of $D_\s \times (-\8,0]$. Up to changing one of the isometries of $D_\s \rightarrow Y_{\pm,k}$, these sets of coordinates agree on $Z_{+,k}\cap Z_{-,k}$, and so we may extend them to $Z_k$ in a well-defined way. For each $z\in D_\s$, choose $t^\pm(z)$ so that $(z,t^+(z))_*\in Y_{+,k}$ and $(z,t^-(z))_*\in Y_{-,k}$. Finally, we record for future reference that (up to choosing the parametrizations of $Y_{+,k}$ consistently), these coordinates are independent of $\s$, in the sense that if $y\in Z_k(x,\s)$ has coordinates $(q,t)_*$ relative to $Z_k(x,\s)$, it will have the same coordinates relative to any $Z_k(x,\s')$ with $\s'>\s$. See Figure \ref{fig:2} for an illustration. [If $c_k^-=0$, we take use $L_k=0$ for both planes, so $Z_k(x,\s)$ is the usual cylinder $Q_\s(x)$ and no special coordinates are needed.]

\begin{figure}
	\centering
	\def\svgwidth{15cm}
	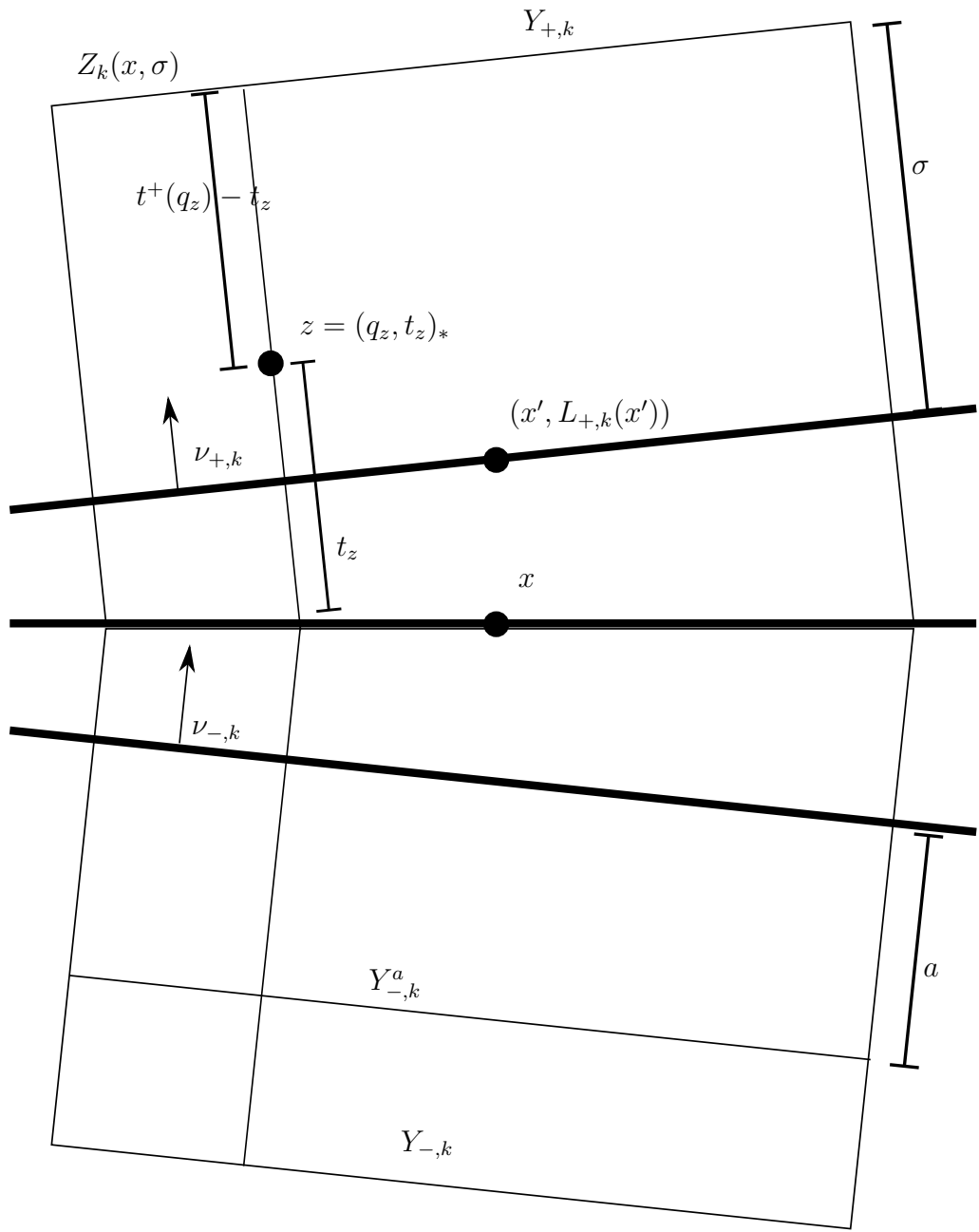
	\caption{The horizontal bold line is $\pi$, while the two bold lines at an angle are the graphs of $L_{\pm,k}$. The shape $Z_k(x,\s)$ is shown, as well as an arbitrary point $z$ with coordinates $(q_z,t_z)_*$; the line departing from $z$ represents the set of all points $w\in Z_k(x,\s)$ with $q_w=q_z$. The upper and lower faces $Y_{\pm,k}$ are labeled, as well as one of the sets $Y^a_{-,k}$ used in our estimates.} \label{fig:2}
\end{figure}

We will now spend some time obtaining a careful lower bound on the density of $\m_k$. Fix $x,\s$ so that $Q_\s (x)\in Q_{\t r_k/24}$; we will show that
\begin{equation}\label{eq:enimpi1}
 \m_k (Z_{k}(x,\s r_k)) \geq [u_k^2(0,r_k) + u_k^2(0,-r_k)] \w_{n-1}(\s r_k)^{n-1} - \xi_k^2 o_k(1)r_k^{n-1}
\end{equation}
for $k$ large. We will abbreviate $Z_k = Z_{k}(x,\s r_k)$, and similarly for the associated objects defined above. First, as $f(u_k,2 r_k)\leq \e_k$, we have that $K_{u_k}\cap Q_{2 r_k} \ss \{ |x_n|\leq C \e_k^{\frac{1}{n+1}} r_k \}$ from Proposition \ref{prop:flatsame}, and we may take $C \e_k^{\frac{1}{n+1}} < \s/10$. This implies that the faces $Y_{\pm,k}$ do not intersect $K$, nor do the disks $Y^{a}_{+,k}:=Y_{+,k} - (\s - a) r_k\nu_+$,  $Y^{a}_{-,k}:=Y_{-,k} + (\s - a) r_k\nu_-$ which are parallel to them, provided $a>2C\e_k^{\frac{1}{n+1}}$. We are using here that $|L_{\pm,k}|\leq C \e_k^{\frac{1}{n+1}} r_k$ as well. As for $\cL^1$-almost every such choice of $a$, we have that
\[
 \int_{Y^{a}_{\pm,k}/r_k}|v-v_k -c^+_k|^2 d\cL^n \rightarrow 0, 
\]
we may use a diagonal argument to find a sequence $a_k > 2C\e_k^{\frac{1}{n+1}}$, with $a_k \rightarrow 0$ and
\begin{equation}\label{eq:conv1}
\int_{Y^{a_k}_{\pm,k}/r_k}|v-v_k -c^+_k|^2 d\cL^n \rightarrow 0.
\end{equation}
Notice that for $\cL^{n-1}$-a.e. point in $z\in D_{\s r_k}$, one of the following two things happens: either no point in the set $\{(z,t)_* : t\in (t^-(z),t^+(z))\}$ lies in $K_{u_k}$, and along both line segments constituting this set $u_k$ is absolutely continuous (the set of such $z$ is called $H_k$), or else this set does intersect $K_{u_k}$, the restriction of $u$ to each line segment is $BV$, and if we set
\[
 t^+_*(z)=\max\{t\in (t^-(z),t^+(z)):(z,t)_* \in K_{u_k}\} 
\]
and
\[
 t^-_*(z)=\min\{t\in (t^-(z),t^+(z)):(z,t)_* \in K_{u_k}\} ,
\]
then $(z,t^\pm_*)_*\in J_{u_k} $ and  $\lim_{t\nearrow t^-_*} u_k((z,t)_*)$,$\lim_{t\searrow t^+_*} u_k((z,t)_*)$ exist and coincide appropriately with one of $\uu_k,\ud_k$ (the set of these $z$ is called $P_k$). That the complement of $H_k \cup P_k$ is $\cL^{n-1}$ negligible follows from applying Proposition \ref{prop:1Drest} on the tilted cylinders.

We estimate the size of $H_k$. For each $z\in H_k$, we may integrate
\begin{equation}\label{eq:enimpi2}
 |u_k((z,t^+(z) -(\s-a_k)r_k)_*) - u_k((z,t^-(z) + (\s-a_k)r_k)_*)| \leq \int_{t^-(z) + (\s-a_k)r_k}^{t^+(z) - (\s-a_k)r_k} |\n u_k((z,t)_*)| d\cL^1(t).
\end{equation}
Squaring and using H\"older's inequality gives
\[
 \int_{Y^{a_k}_{+,k}\cap \{q_x\in H_k\}} |u_k(x) - u_k(-x)|^2 d\cH^{n-1}(x) \leq Ca_k r_k^n \xi_k^2. 
\]
We used here that the length of the interval of integration in \eqref{eq:enimpi2} is controlled by $2 a_k + C\e_k^{\frac{1}{n+1}}\leq Ca_k$.  Rewriting for $v_k$, we have
\[
 \int_{(Y^{a_k}_{+,k}\cap \{q_x\in H_k\})/r_k} |v_k(x) - v_k(-x)|^2 d\cH^{n-1}(x) \leq C a_k.
\]
From \eqref{eq:conv1}, we have
\[
 \int_{(Y^{a_k}_{\pm,k}\cap \{q_x\in H_k\})/r_k} |v_k(x)-c^\pm_k|^2 d\cH^{n-1}(x) \leq 2 \int_{(Y^{a_k}_{\pm,k}\cap \{q_x\in H_k\})/r_k}v^2 d\cH^{n-1}\leq C \frac{\cL^{n-1}(H_k)}{r_k^{n-1}}.
\]
This implies
\[
 \cL^{n-1}(H_k) (c_k^+-c_k^-)^2 \leq C \1\cL^{n-1}(H_k) + a_k r_k^{n-1}\2.
\]
Reabsorbing the first term on the right and using \eqref{eq:bigjump}, we have that
\begin{equation}\label{eq:enimpi3}
 \frac{\cL^{n-1}(H_k)}{r_k^{n-1}}\leq C\xi_k^2 a_k = \xi_k^2 o_k(1)
\end{equation}
for $k$ large.

Now we look at what happens over $P_k$. From \eqref{eq:enimpi4} and $r_k^{\frac{1}{7}} <\e_k \xi_k^2$, we know that
\[
 \sup_{Q_{2 r_k}\{x_n\geq r_k/4\}}|u_k-u_k(0,r_k)| \leq C\e_k \xi_k^2,
\]
and similarly for the other side. We consider then
\[
 \int_{\{(z,t^+_*(z))_*:z\in P_k\}} \lim_{t\searrow t_*^+(z)}u_k^2((z,t)_*) d\cH^{n-1} +  \int_{\{(z,t^-_*(z))_*:z\in P_k\}} \lim_{t\nearrow t_*^-(z)}u_k^2((z,t)_*) d\cH^{n-1} \leq \m_k(Z_k),
\]
(note that we used the assumption that $\m_k(\pi)=0$ here) and estimate each piece from below. The main observation is that
\[
 \int_{\{(z,t^+_*(z))_*:z\in P_k\}} \lim_{t\searrow t_*^+(z)}u_k^2((z,t)_*) d\cH^{n-1}\geq \int_{P_k} \lim_{t\searrow t_*^+(z)}u_k^2((z,t)_*) d\cL^{n-1}.
\]
This follows from the area formula and the fact that the projection $(q,t)_*\rightarrow q$ is a contraction. Integrating over each line segment (or line segment pair, as the case may be),
\[
 |u_k((z,r_k/2)_*) - \lim_{t\searrow t_*^+(z)}u_k((z,t)_*)| \leq  \int_{t^+(z)}^{r_k/2}|\n u_k((z,t)_*)| d\cL^1(t), 
\]
so
\begin{equation}\label{eq:enimpi16}
 |u_k(0,r_k) - \lim_{t\searrow t_*^+(z)}u_k((z,t)_*)| \leq  \int_{t^+(z)}^{r_k/2}|\n u_k((z,t)_*)| d\cL^1(t) + \e_k \xi_k^2.
\end{equation}
Integrating over $P_k$,
\begin{align*}
 \left|u_k^2(0,r_k) \cL^{n-1}(P_k) - \int_{P_k} \lim_{t\searrow t_*^+(z)}u_k^2((z,t)_*) d\cL^{n-1}\right| &\leq \e_k \xi_k^2 r_k^{n-1} + \int_{Q_{2 r_k}}|\n u_k| d\cL^n\\
 &\leq Cr_k^{n-1}(\e_k \xi_k^2 + r_k^{1/2} \xi_k)\\
 &\leq Cr_k^{n-1}\e_k \xi_k^2. 
\end{align*}
Combining with \eqref{eq:enimpi3}, we have just shown that
\begin{align}
 \frac{\m_k(Z_k)}{r_k^{n-1}} & \geq [u_k^2 (0,r_k) + u_k^2(0,-r_k)]\frac{\cL^{n-1}(P_k)}{r^{n-1}_k} - C\xi^2_k o_k(1) \nonumber\\
 & \geq [u_k^2 (0,r_k) + u_k^2(0,-r_k)]\w_{n-1}\s^{n-1} -C\frac{\cL^{n-1}(H_k)}{r^{n-1}_k} - C\xi^2_k o_k(1) \nonumber\\
 &\geq [u_k^2 (0,r_k) + u_k^2(0,-r_k)]\w_{n-1}\s^{n-1} - C\xi^2_k o_k(1).\label{eq:enimpi5}
\end{align}

Define the measure $\a_k = \xi_k^{-2} \cH^{n-1}\mres ([K_{u_k}\sm (\G^+_k \cup \G^-_k)]/r_k)$; using Theorems \ref{thm:lip} and \ref{thm:exbyflat2} we have that $\a_k(Q_{\t/24})\leq C$ uniformly in $k$. Up to passing to a further subsequence we may assume that $\a_k \rightharpoonup \a$ in the weak-$*$ sense, with $\a$ a Borel measure concentrated on $\pi$. 

We will now build a competitor for $u_k$. Observe that over each set $Z_k(x,\s r_k)$, we may view $g_{\pm,k}$ as graphs over the affine planes $\{x: (x\mp he_n)\nu_\pm = 0 \}$, using the notation $g^*_{\pm,k} : D_{\s r_k} \rightarrow \R$ for the mapping of $z$ to the unique value so that $(z,t^+(z)-\s r_k+g^*(z))_*\in \G^+_k$. These functions are $\frac{1}{2}-$Lipschitz, and satisfy
\begin{equation}\label{eq:enimpi8}
 \frac{1}{r_k^{n+1}}\int_{D_{\s r_k}}|g^*_{\pm,k}|^2 d\cL^{n-1} \leq \xi_k^2 o_k(1)
\end{equation}
(this follows from changing variables in \eqref{eq:enimpi7}) and also
\begin{equation}\label{eq:enimpi9}
 \frac{1}{r_k^{n-1}}\int_{D_{\s r_k}}|\n g^*_{\pm,k}|^2 d\cL^{n-1} \leq C\xi^2.
\end{equation}
This second fact comes again from changing variables in \eqref{eq:enimpi6} and using \eqref{eq:liptoplanec4}. It is easy to see that despite the definitions, $g^*_{\pm,k}$ is independent of the radius $\s r_k$ of the ambient cylinder $Z_k$. [If $c_k^-=0$, simply use $g^*=g$.]

We fix $x$ and $\s_0$ so that $Q_{3\s_0}(x)\ss Q_{\t/24}$, and then select a value $\s \in (\s_0/2,\s_0)$ (passing to a further subsequence) so that the following properties hold, with constant possibly depending on $\s$:
\begin{enumerate}
 \item 
 \[
\a(\p Q_\s(x)) + \nu^s(\p Q_{\s} (x)) = 0.  
 \]
 \item
 \[
 \sup \frac{1}{\xi^2_k r_k^{n-2}}\int_{\p D_{\s r_k}} |\n g_{+,k}^*|^2 + |\n g_{-,k}^*|^2 d\cL^{n-2} \leq C.
\]
 \item
\[
 \lim \frac{1}{\xi^2_k r_k^{n}}\int_{\p D_{\s r_k}}| g_{+,k}^*|^2 + | g_{-,k}^*|^2 d\cL^{n-1} = 0.
\]
\item
\[
 \lim \frac{1}{\xi_k^2 r_k^{n-1}}\cH^{n-1}(\{y\in \p Z_k(x,\s r_k): \uu_k(y)>0, g_{-,k}(y')<y_n<g_{+,k}(y')\}) =0.
\]
\end{enumerate}
The first of these is true for $\cH^{n-1}$-a.e. $\s$ simply because $\a +\nu^s$ is a finite Borel measure supported on $\pi$. That (2) and (3) are possible follows from Chebyshev's inequality and our previous estimates \eqref{eq:enimpi9} and \eqref{eq:enimpi8}, respectively. We now carefully check (4). [Note that there is nothing to show if $c_k^-=0$, for then $g_-=g_+$.] First, it suffices to show that
\begin{equation}\label{eq:enimpi11}
 \cL^n(\{u_k>0\} \cap G \cap Q_{2\s_0}(x))\leq r_k^n \xi^2_k o_k(1),
\end{equation}
where $G=\{x\in Q_{\t r_k/24} : g_-(x')<x_n<g_+(x')\}$ is the region between the graphs. We have the easy bound
\begin{equation}\label{eq:enimpi10}
 \frac{\cL^n(\{u_k>0\} \cap G)}{r_k^n} \leq C \xi_k^{\frac{2}{n+1}}
\end{equation}
from estimating the size of the two regions from Proposition \ref{prop:fflatsame} and using Lemma \ref{lem:zerobetween} in the case that $\xi_k^{\frac{2}{n+1}}r_k \ll h_k$. We will now give a mechanism for obtaining an iterative improvement of the size of this set. Fix $\g<1$ so that $Q_{2\s_0}(x)\ss Q_{\g \t r_k/24}$, and $k_0$ so that $(1+2^{-k_0})\g < 1$. Set $A_t = \{u_k>0\} \cap G \cap Q_{\frac{\t \g r_k}{24}(1+t)} \sm K_{u_k}$. Then applying Friedrich's inequality (this is equivalent to using the $BV$ Sobolev inequality on $u_k^2 1_{A_t}$), we have
\[
 \cL^n(A_t)^{\frac{n-1}{n}}\leq C \1\int_{A_t}u_k^{\frac{2n}{n-1}} d\cL^n\2^{\frac{n-1}{n}} \leq C\1 \int_{A_t}|\n u_k|^2 d\cL^n + \int_{\p A_t} u_k^2 d\cH^{n-1}\2.
\]
Here the first inequality used that $u 1_{A_t}\geq \d 1_{A_t}$, from \eqref{eq:dregsize}. Averaging over $t\in (2^{-i},2^{1-i})$, this gives
\[
 \cL^n(A_{2^{-i}})^{\frac{n-1}{n}} \leq C^i \1 \int_{A_{2^{1-i}}} |\n u_k|^2 d\cL^n + \cH^{n-1}(\p A_{2^{1-i}} \cap Q_{\frac{\t \g r_k}{24}(1+2^{1-i})}) + r_k^{-1}\cL^n(A_{2^{1-i}})\2.
\]
The third term captures the contribution of $\p Q_{\frac{\t \g r_k}{24}(1+t)}$ to $\p A_t$, while the second term contains the contribution from $\p \{u_k>0\}$. 

Now, $\p A_{2^{-i}} \cap Q_{\frac{r_k \t \g}{24}(1+2^{-i})}\ss K_{u_k}$, so it is countably $\cH^{n-1}$-rectifiable, and $(\p A_{2^{-i}}\sm \p^* A_{2^{-i}}) \cap Q_{\frac{r_k \g \t}{24}(1+2^{-i})}\ss K_{u_k}$ is $\cH^{n-1}$-negligible. Note that $\p^* A_{2^{-i}} \cap Q_{\frac{r_k \t \g}{24}(1+2^{-i})} \ss K_{u_k} \sm (\Theta^+_k \cup \Theta^-_k)$, with $\Theta^\pm_k$ as in Remark \ref{rem:lipest}. Indeed, at any point $x\in \Theta^+_k\ss \G_k^+$, we have that $\ud(x)=0$ and $\nu_x \cdot e_n >0$; if $x\in \p^* A_{2^{-i}}$ as well, then $\nu_x$ would have to be the measure-theoretic unit normal to $A_{2^{-i}}$ at $x$ as well, and this contradicts that $u\geq \d$ on $A_{2^{-i}}$. This gives
\[
 \1\frac{\cL^n(A_{2^{-i}})}{r_k^n}\2^\frac{n-1}{n} \leq C^i \1 \xi_k^2 + \frac{\cL^n(A_{2^{1-i}})}{r_k^n}\2.
\]
Applying this recursion repeatedly, starting with $l=k_0$ and using \eqref{eq:enimpi10}, we deduce that
\[
 \frac{\cL^n(A_{2^{-l}})}{r_k^n} \leq C^l (\xi_k^{\frac{2n}{n-1}} + \xi_k^{\frac{2}{n+1} (\frac{n}{n-1})^{l-k_0}} ).
\]
Setting $l = k_0 + \lceil \frac{\log (n+1) + \log n - \log(n-1)}{\log n - \log(n-1)}\rceil$ gives
\[
\frac{\cL^n(A_{2^{-l}})}{r_k^n} \leq C^l \xi_k^{\frac{2n}{n-1}} = \xi_k^2 o_k(1),
\]
which gives \eqref{eq:enimpi11}, as $G \cap Q_{2\s_0}(x) \ss A_{2^{-l}}$ for any $l$.

We proceed with the construction of competitors. Let $h_{+,k}$ be the $3/4$-Lipschitz function obtained from applying Lemma \ref{lem:minareaest} to $g_{+,k}^*$ on $D_{\s r_k}$: it satisfies $g_{+,k}^* =h_{+,k}$ on $\p D_{\s r_k}$ and the estimate
\begin{equation}\label{eq:enimpi13}
\int_{D_{\s r_k}} \sqrt{1 +|\n h_{+,k}|^2} d\cL^{n-1}\leq \w_{n-1}(\s r_k)^{n-1}+ C\xi_k^2 r_k^{n-1} \1\frac{1}{\a_k} + \a_k o_k(1)\2
\end{equation}
for an $\a_k \in (1,\frac{r_k}{4\|g_{+,k}^*\|_{L^\8}})$, using also properties (2) and (3) of $\s$. As $\frac{r_k}{4\|g_{+,k}^*\|_{L^\8}}\rightarrow \8$ from \eqref{eq:enimpi12}, we may choose $\a_k$ so that the term in the parentheses is $o_k(1)$. Define $h_{-,k}$ similarly; it is straightforward to check from the definitions that $t^+(q)-\s r_k + h_{+,k}(q) \geq t^-(q) - \s r_k + h_{-,k}(q)$ for $q\in D_{\s r_k}$.

From the definition of $ff(u,\t r)$, we have that $|\n L_{+,k}| \leq \frac{h_k}{50 \t r_k}$, which implies that $L_{+,k}> \frac{49h_k}{50}$ on that set. This means
\[
\{q\in D_{\s r_k}: t^+(q) -\t r_k + h_{+,k}(q) < 0 \}\ss \{q\in D_{\s r_k}: h_{+,k}(q) < \frac{-h_k}{2}\}.
\]
Combining with \eqref{eq:maec2} gives
\begin{equation}\label{eq:enimpi14}
\cL^{n-1}(\{q\in D_{\s r_k}: t^+(q) -\t r_k + h_{+,k}(q) < 0 \}) \leq \frac{r_k^{n-1}}{h_k^2} \xi_k^2 o_k(1).
\end{equation}

Let $S^+_k = \{(q,t^+(q)-\s r_k + h_{+,k}(q))_*: q\in D_{\s r_k}\}$ be the graph generated by $h_{+,k}$, and similarly $S^-_k = \{(q,t^-(q)+\s r_k + h_{-,k}(q))_*: q\in D_{\s r_k}\}$. We now estimate the size of $S^+_k$:
\begin{align}
\cH^{n-1}(S^+_k) &= \int_{q_x \in D_{\s r_k}} \frac{1}{\nu_x \cdot \nu_+} d\cL^{n-1} \nonumber\\
& = \int_{\{t^+(q) -\t r_k + h_{+,k}(q) \geq 0 \}} \sqrt{1 + |\n h_{+,k}|^2} d\cL^{n-1} \nonumber\\
&\quad+ \int_{\{ t^+(q) -\t r_k + h_{+,k}(q) < 0 \}} \sqrt{1 + |\n h_{+,k} - 2 \n L_{+,k}|^2} d\cL^{n-1}\nonumber\\
& \leq \w_{n-1}(\s r_k)^{n-1}+ \xi_k^2 r_k^{n-1} o_k(1) +\int_{\{ t^+(q) -\t r_k + h_{+,k}(q) < 0 \}} |\n L_{+,k}|^2 d\cL^{n-1}\nonumber\\
& \leq \w_{n-1}(\s r_k)^{n-1}+ \xi_k^2 r_k^{n-1} o_k(1) +\frac{r_k^{n-1}}{h_k^2} \xi_k^2 o_k(1) \cdot h_k^2\nonumber\\
& \leq \w_{n-1}(\s r_k)^{n-1}+ \xi_k^2 r_k^{n-1} o_k(1). \label{eq:enimpi15}
\end{align}
Here the first line uses the area formula (note that $S^+_k$ is a Lipschitz graph over the plane $\{(x',L_{+,k}(x'))\}$); here $\nu_x$ represents the unit normal $S_{+,k}$, oriented so that $\nu_x\cdot e_n>0$. The second line just computes these normal vectors in terms of $h_{+,k}$. The third line used \eqref{eq:enimpi13}, while the fourth line is from \eqref{eq:enimpi14}.

Set $\zeta_{+,k}$ to be the largest $3/4$-Lipschitz function agreeing with $g_{+,k}^*$ on $\p D_{\s r_k}$, and likewise $\zeta_{-,k}$ to be the smallest $3/4$-Lipschitz function agreeing with $g_{-,k}^*$ on $\p D_{\s r_k}$. We construct a Lipschitz map $\Phi_{+,k}:Z_k\rightarrow Z_k$ which has bounded Lipschitz constant, leaves the region below (in $*$ coordinates) $t^+-\s r_k + h_{+,k}$ and the sides $\p Z_+$ unchanged, and sends the region between $t^+ -\s r_k + \zeta_{+,k}$ and $t^+ -s r_k + h_{+,k}$ to $S^+_k$ (this is elementary to construct, and is done explicitly in \cite{AFP}). Likewise define $\Phi_{-,k}$ and $\Phi_k$, their composition. Then the competitor $w_k$ is given by
\[
 w_k(y) = \begin{cases}
           u_k(y) & y\notin Z_k \\ 
           0 & t_y\in (t^-(q_y)+\s r_k+h_{-,k}(q_y),t^+(q_y) -\s r_k + h_{+,k}(q_y)) \\
           u\circ \Phi^{-1} & \text{otherwise}.
          \end{cases}
\]
See Figure \ref{fig:3} for a drawing. Set $W_k$ to be the region between the graphs of $t^- + \s r_k +\zeta_{-,k}$ and $t^+ - \s r_k + \zeta_{+,k}$.  Observe that $Z_k\cap K_{w_k}\sm (S^+_k \cup S^-_k)$ is contained in the image under $\Phi$ of $K_{u_k}\cap W_k$, for $K_{u_k}\cap W_k$ is mapped either to $S^\pm_k$ or to the region where $w_k=0$ under $\Phi$. As $Z_k\sm W_k$ is devoid of $\G^\pm_k$ (using that $g^*_{\pm , k}$ are $1/2$-Lipschitz) this gives
\begin{equation}\label{eq:enimpi17}
\cH^{n-1}(Z_k\cap K_{w_k}\sm (S^+_k \cup S^-_k)) \leq C \a_k([Z_k\sm W_k]/r_k) \xi_k^2 r_k^{n-1}.
\end{equation}
Likewise
\begin{equation}\label{eq:enimpi18}
 \int_{Z_k} |\n w_k|^2 d\cL^n \leq C \int_{Z_k\sm W_k} |\n u_k|^2 d\cL^n.
\end{equation}

\begin{figure}
	\centering
	\def\svgwidth{15cm}
	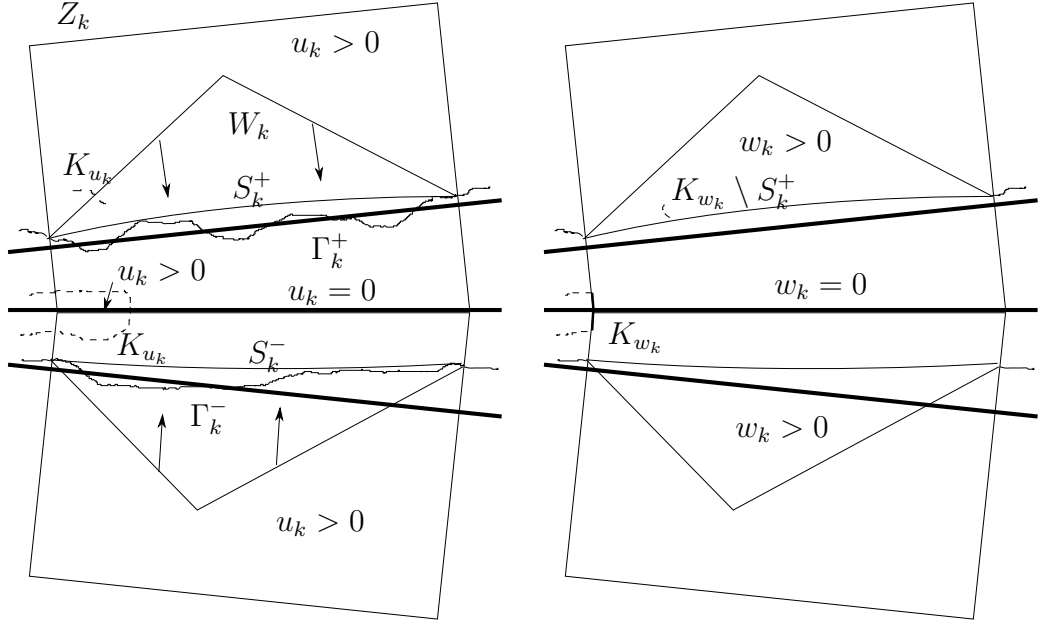
	\caption{On the left, the set $Z_k$ is drawn together with the graphs of $t^\pm \mp \s r_k+\zeta_{\pm,k}$ (these are the triangular shapes defining the boundary of $\W_k$), the sets $S^\pm_k$ (these are the smooth curves), and the graphs $\G^\pm_k$ (these are the solid, jagged curves). Note that all three coincide on $\p Z_k$. The portions of $K_{u_k}$ not contained in $\G^\pm_k$ are shown as dashed curves. The arrows show the effect of the mapping $\Phi_k$,  which collapses the portion of $\W_k$ above $S_k^+$ onto $S_k^+$, and likewise with the region below $S^-_k$. On the right is the resulting function $w_k$, which vanishes between $S_k^\pm$. Note that $K_{w_k}$ consists of $S^\pm_k$, the image under $\Phi$ of $K_{u_k}\sm W_k$ (a piece of which appears as the dashed line above $S_k^+$), and a new contribution in $\p Z_k$ due to setting $w_k=0$ between $S^\pm_k$ (this is the bold line on the left edge).} \label{fig:3}
\end{figure}

Arguing as in \eqref{eq:enimpi16}, we have
\[
|u_k^2(0,r_k) - u_k^2((z,t^+(z)-\s r_k + \zeta_{k,+}(z))_*)| \leq  C\int_{-r_k/2}^{r_k/2}|\n u_k((z,t)_*)| d\cL^1(t) + C\e_k \xi_k^2
\]
for each $z\in P_k$. This gives
\begin{align*}
 \int_{ S^+_k} u_k^2((z,t^+(z)-\s r_k + \zeta_{k,+}(z))_*) d\cH^{n-1} & \leq \cH^{n-1}(S_k^+) u_k^2(0,r_k) + C\cL^{n-1}(H_k) +C\e_k \xi_k^2 r_k^{n-1}\\
 & \leq \w_{n-1} (\s r_k)^{n-1} u^2_k(0,r_k) + \xi_k^2 r_k^{n-1}o_k(1).
\end{align*}
The second line used \eqref{eq:enimpi15} and \eqref{eq:enimpi3}. Combining with  \eqref{eq:enimpi17} gives
\begin{align}
\int_{K_{w_k}\cap Z_k} \overline{w}_k^2 + \underline{w}_k^2 d\cH^{n-1} &\leq \int_{ S^\pm_k} u_k^2((z,t^\pm (z)\mp\s r_k \nonumber\\
&\quad + \zeta_{k,\pm}(z))_*) d\cH^{n-1} + C \cH^{n-1}(Z_k\cap K_{w_k}\sm (S^+_k \cup S^-_k)) \nonumber\\
& \leq  \w_{n-1} (\s r_k)^{n-1} (u^2_k(0,r_k)+u^2(0,-r_k)) \nonumber\\
&\quad + \xi_k^2 r_k^{n-1}(o_k(1) + C \a_k([Z_k\sm W_k]/r_k))\nonumber \\
& \leq  \m_k(Z_k) + \xi_k^2 r_k^{n-1}(o_k(1) + C \a_k([Z_k\sm W_k]/r_k)),\label{eq:enimpi19}
\end{align}
with the last line from \eqref{eq:enimpi5}. Using property (4) of $\s$ gives
\begin{equation}\label{eq:enimpi20}
\cH^{n-1}(\p Z_k\cap K_{w_k}) \leq \xi_k r_k^{n-1}o_k(1).
\end{equation}

Using the minimality relation, \eqref{eq:enimpi18},\eqref{eq:enimpi19}, and \eqref{eq:enimpi20},
\[
 \int_{W_k}|\n u_k|^2 d\cL^n \leq  C \int_{Z_k\sm W_k} |\n u_k|^2 d\cL^n + \xi_k^2 r_k^{n-1} (o_k(1) + \a_k([Z_k\sm W_k]/r_k)).
\]
As $Z_k/r_k \rightarrow Q_\s (x)$ and $W_k /r_k \rightarrow W_\8 = \{|y_n|< \s- \frac{3}{4}|y'-x'|\}$, after dividing by $\xi_k^2 r_k^{n-1}$ and taking $k\rightarrow \8$ (using property (1) of $\s$ to pass to the limit in all of the terms involving measures),
\[
 \int_{W_\8}|\n v|^2 d\cL^n + \nu^s (W_\8)\leq C\left[\int_{Q_\s(x)\sm W_\8}|\n v|^2 d\cL^n + \a(Q_\s(x)\sm W_\8)+\nu^s (Q_\s(x)\sm W_\8)\right]. 
\]
The last terms vanish, as $\a,\nu^s$ are concentrated on $\pi$, so
\[
 \nu^s (W_\8)\leq C\int_{Q_\s(x)}|\n v|^2 d\cL^n \leq C\s^n.
\]
We have shown that
\[
 \nu^s (Q_{\s_0/2} (x))\leq C\s_0^n
\]
for each $\s_0$, which implies that $\nu^s$ is absolutely continuous with respect to Lebesgue measure. As $\nu^s$ was the singular part in a Lebesgue decomposition, this means it must vanish on $Q_{\t/24}$, which completes the argument.
\end{proof}

\section{Competitors}\label{sec:EL2}

Here we prove estimates on various competitors. The first lemma is much easier to prove than the others, and will be used in situations where $f$ and $ff$ are comparable. The second lemma will give information when $ff\ll f$, and will require more delicate constructions.

\begin{lemma}\label{lem:harm1} Let $u\in \sQ(Q_{12 r},\L,24 r)$. For each $\Upsilon>0$, there is a constant $\e_8=\e_8(\Upsilon)$ such that if
\[
 f(u,12 r) + E(u,12 r) + \L r^s + r \leq \e_8, \qquad E(u,12 r) + \sqrt{\L r^s}+r \leq  \Upsilon f(u,12 r),
\]
then for each $\phi$ in $C^\8_c (D_r)$,
\[
 \left|\int_{D_r}[u^2(0,r)\n g_+ + u^2(0,-r)\n g_-] \cdot \n \phi d\cL^{n-1}\right| \leq C(\Upsilon) \|\phi \|_{C^{1}} f(u,6r)r^{n-1}.
\]
Here $g_\pm$ are the functions from Theorem \ref{thm:lip} applied to $Q_r$ with $\nu_* = e_n$.
\end{lemma}

\begin{proof}
We may as well assume that $\|\phi\|_{C^1}$ is sufficiently small. Set $ T = \xi(x_n) \phi(x')e_n$, where $\xi$ is a smooth cutoff function which is $1$ on $(-\frac{r}{2},\frac{r}{2})$ and supported on $(-r,r)$. We now plug $T$ into the equation of Lemma \ref{lem:EL}. This results in
 \begin{align*}
  &\left|\int_{J_{u}}(\uu^2(x)+\ud^2(x))\dvg\nolimits^{\nu_x}T d\cH^{n-1}(x) + \int_{Q_{r}}|\n u|^2 \dvg T - \n u \cdot \n T \n u d\cL^n\right|\\
  &\qquad\leq C(n) \s  r^{n-1}\1 E(u,r) + 1\2 + \frac{\L}{\s}r^{n-1+s}
 \end{align*}
for each $\s$ small enough. Applying Theorem \ref{thm:exbyflat} on $Q_{6 r}$, Theorem \ref{thm:lip} with $\nu_*=e_n$ on $Q_r$, and Remark \ref{rem:lipest} lets us alter the domain of integration in the first term to $\Theta^{+}\cup \Theta^-$ and replace $\m$ with a multiple of Hausdorff measure:
\begin{align*}
 &\left|\int_{\Theta^+}u^2(0,r)\dvg\nolimits^{\nu_x}T d\cH^{n-1} + \int_{\Theta^-}u^2(0,-r)\dvg\nolimits^{\nu_x}T d\cH^{n-1} \right|\\
 &\qquad\leq C(n) \s  r^{n-1} + C\1 E(u,12r)+f(u,12r) + r \2r^{n-1} + \frac{\L}{\s}r^{n-1+s}.
\end{align*}
We also moved all of the energy terms to the right-hand side and estimated them crudely. Another application of Remark \ref{rem:lipest} lets us change the domain of integration to $\G^{\pm}$, up to increasing the constants on the right:
\begin{align*}
 &\left|\int_{\G^+}u^2(0,r)\dvg\nolimits^{\nu_x}T d\cH^{n-1} + \int_{\G^-}u^2(0,-r)\dvg\nolimits^{\nu_x}T d\cH^{n-1} \right|\\
 &\qquad\leq C(n) \s  r^{n-1} +C(\Upsilon) f(u,12r)r^{n-1} + \frac{\L}{\s}r^{n-1+s}.
\end{align*}
Using the hypotheses and setting $\s=f(u,12r)$, we see that the entire right-hand side is controlled by $C(\Upsilon)f(u,12r)r^{n-1}$.

On the other hand, the terms on the left may be rewritten as integrals over the disk, using the area formula:
\[
 \int_{\G^+}u^2(0,r)\dvg\nolimits^{\nu_x}T d\cH^{n-1}=\int_{D_{r}}u^2(0,r)\dvg\nolimits^{\nu_x}T \sqrt{1+|\n g_+|^2}d\cL^{n-1}.
\]
The tangential divergence here may be computed directly in terms of the gradient of $g_{\pm}$, using that $\xi$ is constant on $\G^{\pm}$, and hence $\dvg T =0$:
\begin{align*}
 \dvg\nolimits^{\nu_x}T &= - \frac{(\n g_+,-1)}{\sqrt{1+|\n g_+|^2}} \cdot \n T \frac{(\n g_+,-1)}{\sqrt{1+|\n g_+|^2}}\\
 &=- \frac{(\n g_+,-1)}{\sqrt{1+|\n g_+|^2}} \frac{(0,\n g_+ \cdot \n \phi)}{\sqrt{1+|\n g_+|^2}}\\
 &=\frac{\n g_+ \cdot \n \phi}{1+|\n g_+|^2}.
\end{align*}
Each of the integrals may then be expressed as
\[
\int_{\G^+}u^2(0,r)\dvg\nolimits^{\nu_x}T d\cH^{n-1}=\int_{D_{r}}u^2(0,r) \n g_{+} \cdot \n \phi  d\cL^{n-1} + O(\int_{D_r}|\n g_{+}|^2 d\cL^{n-1}),
\]
with the error term controlled by $C(\Upsilon)f(u,12r)r^{n-1}$ by the hypothesis and Lemma \ref{prop:liptoplane}.
\end{proof}

We will now assume that $u(0,\pm r)\neq 0$, and show that a similar statement is available with $f$ replaced by $ff$. The argument is based on perturbing the graph $\G^+$ while leaving $\G^-$ fixed; Lemma \ref{lem:cut} is used to turn this into a viable competitor, and some additional care is needed to deal with situations where the perturbed $\G^+$ now intersects with $\G^-$.

\begin{lemma}\label{lem:harm2} Let $u\in \sQ(Q_{24r},\L,48 r)$. Assume $u(0,\pm r)\neq 0$. For each $\Upsilon>0$, there is a constant $\e_9=\e_9(\Upsilon)$ such that if
\[
 f(u,24r) + E(u,24r) + \L r^s + r \leq \e_9, \qquad E(u,24r) + \sqrt{\L r^s} + J^{\frac{2}{3}}(u,24r) + r^{\frac{1}{7}} \leq  \Upsilon ff(u,24r),
\]
then for each $\phi \geq 0$ in $C^\8_c (D_r)$,
\[
 \left|\int_{D_r}[\n g_+ - \n L_+] \cdot \n \phi d\cL^{n-1}\right| \leq C(\Upsilon) \|\phi \|_{C^{1}} [\frac{h^2}{r^2}ff^{\frac{1}{2}}(u,6r)+ff(u,6r)]r^{n-1}.
\]
Here $L_+=L_+[u,24 r]$, $h = h(u,24 r)$, $\nu_+ = \nu_+(x,24 r)$ are the minimal parameters for $ff(u,24r)$ and $g_+$ is function from Theorem \ref{thm:lip} applied on $Q_r$ with $\nu_* = \nu_+$.
\end{lemma}

\begin{proof} We may restrict our attention to $\phi$ with $\|\phi\|_{C^1} \leq 1$. Set $u_+=u(0,r)$ and $u_-=u(0,-r)$, and abbreviate $ff=ff(u,24r)$. Up to choosing $\e_9$ small, we may guarantee that
\[
 \sup_{Q_{12 r} \cap \{x_n \geq \frac{r}{10}\}} |u_+ -u(x)| \leq C \sqrt{r\e_9},
\]
and likewise on the lower half of the cylinder, from Lemma \ref{lem:intreg}.

First, we apply Lemma \ref{lem:cut} with $p=2$ and $\xi$ to be chosen to find a closed, countably $\cH^{n-1}$-rectifiable set $Z\ss Q_{2r}\cap \{|x_n|\leq r/10\}$ with the following properties:
\begin{enumerate}
 \item $Z$ separates $Q_r$ into (at least) two disjoint connected open sets, one of which (we call it $S_+$) contains $Q_r \cap \{x_n \geq r/2\}$, and another, $S_-$, contains $Q_r \cap \{x_n \leq -r/2\}$. 
 \item On $S_+$, $u \in (u_+ - C\frac{r^{1/2}}{\xi},u_+ + C\frac{r^{1/2}}{\xi})$, and similarly on $S_-$, $u \in (u_- - C\frac{r^{1/2}}{\xi},u_- + C\frac{r^{1/2}}{\xi})$.
 \item $\cH^{n-1}(Z \sm K)\leq C(\Upsilon)r^{n-1} (ff^{2} + \xi \sqrt{f})$.
 \item The union of the connected components of $Q_{2r} \sm Z$ which contain $Q_r \cap \{u>0\}\sm (S_+ \cup S_-)$ is compactly contained in $Q_{2r}$.
\end{enumerate}

We now construct a one-parameter family of competitors $v_\s$ for $u$, based on the diffeomorphisms $\Psi_\s(x)=x + \s T(x)$, where $T(x) = \phi(x') \zeta(x_n)\e_n$ and $\zeta$ is a compactly supported cutoff which vanishes outside $(-r,r)$, is $1$ on $(-r/2,r/2)$, and has derivative bounded by $4/r$. Notice that $\|T\|_{C^1}\leq C$. Consider the set $\Psi_\s (S_+)$, together with the $SBV$ function $w_\s = u\circ \Psi_\s^{-1} 1_{\Psi_\s(S_+)}$. Proceeding similarly to the proof of Lemma \ref{lem:EL}, we have that
\begin{equation}\label{eq:harm2i7}
 \left|\int_{\Psi_\s(S_+)}|\n w_\s|^2 d\cL^n -\int_{S_+}|\n u|^2 d\cL^n\right| \leq C \s \int_{S_+}|\n u|^2 d\cL^n \|T\|_{C^1},
\end{equation}
as well as the estimate 
\begin{align}
 &\left|\int_{Q_r\cap K_{w_\s}} \overline{w}_\s^2 + \underline{w}_\s^2 d\cH^{n-1} -\int_{Q_r \cap K_{w_0}} \overline{w}_0^2 + \underline{w}_0^2 d\cH^{n-1}  - \s \int_{Q_r \cap K_{w_0}}\dvg\nolimits^{\nu_x} T \1\overline{w}_0^2 + \underline{w}_0^2\2 d\cH^{n-1} \right|\nonumber\\
 &\qquad\leq C\s^2 r^{n-1} \|T\|_{C^1}. \label{eq:harm2i8}
\end{align}

To actually construct $v_\s$, we examine the open set $\Psi_\s (S_+) \cap S_-$. There will be two cases, depending on whether (case 1) $\cH^{n-1}(\Psi_\s(S_+) \cap \p S_-)\leq \cH^{n-1}(\p \Psi_\s(S_+) \cap S_-)$ or (case 2) $\cH^{n-1}(\Psi_\s(S_+) \cap \p S_-)> \cH^{n-1}(\p \Psi_\s(S_+) \cap S_-)$. Define the $SBV$ function $v_\s$ on $\cL^n$-a.e. point on $Q_r$ by
\[
 v_\s (x) =\begin{cases}
            u(x) & x\in S_- \sm \Psi_\s (S_+)\\
            w_\s (x) & x\in \Psi_\s (S_+) \sm S_-\\
            0 & x\in Q_r\sm (S_- \cup \Psi_\s (S_+))\\
            u(x) & x\in \Psi_\s (S_+) \cap S_- \text{ and in case 1}\\
            w_\s (x) & x\in \Psi_\s (S_+) \cap S_- \text{ and in case 2}.
           \end{cases}
\]
This function coincides with $u 1_{\R^n \sm S_B}$, where $S_B$ is the set given in item $(4)$ above, on the boundary of $Q_r$, and so we extend $v_\s$ in this way. By $(4)$, this is a compactly supported perturbation of $u$ on $Q_{2r}$, and so we may use it as a competitor. Let us assume for now that we are in case 1; case 2 works similarly. We will now somewhat laboriously compute how the integral of $\vu_\s^2 +\vd_\s^2$ on $K_{v_\s}$ compares to the measure $\m$ on $Q_r$, with the goal of showing that the difference is essentially the integral of $\s \dvg^{\nu_x} T$ against $\m^+$ on $\G^+$, the upper Lipschitz graph, up to small errors. We will omit the $\s$ subscript on $v$.

Let us think of $v = u 1_{S_-} + w_\s 1_{V}$ (where $V=\Psi_\s (S_+) \sm S_-$), and estimate its jump set accordingly (making use of the $BV$ product rule). First of all, on $J_v \cap S_-^{(1)}$ we have $J_v=J_u$ and $\vu=\uu$, $\vd=\ud$, so
\begin{equation}\label{eq:harm2i1}
 \int_{J_v \cap S_-^{(1)}} \vu^2 +\vd^2 d\cH^{n-1} =  \m (S_-^{(1)}).
\end{equation}
For $J_v \cap S_-^{(0)}$, there are a few cases. On $J_v \cap S_-^{(0)} \cap V^{(0)}$, we must have $\vu=\vd=0$ for $\cH^{n-1}-$a.e. point (contradicting being in $J_v$), which
means this set is $\cH^{n-1}$ negligible. On $J_v \cap S_-^{(0)} \cap V^{(1)}$, $J_v=J_{w_\s}$ and $\vu=\bar{w}_\s$, $\vd = \underline{w}_\s$ so we have
\begin{equation}\label{eq:harm2i2}
 \int_{J_v \cap S_-^{(0)} \cap V^{(1)}} \vu^2 +\vd^2 d\cH^{n-1} =  \int_{J_{w_\s} \cap S_-^{(0)} \cap V^{(1)}} \bar{w}_\s^2 +\underline{w}_\s^2 d\cH^{n-1}.
\end{equation}
Finally, on $J_v \cap S_-^{(0)} \cap V^{(1/2)}$, we have that $\vu = \bar{w}_\s$, while $\vd=0=\underline{w}_\s$, so we still obtain
\begin{equation}\label{eq:harm2i3}
 \int_{J_v \cap S_-^{(0)} \cap V^{(1/2)}} \vu^2 +\vd^2 d\cH^{n-1} =  \int_{J_{w_\s} \cap S_-^{(0)} \cap V^{(1/2)}} \bar{w}_\s^2 +\underline{w}_\s^2 d\cH^{n-1}.
\end{equation}

This leaves only the points in $J_v \cap S_-^{(1/2)}$. For points in the set $J_v \cap S_-^{(1/2)} \cap V^{(0)}$, we have that either they are in $Z\sm K$, or else they are in $J_u$ and $\vu=\uu,\vd=0$. This implies
\begin{equation}\label{eq:harm2i4}
 \int_{J_v \cap S_-^{(1/2)} \cap V^{(0)}} \vu^2 +\vd^2 d\cH^{n-1} - \m(S_-^{(1/2)} \cap V^{(0)}) \leq C\cH^{n-1}(Z\sm K).
\end{equation}
Finally, this leaves $J_v \cap S_-^{(1/2)} \cap V^{(1/2)}$. Here we notice that one of $\vu,\vd$ coincides with $\bar{w}_\s$, while the other with $\overline{u1_{S_-}}$, and so our integral can be separated into these two terms. The second we leave essentially alone, just noting that the boundary of $S_-$ lies outside of $K$ only on a small set:
\begin{equation}\label{eq:harm2i5}
 \left|\int_{J_v \cap S_-^{(1/2)} \cap V^{(1/2)}} \overline{u1_{S_-}}^2 d\cH^{n-1} -\int_{J_u \cap S_-^{(1/2)} \cap V^{(1/2)}} \overline{u1_{S_-}}^2 d\cH^{n-1}\right| \leq C\cH^{n-1}(Z\sm K).
\end{equation}
For the first, we use that we are in case 1 to say that $\cH^{n-1}(S_-^{(1/2)} \cap V^{(1/2)})\leq \cH^{n-1}(\Psi_\s(S_+)^{(1/2)}\cap S_-^{(0)})$. We also observe that the oscillation of $w_\s$ over $\Psi_\s(S_+)$ is the same as the oscillation of $u$ over $S_+$, which is controlled by $Cr^{1/2}/\xi$. We therefore obtain that
\begin{equation}\label{eq:harm2i6}
 \int_{J_v \cap S_-^{(1/2)} \cap V^{(1/2)}} \bar{w}_\s^2 d\cH^{n-1}\leq \int_{J_{w_\s}\cap S_-^{(1)} \cap \Psi_\s(S_+)^{(1/2)}} \bar{w}_\s^2 d\cH^{n-1} + C\frac{\sqrt{r}}{\xi} r^{n-1}.
\end{equation}

Combining all of these inequalities \eqref{eq:harm2i1},\eqref{eq:harm2i2},\eqref{eq:harm2i3},\eqref{eq:harm2i4},\eqref{eq:harm2i5}, and \eqref{eq:harm2i6} together gives
\begin{align*}
 \int_{J_v \cap Q_r} \vu^2 +\vd^2 d\cH^{n-1} &\leq \int_{J_u \cap Q_r} \overline{u1_{S_-}}^2 + \underline{u1_{S_-}}^2 d\cH^{n-1} + \int_{J_{w_\s} \cap Q_r} \overline{w}_\s^2 + \underline{w}_\s^2 d\cH^{n-1}\\
 &\qquad+ C(\Upsilon)[ff^{2} +\xi \sqrt{ff} + \frac{\sqrt{r}}{\xi}] r^{n-1}.
\end{align*}
Applying our prior estimate \eqref{eq:harm2i8} on the second term on the right and combining the integrals gives
\begin{align}
 \int_{J_v \cap Q_r} \vu^2 +\vd^2 d\cH^{n-1} &\leq \m(Q_r) + \s \int_{J_{w_0} \cap Q_r}\dvg\nolimits^{\nu_x}T ( \overline{w}_0^2 + \underline{w}_0^2 )d\cH^{n-1} \nonumber\\
 &\qquad+ [C\s^2+C(\Upsilon)(ff^{2} +\xi \sqrt{ff} + \frac{\sqrt{r}}{\xi}) ] r^{n-1} \label{eq:harm2i11}.
\end{align}
Now, $J_{w_0}$ coincides with $\Theta^+$ except on the union of $Z\sm K$ and $(J_u\cap \p S_+) \triangle \Theta^+$, up to $\cH^{n-1}$ negligible sets, where $\Theta^+$ is as in Remark \ref{rem:lipest}. Let us estimate this second set. First of all, the set $\Theta^+ \sm \p S_+$ has the property that
\[
 \pi(\Theta^+ \sm \p S_+) \ss \pi(Z\sm K).
\]
Indeed, the upward line segment $\{x+te_n: t>0\}$ for each $x\in \Theta^+ \sm \p S_+$ avoids $K$, but must intersect $\p S_+$ before exiting the cylinder $Q_r$. As $\Theta^+$ lies on a $1-$Lipschitz graph, this implies that
\begin{equation}\label{eq:harm2i9}
 \cH^{n-1}(\Theta^+ \sm \p S_+)\leq 2 \cL^{n-1}(\pi(\Theta^+ \sm \p S_+))\leq 2 \cH^{n-1}(Z\sm K).
\end{equation}
On the other hand
\[
\p^* S_+ \cap J_u\sm \Theta^+ \ss (J_u \sm (\Theta^+\cup \Theta^-)) \cup (\p^* S_+ \cap \Theta^-\sm \Theta^+) , 
\]
with the first of these having measure bounded by $C(\Upsilon)ff$ from Remark \ref{rem:lipest} (and Theorem \ref{thm:exbyflat2}). For a point $x$ in the second set, $\p^* S_+ \cap \Theta^-\sm \Theta^+$, we must (by definition of $\Theta^-$) have that the Lebesgue density of $\{u>0\}$ at $x$ is $1/2$; this implies that the density of $S_-$ at $x$ must be $0$ (the density of $S_+$ being $1/2$). Thus $x\in \Theta^- \sm \p^* S_-)$; from \eqref{eq:harm2i9}, we infer that
\[
 \cH^{n-1}(\p^* S_+ \cap \Theta^-\sm \Theta^+)\leq 2\cH^{n-1}(Z\sm K).
\]
Putting everything together, we have shown that
\begin{equation}\label{eq:harm2i10}
 \cH^{n-1}((J_u\cap \p S_+) \triangle \Theta^+)\leq C(\Upsilon)[ff+ff^2 + \xi \sqrt{ff}]r^{n-1}.
\end{equation}
Note that several times we used the fact that for a set $E$ whose boundary is countably $\cH^{n-1}$-rectifiable, we have $\cH^{n-1}(\p E\sm \p^* E)=0$.

Thus if in \eqref{eq:harm2i11} we integrate the term with the tangential divergence of $T$ on $\Theta^+$ instead of on $J_{w_0}$, we incur an error controlled by $C(\Upsilon) \s [ff + \xi\sqrt{ff}] r^{n-1}$ (using the fact that $\|T\|_{C^{1}} \leq C$):
\begin{align*}
 \int_{J_v \cap Q_r} \vu^2 +\vd^2 d\cH^{n-1} &\leq \m(Q_1) + \s \int_{\Theta^+ \cap Q_r}\dvg\nolimits^{\nu_x}T u_+^2 d\cH^{n-1} \\
 &\quad + C(\Upsilon)r^{n-1}[ff^{2} +\xi \sqrt{ff} + \frac{\sqrt{r}}{\xi} +\s^2 + \s ff].
\end{align*}
We once again used the oscillation bound on $u$ over $S_+$. Now apply $v$ as a competitor in the quasiminimizer inequality, to obtain (using also \eqref{eq:harm2i7}) that
\[
 0\leq \s \int_{\Theta^+ \cap Q_r}\dvg\nolimits^{\nu_x}T u_+^2 d\cH^{n-1} + C(\Upsilon)r^{n-1}[ff^{2} +\xi \sqrt{ff} + \frac{\sqrt{r}}{\xi} +\s^2 + \s ff] + \L r^{n-1+s}.
\]
Set $\xi = [\Upsilon^{-1} ff]^{\frac{3}{2}}\geq J(u,24r)$  and then  $\s = ff $, which gives
 \[
  \frac{1}{r^{n-1}}\int_{\Theta^+ \cap Q_r}\dvg\nolimits^{K}T u_+^2 d\cH^{n-1} \geq - C(\Upsilon)[ff +  \frac{\sqrt{r}}{ff^{5/2}}]\geq -C(\Upsilon)ff.
\]

To obtain the final conclusion, we recall from the proof of Lemma \ref{lem:harm1} that
\[
 \int_{\G^+\cap Q_r} \dvg\nolimits^{\nu_x}T d\cH^{n-1} = \int_{D_r} \frac{\n g_+ \cdot \n \phi}{\sqrt{1+|\n g_+|^2}} d\cL^{n-1}.
\]
We estimate
\begin{align*}
 \left| \frac{1}{\sqrt{1+|\n L_+|^2}} -\frac{1}{\sqrt{1+|\n g_+|^2}}\right| &\leq \frac{|\n L_+| |\n g_+ - \n L_+| }{(1+|\n L_+|^2 )^{\frac{3}{2}}} + C |\n g_+ - \n L_+|^2 \\
 &\leq \frac{h}{r} |\n g_+ -\n L_+| + C |\n g_+ - \n L_+|^2
\end{align*}
using a Taylor expansion for $\frac{1}{\sqrt{1 + a^2}}$ around $a = |\n L_+|$. This gives (using that $\|\phi\|_{C^1}\leq 1$)
\begin{align*}
 \big|\int_{D_r}&\frac{\n g_+ \cdot \n \phi}{\sqrt{1+|\n L_+|^2}} - \frac{\n g_+ \cdot \n \phi}{\sqrt{1+|\n g_+|^2}} d\cL^{n-1}\big|\\
 &\leq \frac{h}{r}\int_{D_r}|\n g_+||\n g_+-\n L_+|d\cL^{n-1} + C\int_{D_r}|\n g_+ -\n L_+|^2 d\cL^n\\
 &\leq \frac{h}{r}\int_{D_r}|\n L_+| |\n g_+ -\n L_+|d\cL^{n-1} + C\int_{D_r}|\n g_+ -\n L_+|^2 d\cL^n\\
 &\leq C(\Upsilon)[\frac{h^2}{r^2} \sqrt{ff} + ff]r^{n-1},
\end{align*}
where we used Proposition \ref{prop:liptoplane} and the estimate $|\n L_+|\leq \frac{h}{r}$ in the last step. Note that as planes are harmonic, we have
\[
 \int_{D_r} \n L_+ \cdot \n \phi d\cL^{n-1} =0.
\]
Applying these estimates, we finally arrive at
\[
  \frac{1}{\sqrt{ff}r^{n-1}}\int_{D_r}[\n g_+ -\n L_+]\cdot \n \phi u_+^2 d\cH^{n-1} \geq - C(\Upsilon)(\sqrt{ff} + \frac{h^2}{r^2}).
\]
The other inequality follows by using $-\phi$ instead.
\end{proof}

\section{Flatness Improvement}\label{sec:flatness}

\begin{theorem}\label{thm:flatimp} Let $u\in \sQ(Q_{r},\L,r_0)$ and $2r<r_0$. For each $\Upsilon>0$ and $\t \in (0,\t_0)$, there is a constant $\e_H=\e_H(\Upsilon,\t)$ such that if
\[
 f(u,r) + E(u,r) + \L r^s + r \leq \e_H, \qquad E(u,r) + \sqrt{\L r^s} + J^{\frac{2}{3}}(u,r) + r^{1/7} \leq  \Upsilon ff(u,r),
\]
then for some $\nu$ with $1 - (\nu\cdot e_n)^2 \leq C_H ff(u,r)$, and  $x$ with $x=|x|\nu$ and $|x|^2 \leq C_H ff(u,r)r^2$,
\begin{equation}\label{eq:flatimpc}
  ff(u,x,\t r,\nu)\leq C_H \t^2 ff(u, r).
\end{equation}
The constant $C_H$ is universal. If one of $u(0,\pm r/2)=0$, the same holds with $f$ in place of $ff$.
\end{theorem}

\begin{remark}\label{rem:flatimp}We note that in the conclusion of the theorem, the fact that $x=|x|\nu$ may be replaced with $x= |x|e_n$ to obtain an equivalent statement; this follows from the fact that the cylinders $Q_{\t r, \nu}$ centered around these two points are contained in each other's double. We will prove the theorem with $x = |x|e_n$, which simplifies notation slightly.
\end{remark}

\begin{proof}
The argument is by contradiction. We assume that \eqref{eq:flatimpc} fails for some $\t$ and $\Upsilon$, which means there are sequences $u_k, r_k, \L_k. \e_k$ with $r_k\searrow 0$, $f(u_k,r_k)\rightarrow 0$,
\[
 ff(u_k,r_k) = \e_k^2 \rightarrow 0,
\]
\[
 E(u_k,r_k) + \sqrt{\L_k r_k^s} + J^{\frac{2}{3}}(u_k,r_k) +r_k^{\frac{1}{7}} \leq \Upsilon \e_k^2,
\]
and
\[
 \inf_{x,\nu}ff(u_k,x,\t r_k, \nu) \geq C_H \t  ff(u_k,r_k).
\]
We reduce to a subsequence for which $\tilde{u}_k(x)=u_k(r_k x)$ converges in $L^1$ to a function $u_\8$ which is locally constant on $Q_1 \sm \pi$. Let $\nu_\pm = \nu_\pm[u_k,r]$, $L_{\pm,k}=L_\pm[u_k,r]$ and $h_k=h(u_k,r)$. We may apply Theorem \ref{thm:exbyflat2} (on $Q_{r_k/4}$) and then Theorem \ref{thm:lip} (on $Q_{r_k/24}$, with $\nu_* = \nu_+$) to this sequence, extracting the Lipschitz functions $g_{\pm,k}$ defined on $D_{r_k/24}$ and satisfying (using Proposition \ref{prop:liptoplane})
\[
 \frac{1}{r_k^{n-1}}\int_{D_{r_k/24}} |\n g_{+,k} - \n L_{+,k}|^2 + |\n g_{-,k}-\n L_{-,k}|^2 d\cL^{n} \leq C \e_k^2,
\]
\[
 \frac{1}{r_k^{n+1}}\int_{D_{r_k/24}}|g_{+,k}-L_{+,k}|^2 + |g_{-,k}- L_{-,k}|^2 d\cL^{n-1} \leq C\e_k^2,
\]
and
\[
 \cH^{n-1}(Q_{r_k/24} \cap \G^\pm_k \sm K_{u_k}) + \cH^{n-1}(Q_{r_k/24} \cap K_{u_k} \sm (\G^+_k \cup \G^-_k)) \leq C\e_k^2 r_k^{k-1}.
\]

Define $f_{+,k}(x) = \frac{(g_{+,k}-L_{+,k})(r_k x)}{\e_k r_k}$, which are defined on $D_{1/24}$, and have
\[
 \int_{D_{1/24}} |f_{+,k}|^2 + |\n f_{+,k}|^2 d\cL^{n-1} \leq C.
\]
It follows that we may extract $f_{+,k} \rightarrow f_{+}$ weakly in $H^1$ and strongly in $L^2$. We will now show that $f_+$ is harmonic.

Indeed, if both $u_\8(0, \pm 1/2)\neq 0$, from Lemma \ref{lem:harm2}, we have that for each nonnegative $\phi\in C^1_c (D_{1/24})$ with $\|\phi\|_{C^1}\leq 1$, we have (after using $r_k\phi(x/r_k)$ as the test function)
\[
 \left|\int_{D_{1/24}} \n f_{+,k} \cdot \n \phi d\cL^{n-1} \right| \leq C(\Upsilon)(\e_k + \frac{h_k^2}{r_k^2})\rightarrow 0. 
\]
Taking the limit and using the weak convergence of the gradient, this gives
\[
 \int_{D_{1/24}}\n f_+ \cdot \n \phi  d\cL^{n-1} = 0,
\]
which implies that $f_+$ is harmonic. If one of $u_\8(0,\pm 1/2)=0$, we obtain the same conclusion from Lemma \ref{lem:harm1} instead (using that $g_- = g_+$). Likewise, $f_-$ is also harmonic.

Define the affine function $q_+(x)=f_+(0)+\n f_+(0)x$. By basic estimates on harmonic functions, we have that
\begin{equation}\label{eq:flatimpi1}
 \int_{D_\t} |q_+ -f_+|^2 d\cL^{n-1} \leq C_0(n)\t^{n+3} \int_{D_{1/24}}|\n f_+|^2 d\cL^{n-1} \leq C_0(n)\t^{n+1}.
\end{equation}
Likewise, we have that
\begin{equation}\label{eq:flatimpi2}
 \int_{D_{\t}} |q_- -f_-|^2 d\cL^{n-1} \leq C_0(n)\t^{n+3} \int_{D_{1/24}}|\n f_-|^2 d\cL^{n-1} \leq C_0(n)\t^{n+1}
\end{equation}
for the corresponding function $q_-$. Notice also that $|f_\pm (0)|,|\n f_\pm (0)|\leq C_0(n)$.

Now we use the fact that the conclusion of the theorem was assumed to fail, beginning with the case of when (along a subsequence) $ff(u_k,r_k)<C_1(n)\frac{h_k^2}{r_k^2}$ for a small but universal $C_1$ to be chosen momentarily. First, observe that for $\e_k$ small, the two affine functions 
\[
L_{\pm,k}^*(x'):= L_{\pm,k}(x') + \e_k r_k q_\pm (x'/r_k) = L_{\pm,k}(x') + \e_k r_k f_{\pm}(0) + \e_k \n f_\pm (0)\cdot x'
\]
 will not intersect over $D_{2 \t r_k}$: indeed,
\begin{align*}
 |\n L_{+,k}^*|&\leq \e_k |\n q_+| + |\n L_{+,k}|\\
&\leq \e_k C_0(n) + |\n L_{+,k}| \\
&\leq 2 \frac{h_k}{r_k},
\end{align*}
where we used that $L_{\pm,k}$ are admissible for the definition of $ff$ and that $C_0^2(n) ff(u,r_k)r_k^2 \leq h_k^2$ if $C_1 \leq C_0^{-2}$. This gives that
\[
 \osc_{D_{2\t r_k}} L_{+,k}^* \leq (2\frac{h_k}{r_k}) (4 \t r_k) \leq  8 \t h_k \leq \frac{h_k}{1000},
\]
provided $\t$ is restricted to be small enough. Then as in addition
\[
 |L_{+,k}^*(0) -h_k| =\e_k r_k |f_+(0)| \leq (\frac{\sqrt{C_1} h_k}{r_k}) r_k C_0  \leq \frac{h_k}{1000}
\]
by choosing $C_1$ small, we see that $L_{+,k}^*$ remains trapped in $(\frac{499}{500}h_k,\frac{501}{500}h_k)$, and so certainly does not cross $L_{-,k}^*$.

\begin{figure}
	\centering
	\def\svgwidth{15cm}
	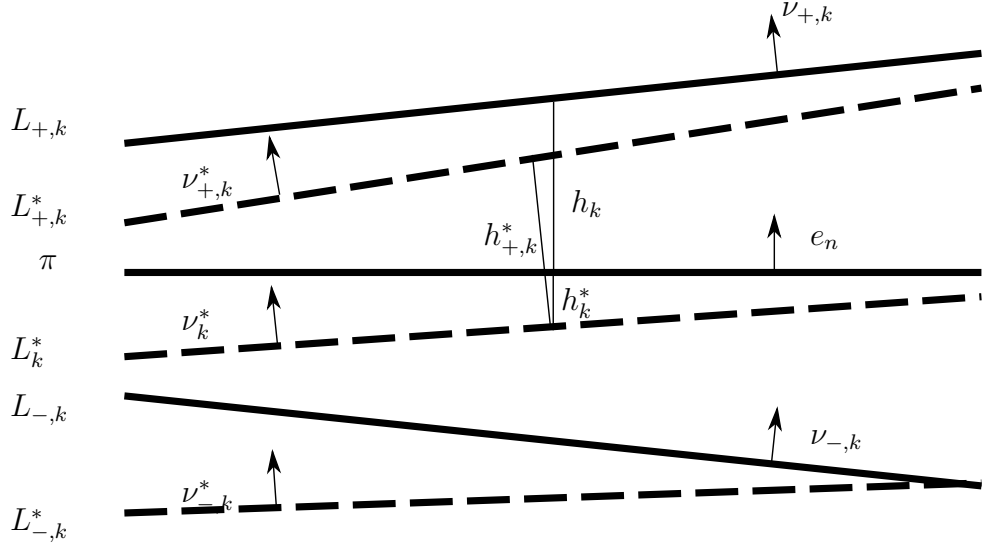
	\caption{The plane $\pi$ and the graphs of $L_{\pm,k}$ are shown in solid bold. The graphs of $L_{\pm,k}^*$, as well as their bisector the graph of $L_k^*$, are shown in dashed bold. The corresponding normal vectors are labeled. The thin lines represent the distances $h_k,h_k^*$, and $h_{+,k}^*$. Note that distances and angles are exaggerated for clarity; in reality the dashed planes would be much closer to the solid ones.} \label{fig:4}
\end{figure}

Define for each $k$ the function $L_k^* = \frac{L_{+,k}^*+L_{-,k}^*}{2}$, the value $h_k^* = L_k^*(0)$, and the corresponding plane $\pi^*_k = \{x=(x',L_k^*(x'))\}$, which is equidistant from the two planes $(x',L_{+,k}^*(x'))$ and $(x',L_{-,k}^*(x'))$. Set  $\nu_k^*$ to be the normal vector to this plane (see Figure \ref{fig:4} for a diagram of all of these planes and related quantities). We have that
\[
 |h^*_k|\leq \frac{1}{2}(|L_{+,k}^*(0) -h_k| + |L_{-,k}^*(0) -h_k|) \leq \frac{h_k}{500},
\]
while
\[
 1 - (e_n \cdot \nu_k^*)^2 = 1 - \1e_n \cdot \frac{(-\n L_k^*,1)}{\sqrt{1+|\n L_k^*|^2}}\2^2\leq C |\n L_k^*|^2 \leq C\e_k^2 (|\n q_+|^2 + |\n q_-|^2) \leq   C\e_k.
\]
It follows from the computation above that $(x',L_{+,k}^*(x'))$ (whose normal vector we denote $\nu_{+,k}^*$) is an admissible plane in the infimum in the definition of $ff(u_k,h^*_k e_n,\t r_k, \nu^*_k)$, which gives (from the contradiction assumption) that
\[
 \int_{Q_{\t r_k,\nu^*_k}(h_k^* e_n)\cap K_{u_k}} d_*^2(x,\nu_{+,k}^*,h^*_{+,k}) d\cH^{n-1} \geq \e_k^2 \t^2 C_H (\t r_k)^{n+1}.
\]
Here $h_{+,k}^*$ is the unique number such that
\[
h_k^* e_n + h_{+,k}^* \nu_{k}^* \in \{(x',L_{+,k}^*(x')) \}.
\]

We also have that $ee(u_k,h_k^* e_n, r_k/8,\nu_k^*,\nu_{k,+}^*)\leq C\e_k^2$ and similarly for $E$ and $J$. We may therefore apply Theorem \ref{thm:lip} (with $\nu_* = \nu_{k,+}^*$ on $Q_{12 \t r_k,\nu_k^*}(h_k^* e_n)$) and Remark \ref{rem:lipest} to $u_k$ to obtain sets $\Theta^{\pm,*}_k$; immediately replace them by $\Theta^{\pm,*}_k \cap \G_\pm^k$ while preserving the same notation. These are then subsets of $\G^\pm_k$,  and have the property that $\cH^{n-1}(Q_{2\t r_k}\cap[K_{u_k}\triangle (\Theta^{+,*}_k\cup \Theta^{-,*}_k)])$ is controlled by $\e_k^2 r_k^{n-1}$. Moreover, they each contain only points that are visible from above (or below) $\pi^*_k$.

Notice that all of $L_{\pm,k}^*, K_{u_k},\G^{\pm}_k,\Theta^{\pm,*}_k$ are contained in an $\e_k^{\frac{2}{n+1}}r_k$ neighborhood of the planes $\pi_{\pm,k}$ (this is from Proposition \ref{prop:fflatsame} and Lemma \ref{lem:gflat}), so we have
\[
 \frac{1}{r_k^{n+1}}\int_{Q_{\t r_k,\nu^*_k}(h_k^* e_n)\cap (\Theta^{+,*}_k\cup \Theta^{-,*}_k)} d_*^2(x,\nu_{+,k}^*,h^*_{+,k}) d\cH^{n-1} \geq \e_k^2 \t^{n+3} C_H - C(\Upsilon)\e_k^{2+\frac{2}{n+1}}.
\]
We now note that, for $x\in \Theta^{+,*}_k \sm \Theta^{-,*}_k$, we have that 
\begin{align*}
d_*(x,\nu_{+,k}^*,h^*_{+,k}) &\leq |\nu_{+,k}^* \cdot (x - L_{+,k}^* (x')e_n)|\\
& \leq  |x - L_{+,k}^* (x')e_n| \\
& =  |g_{+,k}(x') - L_{+,k}^*(x')|
\end{align*}
from the definition. Likewise, for $x\in \Theta^{-,*}_k \sm \Theta^{+,*}_k$ we have 
\[
d_*(x,\nu_{+,k}^*,h^*_{+,k}) \leq |g_{-,k}(x') - L_{-,k}^*(x')|,
\]
while for $x \in \Theta^{-,*}_k \cap \Theta^{+,*}_k$,
\[
d_*(x,\nu_{+,k}^*,h^*_{+,k}) \leq |g_{+,k}(x') - L_{+,k}^*(x')| + |g_{-,k}(x') - L_{-,k}^*(x')|.
\]
After changing variables and accounting for $\G^{\pm}_k\sm \Theta^{\pm,*}_k$ as before, this implies that
\begin{align*}
 \frac{1}{r_k^{n+1}}\int_{D_{2\t r_k}} &|g_{+,k} - L_{+,k}^*|^2 +  |g_{-,k} - L_{-,k}^*|^2 d\cL^{n-1} \\
 &\geq \frac{1}{2 r_k^{n+1}}\int_{D_{2\t r_k}} |g_{+,k} - L_{+,k}^*|^2 \sqrt{1+|\n g_+|^2} +  |g_{-,k} - L_{-,k}^*|^2 \sqrt{1 + |\n g_-|^2} d\cL^{n-1} \\
 &\geq \frac{1}{2 r_k^{n+1}}\int_{Q_{\t r_k,\nu^*_k}(h_k^* e_n)\cap \Theta^{+,*}_k} |g_{+,k}(x') - L_{+,k}^*(x')|^2 d\cH^{n-1}  \\
 &\quad + \frac{1}{2 r_k^{n+1}}\int_{Q_{\t r_k,\nu^*_k}(h_k^* e_n)\cap  \Theta^{-,*}_k}  |g_{-,k} - L_{-,k}^*|^2  d\cH^{n-1} \\
 &\geq \frac{1}{2r_k^{n+1}}\int_{Q_{\t r_k,\nu^*_k}(h_k^* e_n)\cap (\Theta^{+,*}_k\cup \Theta^{-,*}_k)} d_*^2(x,\nu_{+,k}^*,h^*_{+,k}) d\cH^{n-1} \\
 &\e_k^2 \t^{n+3} C_H - C(\Upsilon)\e_k^{2+\frac{2}{n+1}}.
\end{align*}
Rescaling now gives
\[
 \int_{D_{2\t}} |f_{+,k} - q_+|^2 + |f_{-,k} - q_-|^2 d\cL^{n-1} \geq \t^{n+3} C_H  - C(\Upsilon)\e_k^{\frac{2}{n+1}},
\]
and taking the limit as $k\rightarrow \8$ gives a contradiction if $C_H$ is large enough.

We proceed to the case of $ff(u_k,r_k)\geq C_1(n)\frac{h_k^2}{r_k^2}$. We replace $h_k$ by $0$, $\nu_+$ by $e_n$, and $L_{\pm,k}$ by $0$ (essentially $ff$ by $f$) in this case, as now $f(u_k,r_k)$ and $ff(u_k,r_k)$ are comparable. In the limit, we now have the additional information that $f_+\geq f_-$. As before, we consider the tangent planes $q_\pm$, but now we perform an additional estimate to show
\[
 |\n f_+(0) -\n f_-(0)|\leq C(n) [f_+(0)-f_-(0)].
\]
To see this, notice that $f_+ -f_-$ is a nonnegative harmonic function, so from Harnack's inequality,
\[
 \sup_{D_{1/24}} f_+ -f_- \leq C(n)[f_+(0)-f_-(0)].
\]
Then the bound on the gradients comes from interior estimates.

We construct the planes $\pi_{\pm,k}^*$ as previously, with $L_{\pm,k}^*(x')= r_k \e_k f_\pm (0) + \e_k \n f_\pm(0) x'$. We now employ our new estimate to give
\[
 |\n L_{+,k}^* -\n L_{-,k}^*|\leq C(n) \frac{|L_{+,k}^*(0)-L_{-,k}^*(0)|}{r_k}.
\]
In particular, for $\t$ small enough this implies that $L_{k,\pm}^*$ do not cross on $D_{\t r_k}$ (or are both identically $0$):
\begin{align*}
|L{+,k}^*(x) - L_{-,k}^*(x)| &\geq |L{+,k}^*(0) - L_{-,k}^*(0)| -  \t r_k |\n L_{+,k}^* -\n L_{-,k}^*| \\
&\geq |L{+,k}^*(0) - L_{-,k}^*(0)| (1 - C\t).
\end{align*}
We may construct $pi_k^*$, their bisector, as before. We also have that
\[
 |L_{+,k}^*(0)-L_{-,k}^*(0)| \leq C(n) \e_k r_k.
\]
Up to restricting to $\t$ smaller, the vectors $\nu_{\pm,k}^*$ are admissible for the infimum in 
\[
ff(u_k,h_k^* e_n, \t r_k,\nu_k^*),
\]
and so we may obtain the same contradiction.
\end{proof}

\section{The Iteration Procedure}\label{sec:iteration}

The goal of this section is an improvement of flatness lemma, which will be iterated in the following section. In the proof, there will essentially be three possibilities. First, either the relative jump of $u$ or the quasiminimality constant is larger than the energy and flatness, so we simply need to zoom in more. Second, the energy may be much larger than the other quantities, in which case we are in a position to apply Lemma \ref{thm:enimp}. Finally, if none of the above occurs, we will improve the flatness using Theorem \ref{thm:flatimp}. 

First, we discuss the improvement of the jump of $u$.  The following lemma shows how to improve the estimate on the jump:
\begin{lemma}\label{lem:jump}Let $u\in \sQ(Q_{r},\L,2r)$. Then for every $\t$ small enough, there is a constant $\e_J(\t)$ such that if
\[
 f(u,r)+E(u,r)+r + |x| + \|e_n-\nu\| + \L r^s\leq \e_J,
\]
then
\[
 J(u,x,\t r,\nu) \leq 2\sqrt{\t} J(u,r),
\]
\end{lemma}

\begin{proof}
 Fix $\t$, assume  $\e_J$ is small enough that $f_\8(u,r) \leq \frac{\t}{8}$, and let $x_+,x_-$ be the points on which $J(u,x,\t r,\nu)$ is evaluated. Applying Lemma \ref{lem:intreg}, we have that
 \[
  \osc_{Q_{r/2}\cap \{|x_n|\geq \t r/4\}} u \leq C(\t) \sqrt{r (E(u,r) +\L r^s)}\leq C(\t) \sqrt{r \e_j},
 \]
and likewise for the lower part of the cylinder. Now, choose $\e_J$ small enough that $ \sqrt{\e_j} \leq \frac{1}{4C(\t)J(u,r)}$ (as we know that $J(u,r)$ is bounded by a universal constant from above). Then we have that
\[
 |u(0,r/2)-u(x_+)|\leq C(\t)\sqrt{r\e_J }\leq \frac{\sqrt{r}}{4}J^{-1}(u,r).
\]
Then
\[
|u(x_+)-u(x_-)| \geq sqrt{r}J^{-1}(u,r) -  |u(0,r/2)-u(x_+)|-|u(0,-r/2)-u(x_-)|\geq \frac{\sqrt{r}}{2J(u,r)}.
\]
This gives that
\[
 J(u,x,\t r,\nu) \leq 2\sqrt{\t} J(u,r), 
\]
as claimed.
\end{proof}

We note for future reference that if one of $u(0,\pm r/2)=0$, then trivially $J(u,r)\leq C\sqrt{r}$.

\begin{lemma}\label{lem:iterate} Let $u\in \sQ(Q_{r},\L,2r)$. Then for any $\t>0$ sufficiently small, there are $\e_I=\e_I(\t)$ and $\l_I,C_I$ such that if
\begin{equation}\label{eq:iterateh1}
  f(u,r) + E(u,r) + r  + \L r^s\leq \e_I,
\end{equation}
then either
\begin{equation}\label{eq:iteratec1}
 ff(u,\t r) + \l_I E(u,\t r) \leq C_I \1 \sqrt{\L} r^{s/2} + J^{\frac{2}{3}}(u,r) + r^{\frac{1}{7}}\2,
\end{equation}
\begin{equation}\label{eq:iteratec2}
 ff(u,\t_1 r) + \l_I E(u,\t_1 r) \leq \t_1^{1/2}\1ff(u, r) + \l_I E(u, r)\2,
\end{equation}
or else
\begin{equation}\label{eq:iteratec3}
 ff(u,x,\t_2 r,\nu) + \l_I E(u,x, \t_2 r,\nu) \leq \t_2^{1/2}\1ff(u, r) + \l_I E(u, r)\2,
\end{equation}
where $x = |x|\nu $, $|x|^2 \leq C_I r^2 ff(u,r)$, $1-(e_n \cdot \nu)^2 \leq C_I ff(u,r)$. Here $\t_1 = \t/50$ and $\t_2 = \t^{\frac{1}{3(n-1)}}$. If one if $u(0,\pm r/2)=0$, the same holds with $f$ in place of $ff$.
\end{lemma}

\begin{proof}
The proof is identical to that in \cite[Lemma 8.20]{AFP}, after noting that if the assumption on $J(u,r)$ is not satisfied in Theorem \ref{thm:flatimp} or Corollary  \ref{cor:enimp}, then \eqref{eq:iteratec1} holds automatically.  
\end{proof}

\section{Proof of Main Theorem, With Extra Assumptions}\label{sec:mainpf}

In this section, we use Lemma \ref{lem:iterate} to prove that near a point where the flatness, energy, and relative jump of $u$ are small enough, $K_u$ is given by a pair of $C^{1,\a}$ graphs.

\begin{theorem}\label{thm:badmain}Let $u\in \sQ(Q_{r},\L,2r)$, and assume $0\in K$. Then there are universal $\e_M,\eta_M>0$ such that if
\[
 f(u,r)+E(u,r)+J(u,r)+r + \L r^s\leq \e_M,
\]
then $K\cap Q_{\eta_M r}$ coincides with the graphs of the Lipschitz functions $g_+,g_-$, and moreover these functions are actually $C^{1,\a}$ for some $\a>0$, with
\[
 \|g_+\|_{C^{1,\a}(D_{\eta_M r})} +\|g_-\|_{C^{1,\a}(D_{\eta_M r})} \leq 1.
\]
\end{theorem}

\begin{proof}
 \textbf{Part 1.}
 
Assume in this part that $u(0,\pm r/2)\neq 0$. We attempt to iterate Lemma \ref{lem:iterate} at $0$, in the following way. Fix $\t$ small enough that if we apply Lemma \ref{lem:iterate} with this $\t$, we have that $\t,\t_1,\t_2\leq T^{-2}$, for a universal $T$ to be chosen, and apply it once. This gives a cylinder $Q_{\r_1 r,\nu_1}(x_1)$, with $\r_1 \in \{\t,\t_1,\t_2\}$, on which one of the conclusions holds. Now if $|x_1|\leq \s \r_1 r$ (with $\s=\frac{1}{10}\min\{\t,\t_1,\t_2\}$), and all of the hypotheses of the lemma are met, apply it again to obtain another cylinder $Q_{\r_2 r,\nu_2}(x_2)$, and so forth. Either we may do this indefinitely (in which case see below), or else for some $k$, we fail to proceed to the $k+1$ step. We explain what to do in this case.
 
 Let us first explain that the jump $J(u,r)$ decays geometrically as long as we can keep applying the lemma. In particular, we have that
 \begin{equation}\label{eq:badmaini5}
  f(u,x_j,\r_j r,\nu_j) + E(u,x_j,\r_j r, \nu_j) \leq \e_I,
 \end{equation}
Then from Lemma \ref{lem:jump}, we have that
\[
 J(u,x_j,\r_j r,\nu_j)\leq \frac{2}{T} J(u,x_{j-1},\r_{j-1} r,\nu_{j-1})\leq \1\frac{2}{T}\2^j \e_M.
\]
Using that $\r_j\geq \s^j$, this gives that
\begin{equation}\label{eq:badmaini1}
 J(u,x_j,\r_j r,\nu_j)\leq \r_j^\a \e_M
\end{equation}
for some small $\a=\a(\t)$.
 
 Likewise, we may verify by induction that as long as the iteration continues,
 \[
  ff(u,x_j,\r_j r,\nu_j) + \eta_I E(u,x_j,\r_j r,\nu_j) \leq C(\t)\r_j^\a \e_M.
 \]
 Indeed, if this holds for $j$, and either \eqref{eq:iteratec2} or \eqref{eq:iteratec3} applies at the $j$-th step, then this follows directly. On the other hand, if it is \eqref{eq:iteratec1} which applies, then
\begin{align*}
ff(u,x_{j+1},\r_{j+1} r,\nu_{j+1}) + \eta_I E(u,x_{j+1},\r_{j+1} r,\nu_{j+1}) &\leq C (\sqrt{\L \r_j^s r^s} + J^{\frac{2}{3}}(u,x_j,\r_j r,\nu_j)  + (\r_j r)^{\frac{1}{7}} ) \\
& \leq C \r_j^\a \e_M,
\end{align*}
using \eqref{eq:badmaini1}.

 Now, to obstruct us from going further, one of two things must have happened: either $f(u,x_k,\r_k r,\nu_k)\geq \e_I$ (case 1), or we could apply Lemma \ref{lem:iterate}, but the resulting $x_{k+1}$ has $|x_{k+1}|\geq \s \r_{k+1}r$ (case 2). If we are in case 1, we may choose $\e_M$ so small that the two properties
 \[
  ff(u,x_k,\r_k r,\nu_k)\leq C(\t)\r_k^\a \e_M \leq C\e_M
 \]
and
 \[
 C(\t) \e_I > f(u,x_k,\r_k r,\nu_k) \geq  \e_I
 \]
imply that $K \cap Q_{\r_k r/2,\nu_k}(x_k)$ lies in two separate, disjoint neighborhoods of the planes $\pi_\pm (u,x_k,\r_k r, \nu_k)$ (optimal for $ff(u,x_k,\r_k r,\nu_k)$), with the distance between these neighborhoods at least $\frac{1}{4} \r_k r \e_I^{\frac{1}{2}}$ (this is possible by using Proposition \ref{prop:fflatsame}). We also have that $u=0$ on $\{x:\nu_k\cdot (x-x_k)=0\} \cap Q_{\r_k r/2,\nu_k}(x_k)$ (this is by Lemma \ref{lem:zerobetween}). Finally, $0$ lies in one of these two neighborhoods (say the upper one), and so $|(x_k)_n|$ is bounded from above and below by $\r_{k}r \e_I^{\frac{1}{2}}$. Then if we set $x_{k+1}=0$, $\nu_{k+1}=\nu_{+,k}=\nu_+[u,x_k,\r_k r,\nu_k]$, and $\r_{k+1}r = \frac{1}{2} |(x_k)_n|$, we have that
\[
 f(u,x_{k+1},\r_{k+1}r,\nu_{k+1}) + E(u,x_{k+1},\r_{k+1}r,\nu_{k+1})+ J(u,x_{k+1},\r_{k+1}r,\nu_{k+1}) \leq C(\t)\r_{k+1}^\a \e_M
\]
and $u$ vanishes on one side of the cylinder. Take this information and proceed to Part 2.

Finally, we deal with case 2. In fact, this case never happens: if it did, we would have (from the fact that Lemma \ref{lem:iterate} could be applied) that
\[
 f(u,x_k,\r_k r,\nu_k)\leq \e_I,
\]
which means that by Proposition \ref{prop:flatsame} $K\cap Q_{\r_k r/2, \nu_k}(x_k)$ is contained in a $C\e_I^{\frac{1}{n+1}} \r_k r$ neighborhood of the plane $\pi_k$ normal to $\nu_k$ and passing through $x_k$. Yet we know that $0\in Q_{\r_k r/5,\nu_k}(x_k)\cap K$ by the definition of $\s$. Together with the fact that $x_k = |x_k| \nu_k$, this implies that $|x_k|\leq C\e_I^{\frac{1}{n+1}} \r_k r$. Then when we apply Lemma \ref{lem:iterate}, we would obtain that
\begin{align*}
|x_{k+1}| &\leq |x_{k+1}-x_k| + |x_k|\\
&\leq C\r_k r ff(u,x_k,\r_k r, \nu_k) + C\e_I^{\frac{1}{n+1}} \r_k r\\
& \leq C (\e_M \r_k^\a +\e_I^{\frac{1}{n+1}}) \r_k r \\
& < \s \r_{k+1}r
\end{align*}
provided $\e_M,\e_I$ are small enough; this is a contradiction. 

We summarize the state of the iteration: we have obtained a sequence of numbers $\r_k$ with ratio uniformly between $0<\frac{\r_{k+1}}{\r_{k}}<1$, as well as vectors $\nu_k$ and points $x_k$, such that the cylinders $Q_k:=Q_{\r_k r ,\nu_k}(x_k)$ contain $0$ well inside of them, and have $ff+E+J\leq C\r_k^\a \e_M$. It is possible that this sequence terminated after finitely many steps in the manner described in case 1 (we will consider this further in Part 2), or that it is infinite. In the latter case, we note the following: say that
\[
 \frac{ff^{\frac{1}{n+1}}(u,x_k,\r_k r, \nu_k)}{f^{1/2}(u,x_k,\r_k r, \nu_k)}.
\]
is sufficiently small. This implies that $K$ is contained in two neighborhoods of the planes $\pi_{\pm,k}$ which do not intersect, and $u$ vanishes on a region between them (as in case 1). It is now elementary to check that after a finite number of further iterations, we will arrive in case 1. Hence if there is an infinite sequence as described which avoids Part 2, it must also enjoy the geometric decay of the flatness:
\begin{equation}\label{eq:badmaini2}
 f(u,x_k,\r_k r, \nu_k) \leq C ff^{\frac{2}{n+1}}(u,x_k,\r_k r, \nu_k) \leq C\r_k^\frac{2\a}{n+1}\e_M^{\frac{2}{n+1}},
\end{equation}
where this constant is universal.

\textbf{Part 2.}

Now we consider what happens if on some cylinder in the iteration, $u(x_k\pm \r_k r/2 \nu_{ k})$ becomes zero. This may happen from the start, or it might be from case 1 in part 1. Let us recall that in any event, we have
\[
 f(u,x_k,\r_k r, \nu_k) + E(u,x_k,\r_k r, \nu_k) + J(u,x_k,\r_k r, \nu_k) \leq C\r_k^\a \e_M.
\]
Indeed, we now proceed as before, simply noting two things: first, $u$ will stay $0$ on one side of the cylinder in the next phase of the iteration, and second, we may now use the variants of all of our theorems that involve just the ordinary flatness $f$. In particular, there is no possibility of case 1, as we are directly controlling the flatness, or of case 2, by the same argument as above. We arrive at the same conclusion, namely that 
\begin{equation}\label{eq:badmaini6}
 f(u,x_k,\r_k r, \nu_k) + E(u,x_k,\r_k r, \nu_k) + J(u,x_k,\r_k r, \nu_k) \leq C\r_k^\a \e_M
\end{equation}
for all larger $k$.

\textbf{Part 3.}

We estimate the tilt $1 - (\nu_k,\nu_{k+1})^2$ incurred in each step of the iteration. 
%
Except in the event that $k$ falls into case 1 in Part 1, we have directly from Lemma \ref{lem:iterate} that
\[
 1 - (\nu_k,\nu_{k+1})^2 \leq C ff(u,x_k,\r_k r, \nu_k)\leq C \r_k^\a \e_M
\]
if in Part 1,
or
\[
 1 - (\nu_k,\nu_{k+1})^2 \leq C f(u,x_k,\r_k r, \nu_k)\leq C \r_k^\a \e_M
\]
if we are in Part 2. If we are in case 1 (recall this happens at most once) then
\[
 1 - (\nu_k,\nu_{k+1})^2=1 - (\nu_k,\nu_{+,k})^2 \leq C f(u,x_k,\r_k r, \nu_k)\leq C \e_I.
\]
We may sum this to obtain that
\begin{equation}\label{eq:badmaini4}
 1 - (\nu_k,e_n)^2\leq C(\e_M+\e_I).
\end{equation}
Finally, we deduce that (using the fact that $\r_{k+1}/\r_k$ are bounded, and $0$ is well inside each $Q_k$)
\begin{equation}\label{eq:badmaini3}
 e(u,\r r) + E(u,\r r) + J(u,\r r) \leq C(\e_M+\e_I)
\end{equation}
for every $\r\leq \frac{1}{2}$, say. Using \eqref{eq:badmaini2} if Part 2 is never reached, or else \eqref{eq:badmaini5} and \eqref{eq:badmaini6} if it is, we also have
\begin{equation}\label{eq:badmaini7}
f(u,\r r)\leq C(\e_M^{\frac{2}{n+1}}+\e_I).
\end{equation}

We also note that for each fixed $x\in Q_{r/2}\cap K$, applying this argument centered at $x$ generates a unique limiting vector $\nu_x$, with $1 - (\nu_x,e_n)^2\leq C(e_M+\e_I)$, such that
\[
 \lim_{\r} f(u,x,\r,\nu_x) + E(u,x,\r,\nu_x) + e(u,x,\r,\nu_x) = 0.
\]
Carefully observe that while after some $\r_x$ the rate becomes geometric, we have not established any uniform control over this $\r_x$. 

\textbf{Part 4.} 

We show that $K\cap Q_{\eta_M r}$ is the union of the two Lipschitz graphs $\G_\pm$. Indeed, apply Theorem \ref{thm:lip} to $Q_{r/4}$, and recall the definitions of the sets $G$ and $G^{\pm}$ from there. We have just shown in Part 3 that $K\cap Q_{r/24} = G \cap Q_{r/24}$. We will now show that $\pi(\bar{G}^+\cap Q_{r/50})=D_{r/50}$; together with the corresponding fact for $\bar{G}^-$, this implies the conclusion (from the claim in the proof of Theorem \ref{thm:lip}, this would mean that $\pi(G\sm (\bar{G}^+\cup \bar{G}^-))$ is empty). It suffices to show that $\cL^{n-1}(D_{r/50}\sm \pi(G^+))=0$. 

We first show that $\pi(K)$ covers all of $D_{r/40}$. Indeed, take any $x\in D_{r/40} \sm \pi(K)$ and choose $t$ so that $D_t(x)$ is the maximal disk in $D_{r/24}\sm \pi(K)$, and $z\in K$ is a point on the boundary of the cylinder $D_t(x)\times (-r,r)$. If $\e_M$ is small enough, there must be such a $z$ (i.e. as $t$ is increased, $D_t(x)\times (-r,r)$ will intersect $K$ before $\p D_{r/24}$), for otherwise if $H$ is the ``holes set" from \eqref{eq:ldc2} of Lemma \ref{lem:lowerdensity}, we have $D_t(x) \ss H$ and $t>\frac{r}{100}$, violating that estimate. But applying \eqref{eq:badmaini3} and \eqref{eq:badmaini7} to $z$ with $\r = 2 t /r$, we still obtain a contradiction to \eqref{eq:ldc2}: on one hand, from Lemma \ref{lem:lowerdensity}, we have that $\cL^{n-1}(H) \leq C(\e_M^{\frac{2}{n+1}} + \e_I) t^{n-1}$, but on the other hand, $D_t \ss H$.

Now we show the main claim. It suffices after the preceding argument to show that $\pi(G^+)$ contains $0$ if $0\in A^+$; this follows from the density of $\pi(A^+)$ and then applying to translations of $u$. Let $Q_k$ be the nested cylinders of Parts 1 and 2, based around $0$. Let $Q^+_k$ denote the upper half of each such cylinder, meaning the set $\{(x-x_k)\cdot \nu_k >\theta\r_k r/2\}$, where $\theta$ is the lower bound on $\r_{k+1}/\r_k$; we know that from Lemma \ref{lem:intreg},
\[
 \osc_{Q_k^+} u \leq C (\r_k r \e_M)^{1/2} .
\]
Also, we claim that the $Q_k^+$ are pairwise overlapping and cover the line segment $\{0\}\times (0,r/2)$. This is obvious except for the $k$ at which case 1 occurs (if it does), by \eqref{eq:badmaini4}. If case 1 does occur, the cylinder is shifted in the $\nu_{+,k}$ direction from its original position: recall that in case 1, $K$ is segregated into two disjoint portions on $Q_k$, and there is a region where $u=0$ between them; the claim here is that $0$ is in the upper portion. As $0\in A^+$, the line segment $\{0\}\times (0,r/2)$ is disjoint from $K$, and hence $u$ is nonzero along the entire segment. If $0$ were in the lower portion of $K$, this would result in a contradiction.

We may now compute the oscillation of $u$ along this segment by summing the oscillations on the overlapping cylinders. This gives that
\[
 \osc_{\{0\}\times (0,r/2)} u \leq C \sum_k \sqrt{\r_k r \e_M}   \leq C \sqrt{ r \e_M}.
\]
Choosing $\e_M$ sufficiently small, this implies $0\in G^+$, as desired.

\textbf{Part 5.}

We finally show that $g_\pm$ are $C^{1,\a}$. Indeed, we may apply Proposition \ref{prop:liptoplane} on every cylinder $Q_k$ in the sequence obtained from Parts 1 and 2,  based at point $z$. After applying Part 3, changing coordinates from the plane $\pi_k = \{x:(x - x_k)\cdot \nu_k\}$ to $\pi$, and reducing the domain of integration, we obtain that on every disk $D_t(z')\ss D_{r/50}$
\[
 \frac{1}{t^{n-1}}\int_{D_t(z')}| \n g_+ - m_{z,t}|^2 d\cL^{n-1}\leq C t^{\a} \e_M
\]
for some vectors $m_{z,t}$ (these are $\n L_{+,t}$, where $L_{+,t}$ parametrize the optimal planes $\{x:(x-x_k)\cdot \nu_{+,k}\}$ as graphs over $\pi$). The conclusion is now immediate from applying Campanato's embedding.
\end{proof}

\section{Weakening the Hypotheses: Blow-ups}\label{sec:blowup}

We finally prove the main theorem in the generality promised. The main step remaining is to ensure that if the flatness is small enough, the energy is automatically small. We will proceed via a blow-up argument with homogeneous blow-ups, which is different from our prior compactness arguments. We note that there is some overlap in the proof with arguments given in \cite{CK}, but in an effort to keep the proof self-contained, we repeat them anyways.

\begin{theorem}\label{thm:main}Let $u\in \sQ(Q_{2r},\L,4r)$, and assume $0\in K$. Then there are universal $\e_*>0$ such that if
\[
 f(u,2r)+r +\L r^s\leq \e_*,
\]
then $K\cap Q_{r/2}$ coincides with the graphs of the Lipschitz functions $g_+,g_-$, and moreover these functions are actually $C^{1,\a}$ for some $\a>0$, with
\[
 \|g_+\|_{C^{1,\a}(D_{ r/2})} +\|g_-\|_{C^{1,\a}(D_{r/2})} \leq 1.
\]
If either of $u(0,\pm r)=0$, we may take $g_+=g_-$; if not, $u$ vanishes on $\{g_-<g_+\}\cap Q_{r/2}$.
\end{theorem}

\begin{proof}
\textbf{Part 1.}

First, we show that if the hypotheses included the statement
 \[
  E(u,r)\leq \b,
 \]
for some (small) universal $\b$, the theorem would follow. Then we will check that this is always true if $\e_*$ is small enough.

First, by choosing $\b,\e_*$ very small in terms of $\e_M$ and applying Lemma \ref{lem:jump}, we may obtain that
\[
 J(u,x,\t r) \leq 2\t^{1/2}C\leq \e_M/4
\]
for some small $\t$ and any $x\in Q_{3r/4}\cap K$, and that moreover
\[
 E(u,x,\t r) + f(u,x,\t r)\leq \e_M/2.
\]
Applying Theorem \ref{thm:badmain} gives that $Q_{\eta_M \t r}\cap K$ is the union of two $C^{1,\a}$ graphs. Repeating for all $z$, we obtain that in fact $Q_{5r/8}\cap K$ is a union of two such graphs. 

Finally, we need to check that $u=0$ in the region between the graphs. Take a component $V$ of $\{g_-<g_+\}\cap Q_{5r/8}$ which intersects $Q_{r/2}$, and suppose that $u$ is positive on $V$. If $\pi(V)$ contains $D_{r/2}$, we easily reach a contradiction to Lemma \ref{lem:upperdensity}, as we would have
\[
 \mu (Q_{r/2})\geq \w_{n-1}\1\frac{r}{2}\2^{n-1} \1 u^2(0,r)+u^2(0,-r) +2\d^2 - o(\e_*) \2.
\]
The last term accounts for the oscillation of $u$ over the Lipschitz domains $x_n\geq g_+$ and $x_n\leq g_-$; this is at most $C\sqrt{r}$ arguing as in the first lines of the proof of Lemma \ref{lem:cut}. 

Now, take a point $z'$ in $\p \pi(V)\cap D_{r/2}$ with the property that there is a small disk $D_\r(y')\ss \pi(V)$ with $z$ in the boundary of this disk and $\r< r/10$. It is clear that such a point exists, as we may start with a point $y'\in \pi(V)$ which has $d(y',\p \pi (V))< \min\{\frac{r}{2} -|y'|,\frac{r}{10}\}$, take the largest disk $D_\r(y')$ which remains in $\pi(V)$, and choose an appropriate point on the boundary. Now, we may also find a (unique) point $z\in \p V$ which has $z=(z',z_n)$, and we will have that $z\in \G^+ \cap \G^-$.

At $z'$, the continuously differentiable functions $g_\pm$ coincide, and as $g_+\geq g_-$, we must have $\n g_+(z')=\n g_-(z')$, and $|\n g_+|\leq C(\e_*+\b)$. In particular, this implies that
\[
 f(u,z,4\r,\nu)+\L (4\r)^s +\r \leq C(\e_*+\b)^2\ll \e_1(\g),
\]
where $\e_1$ is the constant of Lemma \ref{lem:upperdensity} for a $\g$ about to be determined. On the other hand, we have an elementary estimate on $\mu$:
\[
 \mu (Q_{2\r}(z))\geq \w_{n-1}\12\r\2^{n-1} \1 u^2(0,r)+u^2(0,-r)- o(\e_*)\2 +2\d^2\cL^{n-1}(\pi(V)\cap D_{2\r}(z')).
\]
By construction, the last term is at least $2\d^2\cL^{n-1}(D_\r(y'))= 2\d^2\w_{n-1}\r^{n-1}$, and so for a small $\g$ we obtain a contradiction with Lemma \ref{lem:upperdensity} applied to $Q_{4\r}(z)$.

We note that this argument required extra work entirely because of our very weak notion of quasiminimality. For the minimizers of Section \ref{sec:highreg}, this would be a trivial consequence of the uniform lower bound on the Lebesgue density of each component of $\{u>0\}$.

\textbf{Part 2.}

 In the remainder of the proof, we show that if $\e_*$ is small enough, then $E(u,r)\leq \b$. Assume this is false: then we may find a sequence $u_k\in \sQ(Q_{2r_k},\L_k,4r_k)$ with $r_k+\L_k r_k^s \searrow 0$ and
 \[
  f_\8(u_k,2r_k)=\e_k \rightarrow 0,
 \]
but $E(u_k,r_k)\geq \b$. Set $u_{+,k}= u_k(0,r_k)$, and likewise $u_{-,k}= u_k(0,-r_k)$. First, we have that along a subsequence $K_{u_k}/r_k$ converges in the local Hausdorff sense to $K_\8\ss \pi \cap Q_{2r}$ (with $0\in K_\8$), and $u_k(r_k \cdot) \rightarrow u_\8$ in $L^1$, with $u_\8$ locally constant on $Q_2 \sm K_\8$. The fact that $u_\8$ is constant away from its jump set follows because
\begin{equation}\label{eq:maini1}
 E(u_k,2 r_k) = \frac{1}{(2r_k)^{n-1}}\int_{Q_{2 r_k}} |\n u_k|^2 d\cL^n \leq C
\end{equation}
from \eqref{eq:remball}, so $\n u_\8=0$. We also have that for any open cylinder $V \cc Q_2 \sm K_\8$, for large $k$, $d(V,K_{u_k}/r_k)>c$, and so from Lemma \ref{lem:intreg},
\[
  [u_k(r_k \cdot)]_{C^{\frac{1}{2}}(V)} \leq  C(V) r_k^{1/2},
\]
and so $u_\8$ is continuous (and locally constant) away from $Q_{2} \cap K_\8$. Let $u_\pm$ be the two values of $u_\8$ on $Q_2\sm \pi$ (with $u_+$ the value on $\{x_n>0\}$ and $u_-$ the value on $\{x_n<0\}$; note that they may be the same).

There are now two possibilities: either $K_\8 = \pi$ and disconnects $Q_2$, or else $Q_2 \sm K_\8$ is connected. In the latter case, note that $|u_{+,k}-u_{-,k}|\rightarrow 0$. We begin with the case that $Q_2\cap K_\8$ is disconnected. Define
\[
 v_k(z) = \begin{cases}
           \frac{u_k(r_k z)- u_{+,k}}{\sqrt{r_k}} & z\in Q_{2}\cap \{z_n \in (2\e_k,2)\}\\
           \frac{u_k(r_k z)- u_{-,k}}{\sqrt{r_k}} & z\in Q_{2}\cap \{z_n \in (-2,-2\e_k)\}\\
           0 & z\in Q_{2}\cap \{z_n \in (-2\e_k,2\e_k)\}
          \end{cases}
\]
Then we have that
\[
 \int_{Q_2}|\n v_k|^2 d\cL^n\leq C
\]
from \eqref{eq:maini1}, and that
\[
 \| v_k\|_{C^{1/2}(Q_2 \cap \{z_n \in (t,2))\})}+\| v_k\|_{C^{1/2}(Q_2 \cap \{z_n \in (-2,-t))\})} \leq C(t).
\]
It follows that along a subsequence, $v_k\rightarrow v$ weakly in $H^1_{\text{loc}}(Q_2\sm \pi)$ and locally uniformly on $Q_2 \sm \pi$. If $Q_2 \sm K_\8$ is connected, then we proceed similarly, but with
\[
 v_k(z) = \begin{cases}
           \frac{u_k(r_k z)- u_{+,k}}{\sqrt{r_k}} & z\in Q_{2}\cap (  \{d(z,K_\8/r_k)> 2 d(K_{u_k}/r_k, K_\8)\}\cup \{|x_n| > 2\e_k\})\\
           0 &  z\in Q_{2}\cap \{d(z,K_\8/r_k)\leq 2 d(K_{u_k}/r_k, K_\8)\} \cap \{|x_n|\leq 2\e_k\}
          \end{cases}
\]
It may be checked that the set $\{d(z,K_\8/r_k)> 2 d(K_{u_k}/r_k, K_\8)\}\cup \{|x_n| > 2\e_k\}$ is connected and devoid of $K_\{u_k\}$ for large $k$ (using that $K_\8 \ss \pi$, that $K_{u_k}\r_k$ lies in a $2\e_k$ neighborhood of $\pi$, and that $ d(K_{u_k}/r_k, K_\8)$ goes to $0$). We then obtain the conclusion that $v_k\rightarrow v$ weakly in $H^1_{\text{loc}}(Q_2\sm K_\8)$ and locally uniformly on $Q_2 \sm K_\8$.

Next, note that as
\[
  \int_{Q_t(z)} |\n v|^2 d\cL^n\leq \liminf_k \int_{Q_t(z)} |\n v_k|^2 d\cL^n \leq Ct^{n-1}
\]
for any cylinder $Q_t(z)\ss Q_2$. we may apply the Poincar\'e inequality and Campanato's criterion on $Q_2 \cap \{x_n> 0(x_n<0)\}$ to obtain that $v$ is uniformly H\"older-$1/2$ on each half-cylinder, and so in particular is bounded (a similar argument gives that $v_k$ is uniformly H\"older-$1/2$ on $\{|x_n|\geq 2\e_k\}$). Applying \cite[Proposition 4.4]{AFP} gives that $v$ is in $SBV$. Then applying \cite[Proposition 7.9]{AFP}, we have that
\begin{equation}\label{eq:main}
 \limsup_{\r\searrow 0}\frac{1}{\r^{1-n}}\int_{Q_\r(z)}|\n v|^2 d\cL^n =0
\end{equation}
for $\cH^{n-1}$-a.e. $z\in Q_2$. Finally, Lemma \ref{lem:harmacon} gives that $v$ is harmonic on $Q_2\sm K_\8$, and $v_k\rightarrow v$ strongly in $H^1_{\text{loc}}(Q_2\sm K_\8)$.

\textbf{Part 3.}

We now establish a lower semicontinuity property for the measures $\m_k:=r_k^{1-n}\m_{u_k}(r_k\cdot)$: we claim that for any $x\in \pi$ and $Q_t(x)\cc Q_2$,
\begin{equation}\label{eq:maini3}
 2u_+^2\cH^{n-1}(Q_t(x)\cap K_\8)\leq \liminf \m_k(\bar{Q}_t(x))
\end{equation}
if $Q_2\sm K_\8$ is connected, and
\begin{equation}\label{eq:maini4}
 (u_+^2+u_-^2)\cH^{n-1}(Q_t(x)\cap \pi)\leq \liminf \m_k(\bar{Q}_t(x))
\end{equation}
otherwise. Fix a subsequence along which the liminfs above are attained. Let   
\[
 S_{a,\s} = \3y\in Q_t(x)\cap K_\8: \limsup_{k\rightarrow \8}  \frac{1}{a^{n-1}}\int_{Q_{a}(y)\cap \{|z_n|\geq a^2\}}|\n v_k|^2 d\cL^n \leq \e_2'(\s) \4,
\]
where $\e_2'(\s)$ is the constant from Lemma \ref{lem:lowerdensity2}. Now take $y\in S=\cap_{\a>0}\cup_{a\leq\a ,1-\s \leq \a} S_{a,\s}$, the set of points belonging to $S_{a,\s}$ for arbitrarily small $a$ and $1-\s$. For any $a,\s$ for which $y\in S_{a,\s}$, we may apply Lemma \ref{lem:lowerdensity2} to see that
\begin{equation}\label{eq:maini2}
 \s (u_+^2+u_-^2)(a/2)^{n-1}\w_{n-1} \leq \liminf \m_k(Q_{a/2}(y)).
\end{equation}
Fix $\a>0$; then disks $D_a(y)$ with $a<\a$ satisfying \eqref{eq:maini2} for $\s$ with $1-\s <\a$ form a fine cover of $S$, so using the Vitali covering lemma we may find a disjoint cover of $S$ by disks $\bar{D}_{a(y)}(y)$ with 
\[
(1-\a) (u_+^2+u_-^2)(a(y))^{n-1}\w_{n-1}\leq \s(y) (u_+^2+u_-^2)(a(y))^{n-1}\w_{n-1} \leq \liminf \m_k(Q_{a(y)}(y))
\]
such that $\cL^{n-1}(S\sm \cup_y D_{a(y)}(y))=0$. Summing, we obtain that
\[
 (1-\a) (u_+^2+u_-^2)\cL^{n-1}(S) \leq \liminf \m_k(Q_{t+\a})(y)),
\]
where we used that the disks have radii bounded by $\a$. Now send $\a\rightarrow 0$ to get
\[
 (u_+^2+u_-^2)\cL^{n-1}(S) \leq \liminf \m_k(\bar{Q}_{t})(y)).
\]
It remains to show that $\cL^{n-1}$-a.e. point $y$ in $K_\8$ is in $S_{a,\s}$ for arbitrarily small $a$ and $\s$ close to $1$. This, however, is true for any $y$ at which \eqref{eq:main} holds, by applying the strong convergence property of Lemma \ref{lem:harmacon} to $v_k$.

In principle, it is now possible to establish that $v$ is a minimizer of a certain functional (this is explored in \cite{CK}). For ease of exposition, however, we content ourselves with two related properties of $v$. First, we check that
\[
 \lim_k \frac{1}{r_k^{n-1}}\int_{Q_{t r_k}(z r_k)} |\n u_k|^2 d\cL^n = \int_{Q_t(z)} |\n v|^2 d\cL^n
\]
for any $z\in \pi$ and $Q_t (z)\cc Q_2$. To see this, we construct competitors for $u_k$ based on $v$. Indeed, fix $\eta$ a smooth cutoff function equal to $1$ on $Q_t(z)$ and vanishing outside $Q_{\g t}(z)$, with $\g>1$ small, and let $\eta_k (y) =\eta (y/r_k)$. Pick a bounded function $q\in H^1(Q_2\sm K_\8)$ and coinciding with $v$ outside of $Q_{t}(z)$, and set
\[
 \tilde{v}_k(y) =\begin{cases}
          \sqrt{r_k}[q(y/r_k)+ u_{+,k}]  & \{y_n >0\}\\
          \sqrt{r_k}[q(y/r_k)+ u_{-,k}]  & \{y_n <0\},
         \end{cases}
\]
and then 
\[
 w_k(y) = \begin{cases}
	0 & y\in r_k T_\g\\
        \tilde{v}_k \eta_k + u_k (1-\eta_k) & y\notin r_k T_\g,
       \end{cases}
\]
where 
\[
T_\g := \{(x',x_n)\in Q_2: t\leq |x' - z'|\leq \g t, |x_n|\leq \g-1\} = (\bar{Q}_{t\g}(z)\sm Q_t(z))\cap \{|x_n \leq \g-1\}.
\]

The energy of $w_k$ may be estimated as
\begin{align*}
 \int_{Q_{2 r_k}}&|\n w_k|^2 d\cL^n \leq \int_{Q_{2r_k}\sm r_k T_\g} \eta_k^2 |\n \tilde{v}_k|^2 + (1-\eta_k)^2 |\n u_k|^2 d\cL^n \\
 &+ \int_{Q_{2r_k}\sm r_k T_\g} (\tilde{v_k} -u_k)^2 |\n \eta_k|^2 d\cL^n\\
 &+ \int_{Q_{2r_k} \sm r_k T_\g}2 \eta_k (1-\eta_k) \n \tilde{v}_k\cdot \n u_k + 2 (\tilde{v}_k -u_k) \n \eta_k \cdot [\eta_k \n \tilde{v}_k +(1-\eta_k) \n u_k]  d\cL^n. \\
\end{align*}
We change variables in the right-hand side, to see that
\begin{align}
 \frac{1}{r_k^{n-1}}\int_{Q_{2 r_k}}|\n w_k|^2 d\cL^n \leq & \frac{1}{r_k^{n-1}} \int_{Q_{2 r_k}} (1-\eta)^2 |\n u_k|^2 d\cL^n \label{eq:mains1}\\
 &+ \int_{Q_{2}} \eta^2 |\n q|^2 d\cL^n \label{eq:mains2}\\
 &+ \int_{Q_{2}\sm T_\g} (q -v_k)^2 |\n \eta|^2 d\cL^n\label{eq:mains3}\\
 &+ \int_{Q_{2} \sm T_\g}2 \eta (1-\eta) \n q \cdot \n v_k d\cL^n \label{eq:mains4}\\
 & +\int_{Q_2\sm T_\g} 2 (q -v_k) \n \eta \cdot [\eta \n q +(1-\eta) \n v_k]  d\cL^n. \label{eq:mains5}
\end{align}
We used that \eqref{eq:mains3}, \eqref{eq:mains4}, and \eqref{eq:mains5} are supported on $\supp(\n \eta)\sm T_\g\ss \{|x_n|\geq \g-1\}$ in order to replace $\frac{u_k(r_k \cdot) - u_{\pm,k}}{\sqrt{r_k}}$ by $v_k$ and $\sqrt{r_k} \n u_k (r_k \cdot)$ by $\n v_k$. Note that we may use the convergence of $v_k\rightarrow v$ in $H^1$ of this region to see that the limit as $k\rightarrow \8$ of \eqref{eq:mains3} and \eqref{eq:mains5} vanishes, while the \eqref{eq:mains4} passes to a (nonzero) limit. 

Now we consider the surface term associated to $w_k$. There are three contributions:
\begin{align}
 \frac{1}{r_k^{1-n}}\int_{Q_{2r_k}\cap J_{w_k}} &\overline{w}_k^2 +\underline{w}_k^2 d\cH^{n-1} \leq \m_k (Q_{2}\sm \bar{Q}_{\g t}(z)) \label{eq:mains6}\\
 &+ C (\g-1) \label{eq:mains7}\\
 &+ \int_{Q_t(z)\cap K_\8} (\sqrt{r_k}q + u_{+,k})^2 + (\sqrt{r_k}q + u_{-,k})^2 d\cH^{n-1} \label{eq:mains8}.
\end{align}
The second term is an estimate of the perimeter of $T_\g/r_k$, and uses that $w_k$ is uniformly bounded. Note that we also have $\sqrt{r_k}q\rightarrow 0$ uniformly.  We use $w_k$ as a competitor for $u_k$, to obtain that
\begin{align*}
\frac{1}{r_{k}^{n-1}}&\int_{Q_{2r_k}} |\n u_k|^2 d\cL^n + \m_k(Q_{2}) \\
&\leq \frac{1}{r_k^{n-1}} \1 \int_{Q_{2 r_k}} |\n w_k|^2 d\cL^n + \int_{Q_{2r_k}\cap J_{w_k}}  \overline{w}_k^2 +\underline{w}_k^2 d\cH^{n-1} \2 + \L_k r^s_k\\
&\leq \eqref{eq:mains1} + \eqref{eq:mains2} + \eqref{eq:mains3} + \eqref{eq:mains4} +\eqref{eq:mains5} \\
&\quad +\eqref{eq:mains6} +\eqref{eq:mains7} + \eqref{eq:mains8} + \L_k r_k^s.
\end{align*}
Subtracting \eqref{eq:mains1} and \eqref{eq:mains6} from both sides and taking limsup gives
\begin{align*}
 &\limsup  \1\frac{1}{r_k^{n-1}}\int_{Q_{2 t r_k}}  |\n u_k  |^2 d\cL^n + \m_k (\bar{Q}_{ t r_k}(r_k z)) \2 \\
&\leq \limsup  \1\frac{1}{r_k^{n-1}}\int_{Q_{2 r_k}}[1-(1-\eta)^2]  |\n u_k  |^2 d\cL^n + \m_k (\bar{Q}_{\g t r_k}(r_k z)) \2 \\
 &\leq \int_{Q_2} \eta^2 |\n q|^2 +\eta (1-\eta)\n v\cdot \n q d\cL^n + C (\g-1)r_k^{n-1}  + (u_{+}^2+u_{-}^2)\cL^{n-1}(K_\8 \cap Q_t(z)).
\end{align*}
We now send $\g\rightarrow 1$:
\begin{align*}
 \limsup & \1\frac{1}{r_k^{n-1}}\int_{Q_{2 tr_k}(r_k z)}|\n u_k  |^2 d\cL^n + \m_k (\bar{Q}_{t r_k}(r_k z)) \2\\
 &\leq \int_{Q_t(z)} |\n q|^2 d\cL^n + (u_{+}^2+u_{-}^2)\cL^{n-1}(K_\8 \cap Q_t(z)).
\end{align*}
Together with the lower semicontinuity properties \eqref{eq:maini3} and \eqref{eq:maini4}, setting $q=v$ gives the convergence
\[
 \lim  \frac{1}{r_k^{n-1}}\int_{Q_{tr_k}(r_k z)}|\n u_k  |^2 d\cL^n = \int_{Q_t(z)} |\n v|^2 d\cL^n,
\]
as claimed, while taking other $q$ we see that
\[
 \int_{Q_t(z)} |\n v|^2 d\cL^n \leq  \int_{Q_t(z)} |\n q|^2 d\cL^n
\]
for all $q\in H^{1}(Q_2\sm K_\8)$ coinciding with $u$ on the boundary of $Q_t(z)$. In particular, if $K_\8\cap Q_t(z) = \pi\cap Q_t(z)$, we have that $u$ satisfies a Neumann condition on either side of $\pi$ on this cylinder.

\textbf{Part 4.}

We now choose a point $y\in K_\8\cap Q_1$ for which \eqref{eq:main} holds. Then for a small $\r>0$,
\[
 \int_{Q_{2\r}(y)} |\n v|^2 d\cL^n \leq \r^{n-1} \b/2,
\]
with $\b$ as chosen in Part 1. Then from Part 3,
\[
 E(u_k,r_k y, 2\r r_k)\leq  \b/2,
\]
and so for sufficiently large $k$ we may apply Part 1 to this cylinder. This means that $K_{u_k}\cap Q_{\r r_k/2}(r_k y)$ is given by a pair of $C^{1,\a}$ graphs over $\pi$, with uniform constant. We fix a vector field $T=\phi(x')e_n\in C^1_c(Q_{\r/4}(y):\R^n)$ with $\phi$ nonnegative, and apply Lemma \ref{lem:EL} to $T(r_k\cdot)/r_k$, to obtain that
\[
 \lim \1\int_{\G^+_k/r_k} u_k(r_k x)^2 \dvg\nolimits^{\nu_x} T d\cH^{n-1} + \int_{\{x_n\geq 0\}} |\n v_k|^2 \dvg T - 2\n v_k \cdot \n T\n v_kd\cL^n\2\geq0.
\]
From the uniform bound on the $C^1$ norms of $\G^+_k/r_k$ (and Remark \ref{rem:lipest} to explain the convergence of $u_k$), the first term converges to
\[
 \int_{\pi} u_{+,\8}^2 \dvg\nolimits^{e_n} T d\cL^{n-1}.
\]
For our choice of vector field $T=\phi(x')e_n$, we have that $\dvg\nolimits^{e_n} T=0$, and hence this term vanishes. Thus we are left with
\[
 \int_{\{x_n\geq 0\}} |\n v|^2 \dvg T - 2\n v \cdot \n T\n v d\cL^n \geq 0.
\]
Recalling Rellich's identity for harmonic functions
\[
 \dvg ( T |\n v|^2 -2 \n v T \n v) = |\n v|^2 \dvg T - 2\n v \cdot \n T\n v
\]
and applying the divergence theorem, we have that
\[
 \int_{\pi} |\n v|^2 T\cdot e_n - 2 (\n v \cdot T)( \n v \cdot e_n )d\cL^{n-1}\leq 0 
\]
where we abuse notation and use $\n v$ to refer to the trace from the $\{x_n>0\}$ side. As we saw in Part 3, $v$ satisfies $\n v \cdot e_n =0$ along $\pi \cap Q_{t/4}(z)$, and so we have
\[
 \int_{\pi} |\n v|^2 \phi(x)d\cL^{n-1}\leq 0.
\]
This implies that $\n v$ vanishes along $\pi \cap Q_{t/4}(z)$, which from unique continuation means that $v$ is constant. We may apply the same argument to the other side, to obtain that $\n v=0$ on $Q_1$.

Finally, from Part 3 we have that
\[
 \lim \frac{1}{r_k^{n-1}}\int_{Q_{r_k}}|\n u_k|^2 d\cL^n = \int_{Q_1}|\n v|^2 d\cL^n =0,
\]
which is in clear contradiction of 
\[
 \frac{1}{r_k^{n-1}}\int_{Q_{r_k}}|\n u_k|^2 d\cL^n\geq \b>0.
\]
\end{proof}

\section{Strengthening the Conclusion: Higher Regularity}\label{sec:highreg}

Under the assumption that $u$ is a quasiminimizer, the $C^{1,\a}$ regularity of the graphs is likely optimal, up to finding the best $\a$ in terms of $s$. We are more interested, however, in the minimizers studied in \cite{CK}, which minimize the following functional locally:
\[
 F_1(u;B_r)= \int_{B_r \sm K}|\n u|^2 d\cL^n + \int_{B_r \cap K_u} \uu^2 + \ud^2 d\cH^{n-1} + \bar{C}\cL^n (B_r \cap \{u>0\}).
\]
Here we assume that we are in the $\d$-regular configuration. It will be helpful to know that for those flat points that are on the boundary of $\{u=0\}$, we have that $K$ is smooth.

\begin{theorem}\label{thm:higherreg}Let $u$ be as in Theorem \ref{thm:main}, and also be a local minimizer of $F_1$. If $u(x,-r)=0$, then $g$, the graph parametrizing $K$, is $C^\8$.
\end{theorem}

\begin{proof}
 We may apply Theorem \ref{thm:main} to learn that $K\cap Q_{r/2}$ is given by a  $C^{1,\a}$ graph, and that $u=0$ in $\{x_n\leq g(x')\}$. Let $\nu_x$ be the upward-oriented unit normal to $\G$ (i.e. pointing into $\{u>0\}$). As $u$ minimizes $F_1$, it is harmonic on $Q_{r/2}$ and satisfies a Robin condition along $\G$ ($\p_{\nu_x} u = u$ on $\G$). It follows that $u\in C^{1,\a}(Q_{1/4})$, with universal constant. 
 
From a computation completely analogous to that in Lemma \ref{lem:EL} (as there is only one graph, we may perform perturbations with $t<0$ as well), we have that for any vector field $T\in C_c^1 (Q_{r/4};\R^n)$,
\[
 \int_{\G} u^2 \dvg\nolimits^{\nu_x} T - \bar{C}T\cdot \nu_x d\cH^{n-1} + \int |\n u|^2 \dvg T - 2\n u \cdot \n T \n u d\cL^n = 0.
\]
We use Rellich's identity
\[
 \dvg ( T |\n u|^2 -2 \n u \cdot T \n u) = |\n u|^2 \dvg T - 2\n u \cdot \n T\n u, 
\]
the divergence theorem, and the Robin condition to rewrite this as
\[
 \int_{\G} u^2 \dvg\nolimits^{\nu_x} T - \bar{C}T\cdot \nu_x - |\n u|^2 T\cdot \nu_x + 2 u \n u \cdot T d\cH^{n-1}  = 0.
\]
We may integrate the first term by parts to obtain that
\[
 \int_{\G} u^2 H_\G  T\cdot \nu_x - 2 u (\n u\cdot T - \p_{\nu_x} u T\cdot \nu_x) - \bar{C}T\cdot \nu_x - |\n u|^2 T\cdot \nu_x + 2 u \n u \cdot T d\cH^{n-1}  = 0,
\]
where $H_\G$ is the distributional mean curvature of $\G$ as oriented by $\nu_x$. Combining terms and using the Robin condition,
\[
 \int_{\G} u^2 H_\G  T\cdot \nu_x -\bar{C} T\cdot \nu_x - |\n u|^2 T\cdot \nu_x + 2 u^2 T\cdot \nu_x d\cH^{n-1}  = 0
\]
This means that $\G$ is a distributional solution of the mean curvature equation
\[
 u^2 H_\G -\bar{C} - |\n u|^2 +2 u^2=0
\]
Now it follows from standard Schauder theory that a $C^{1,\a}$ solution to this equation is $C^{\b+1}$ if $u\in C^\b$, where $\b\notin \N$. On the other hand, a harmonic function solving the Robin problem on a $C^\b$ domain is $C^\b$ itself, so a bootstrap argument gives that $\G,u\in C^\8$.
\end{proof}

\section*{Acknowledgments}

The author is grateful for all of the help and encouragement he was given Luis Caffarelli, without whom this project would not have been possible. He was supported by NSF grant DMS-1065926 and the NSF MSPRF fellowship DMS-1502852.

\bibliographystyle{plain}
\bibliography{Paper2}

\begin{thebibliography}{10}

\bibitem{A}
William~K. Allard.
\newblock On the first variation of a varifold.
\newblock {\em Ann. of Math. (2)}, 95:417--491, 1972.

\bibitem{AC}
H.~W. Alt and L.~A. Caffarelli.
\newblock Existence and regularity for a minimum problem with free boundary.
\newblock {\em J. Reine Angew. Math.}, 325:105--144, 1981.

\bibitem{AFPpt2}
Luigi Ambrosio, Nicola Fusco, and Diego Pallara.
\newblock Partial regularity of free discontinuity sets. {II}.
\newblock {\em Ann. Scuola Norm. Sup. Pisa Cl. Sci. (4)}, 24(1):39--62, 1997.

\bibitem{AFP}
Luigi Ambrosio, Nicola Fusco, and Diego Pallara.
\newblock {\em Functions of bounded variation and free discontinuity problems}.
\newblock Oxford Mathematical Monographs. The Clarendon Press, Oxford
  University Press, New York, 2000.

\bibitem{APpt1}
Luigi Ambrosio and Diego Pallara.
\newblock Partial regularity of free discontinuity sets. {I}.
\newblock {\em Ann. Scuola Norm. Sup. Pisa Cl. Sci. (4)}, 24(1):1--38, 1997.

\bibitem{Bon}
Alexis Bonnet.
\newblock On the regularity of the edge set of {M}umford-{S}hah minimizers.
\newblock In {\em Variational methods for discontinuous structures ({C}omo,
  1994)}, volume~25 of {\em Progr. Nonlinear Differential Equations Appl.},
  pages 93--103. Birkh\"auser, Basel, 1996.

\bibitem{CK}
L.~Caffarelli and D.~Kriventsov.
\newblock A free boundary problem related to thermal insulation.
\newblock 2015.
\newblock preprint.

\bibitem{Camp}
S.~Campanato.
\newblock Propriet\`a di h\"olderianit\`a di alcune classi di funzioni.
\newblock {\em Ann. Scuola Norm. Sup. Pisa (3)}, 17:175--188, 1963.

\bibitem{David}
Guy David.
\newblock {$C^1$}-arcs for minimizers of the {M}umford-{S}hah functional.
\newblock {\em SIAM J. Appl. Math.}, 56(3):783--888, 1996.

\bibitem{D}
Guy David.
\newblock {\em Singular sets of minimizers for the {M}umford-{S}hah
  functional}, volume 233 of {\em Progress in Mathematics}.
\newblock Birkh\"auser Verlag, Basel, 2005.

\bibitem{DGms}
Ennio De~Giorgi.
\newblock {\em Frontiere orientate di misura minima}.
\newblock Seminario di Matematica della Scuola Normale Superiore di Pisa,
  1960-61. Editrice Tecnico Scientifica, Pisa, 1961.

\bibitem{Federer}
Herbert Federer.
\newblock {\em Geometric measure theory}.
\newblock Die Grundlehren der mathematischen Wissenschaften, Band 153.
  Springer-Verlag New York Inc., New York, 1969.

\end{thebibliography}

\end{document}